\documentclass[12pt]{amsart}
\usepackage{hyperref}
\hypersetup{
    colorlinks=true, 
     citecolor=teal, 
     linkcolor=blue, 
    }
\newcommand{\rstr}{\:\mbox{\rule{0.1ex}{1.2ex}\rule{1.1ex}{0.1ex}}\:}
\newcommand{\slice}[3]{\langle#1,#2,#3\rangle}
\newcommand{\mass}[2][]{{\mathbf M_{#1}}(#2)}
\usepackage[utf8]{inputenc}
\usepackage[english]{babel}
\usepackage[T1]{fontenc}
\usepackage{amsmath}
\usepackage{amsfonts}
\usepackage{amssymb}
\usepackage{subfigure}

\usepackage{mathtools}

\usepackage[toc,page]{appendix} 
\usepackage{amsthm} 
\usepackage[dvipsnames]{xcolor} 
\usepackage{graphicx} 
\usepackage{caption}
\usepackage{tikz-cd}
\usepackage{IEEEtrantools} 
\usepackage{bm} 
\usepackage{graphicx}
\usepackage[percent]{overpic}
\usepackage{textcomp} 

\usepackage[utf8]{inputenc}
\usepackage{amsmath}
\usepackage{amsfonts}
\usepackage{amssymb}
\usepackage{tikz}
\usetikzlibrary{decorations.markings}
\usetikzlibrary{patterns}

\usepackage{pgfplots}
\usetikzlibrary {math}
\usepgfplotslibrary{colormaps}
\usepgfplotslibrary{fillbetween}

\usepackage[utf8]{inputenc}
\usepackage{amsmath}
\usepackage{amsfonts}
\usepackage{amssymb}
\usepackage{tikz}
\usetikzlibrary{decorations.markings}
\usepackage{tikz}
\usepackage{graphicx}
\usepackage[utf8]{inputenc}
\usepackage{amsmath}
\usepackage{amsfonts}
\usepackage{amssymb}
\usepackage{tikz}
\usetikzlibrary{decorations.markings}
\usepackage{pgfplots}
\usetikzlibrary{patterns}
\makeatletter
\pgfdeclarepatternformonly[\LineSpace]{my north east lines}{\pgfqpoint{-1pt}{-1pt}}{\pgfqpoint{\LineSpace}{\LineSpace}}{\pgfqpoint{\LineSpace}{\LineSpace}}%
{
    \pgfsetcolor{\tikz@pattern@color}
    \pgfsetlinewidth{0.4pt}
    \pgfpathmoveto{\pgfqpoint{0pt}{0pt}}
    \pgfpathlineto{\pgfqpoint{\LineSpace + 0.1pt}{\LineSpace + 0.1pt}}
    \pgfusepath{stroke}
}

\pgfdeclarepatternformonly[\LineSpace]{my north west lines}{\pgfqpoint{-1pt}{-1pt}}{\pgfqpoint{\LineSpace}{\LineSpace}}{\pgfqpoint{\LineSpace}{\LineSpace}}%
{
    \pgfsetcolor{\tikz@pattern@color}
    \pgfsetlinewidth{0.4pt}
    \pgfpathmoveto{\pgfqpoint{0pt}{\LineSpace}}
    \pgfpathlineto{\pgfqpoint{\LineSpace + 0.1pt}{-0.1pt}}
    \pgfusepath{stroke}
}
\makeatother

\newdimen\LineSpace
\tikzset{
    line space/.code={\LineSpace=#1},
    line space=10pt
}

\usepackage{tkz-euclide}

\numberwithin{equation}{section}

\newtheorem{thm}{Theorem}[section]
\newtheorem{prop}[thm]{Proposition}
\newtheorem{cor}[thm]{Corollary}
\newtheorem{lem}[thm]{Lemma}

\theoremstyle{definition}
\newtheorem{defn}[thm]{Definition}
\newtheorem{claim}[thm]{Claim}
\newtheorem{question}[thm]{Question}

\theoremstyle{remark}
\newtheorem{rem}[thm]{Remark}
\theoremstyle{example}
\newtheorem{ex}[thm]{Example}

\newcommand{\SL}{\mathrm{SL}}

\newcommand{\RR}{\mathbb{R}}

\newcommand{\goodgap}{%
\hspace{\subfigtopskip}%
\hspace{\subfigbottomskip}}

\begin{document}
\title{Coarse embeddings of symmetric spaces and Euclidean buildings}
\author{Oussama Bensaid}

\maketitle
\begin{abstract}
Introduced by Gromov in the 80's, coarse embeddings are a generalization of quasi-isometric embeddings when the control functions are not necessarily affine. In this paper, we will be particularly interested in coarse embeddings between symmetric spaces and Euclidean buildings. The quasi-isometric case is very well understood thanks to the rigidity results for symmetric spaces and buildings of higher rank by Anderson--Schroeder, Kleiner, Kleiner--Leeb, Eskin--Farb and Fisher--Whyte. In particular, it is well known that the rank of these spaces is  monotonous under quasi-isometric embeddings. This is no longer the case for coarse embeddings as shown by horospherical embeddings. However, we show that in the absence of a Euclidean factor in the domain, the rank is monotonous under coarse embeddings. This answers a question by David Fisher and Kevin Whyte. This still holds when we replace the target space by a proper cocompact CAT(0) space or by a mapping class group. Between symmetric spaces and Euclidean buildings, we can also relax the condition on the domain by allowing it to contain a Euclidean factor of dimension $1$, answering a question by Gromov.
\end{abstract}
\maketitle
\tableofcontents

\section{Introduction}
\noindent The purpose of the present paper is to provide obstructions to the existence of coarse embeddings from symmetric spaces and Euclidean buildings into CAT(0) spaces and mapping class groups. In order to put our main results in perspective, we start this introduction with a brief survey on the notion of coarse embedding and its applications in geometric group theory.

\subsection{Coarse embeddings in geometric group theory}

\noindent We recall that a map $f :(X,d_X)\to (Y,d_Y)$ is a coarse embedding if there exist functions $\rho_{\pm}:[0,\infty)\to [0,\infty)$ such that $\rho_-(r)\to\infty$ as $r\to\infty$ and for all $x,y\in X$
$$\rho_-(d_X(x,y)) \leq d_Y(f(x),f(y)) \leq \rho_+(d_X(x,y)).$$
\\
\\Coarse embeddings have been introduced
under the name of “placements” or “placings” in \cite{gromov33rigid}. They also appear in the litterature as uniform embeddings in \cite{gromov1993asymptotic},\cite{shalom2004harmonic},\cite{fisher2018quasi}, effectively proper Lipschitz maps in \cite{block1992aperiodic} or uniformly proper embeddings \cite{mosher2003quasi}.
\\
\\Coarse embeddings from groups or infinite graphs into Hilbert spaces (or more general Banach spaces) have received a lot of attention due to their connection with the Baum--Connes and the Novikov conjecture (see \cite{yu00coarse}, \cite{skandalis2002coarse}, \cite{kasparov2006coarse}). However, they have originally been much less studied than quasi-isometric embeddings among finitely generated groups, probably due to their much greater flexibility. As an example, a subgroup inclusion between finitely generated groups is always a coarse embedding, while it is a quasi-isometric embedding only when the subgroup in question is undistorted.  Very few geometric invariants are known to be monotonous under coarse embeddings; the only examples (which takes infinitely many different values) being the volume growth, the asymptotic dimension \cite{gromov1993asymptotic}, the separation profile \cite{benjamini2012separation} and the Poincaré profiles \cite{hume2020poincare}.
\\
\\The volume growth allows one to say for example that a metric space with exponential growth cannot be coarsely embedded into a space with polynomial growth, but it does not distinguish between spaces of exponential growth.  More recently the F\o lner function has been shown to be monotonous under coarse embeddings between amenable groups \cite{delabie2020quantitative}. This provides a much more refined invariant to distinguish among solvable groups.
\\
\\The asymptotic dimension was introduced by Gromov in \cite{gromov33rigid} and \cite{gromov1993asymptotic} and is a large-scale analogue of Lebesgue covering dimension. For symmetric spaces of nonpositive curvature and Euclidean buildings, it simply coincides with the dimension. For instance it allows one to rule out coarse embeddings from $\mathbb{H}^n$ to $\mathbb{H}^p$ for $p<n$, but it does not prevent for instance the existence of a coarse embedding of the complex hyperbolic plane to $\mathbb{H}^4$. 
\\
\\
The separation profile is a powerful monotone coarse invariant introduced by Benjamini, Schramm and Timár \cite{benjamini2012separation}.
The separation profile of an infinite, bounded degree graph at $n \in \mathbb{N}$ is the supremum
over all subgraphs of size $\leq n$, of the number of vertices needed to be
removed from the subgraph, in order to cut it into connected pieces of
size at most $n/2$. The separation profile of a metric space with bounded geometry can be defined as the separation profile of any graph that is quasi-isometric to it.  
For rank 1 symmetric spaces, it has been shown to detect the conformal dimension of the boundary \cite{hume2020poincare}, providing for instance an obstruction to the existence of a coarse embedding from the complex hyperbolic plane to  $\mathbb{H}^4$. 
Later Hume--Mackay--Tessera \cite{hume2020poincare} introduced the $L^p$-Poincaré profiles, generalizing the separation profile. These invariants are very efficient in distinguishing between spaces of the form $X \times \mathbb{R}^n$ where $X$ is a hyperbolic space. 
However, for higher rank symmetric spaces all $L^p$-Poincaré profiles are $\simeq n/\log n$: this provides an obstruction to coarse embeddings from higher rank to $X \times \mathbb{R}^n$, where $X$ is rank one, but does not provide any obstruction among  higher rank symmetric spaces. 
\\
\\Hume--Sisto gave in \cite{hume2017groups} an obstruction for a group to admit a coarse embedding into a real hyperbolic space: admitting exponentially many fat bigons.
 It is shown for instance that direct products of two infinite groups one of which has exponential growth, solvable
groups that are not virtually nilpotent, and uniform higher-rank lattices do not coarsely embed into a hyperbolic space.
\\
\\Using a combination of results among the ones mentioned above and a result by Le Coz and Gournay \cite{coz2019separation}, Tessera recently proved the following statement  \cite{tessera2020coarse}:  an amenable group admits a coarse embedding into a hyperbolic group if and only if it is virtually nilpotent. We conclude this survey by mentioning an interesting application of the previous result in pseudo-riemmanian geometry. The story starts with an early observation of Gromov \cite[Section 4.1]{gromov33rigid} according to which the isometry group of a compact $(n + 1)$-dimensional Lorentzian manifold admits a coarse embedding into the real hyperbolic space $\mathbb{H}^n$. Frances \cite{frances2021isometry} exploited this point of view and the above result to prove a Tits alternative for discrete subgroups of the isometry group of compact Lorentz manifolds. Namely, a discrete, finitely generated, subgroup of the isometry group of a compact Lorentz manifold of dimension $n\geq2$ either contains a free
subgroup in two generators, in which case it is virtually isomorphic to a discrete
subgroup of $\textup{PO}(1, d)$, or is virtually nilpotent of growth degree $\leq n-1$.
\\
\\ We deduce from this short survey that the only known obstruction for coarse embeddings among higher rank symmetric spaces is given by their dimension. Our goal is to address this problem by showing that, under certain conditions, the rank is also a coarse monotonous invariant. 
\\
\subsection{Main results}\label{sec:mainresults}

\

\noindent
 Let $S$ be a product of symmetric spaces of non-compact type and $B$ a product of thick Euclidean buildings with cocompact affine Weyl group, with bounded geometry and no Euclidean factor. Let us call the spaces of the form $X = \mathbb{R}^n \times S \times B$ \textit{model spaces}. Below are a few known examples of embeddings between such model spaces:
 \begin{itemize}
\item Every regular tree $T_d$ quasi-isometrically embeds into the hyperbolic plane $\mathbb{H}^2$.
\item There are quasi-isometric embeddings $\mathbb{H}^3 \to \mathbb{H}^2 \times \mathbb{H}^2$ and $\mathbb{H}^5 \to \mathbb{H}^3 \times \mathbb{H}^3$ (see \cite{brady1998filling} for a more general statement).
\item There exist coarse embeddings $\mathbb{H}^2 \to \mathbb{H}^3 $ with arbitrarily small lower control \cite{benoist2021harmonic}.
\item$\mathbb{H}^n$ quasi-isometrically embeds into a product of $n$ binary trees $T_3 \times \dots \times T_3 $ \cite{buyalo2007embedding}.
\item There exist quasi-isometric embeddings of the product of $n$ copies of the hyperbolic plane $\mathbb{H}^2 \times \dots \times \mathbb{H}^2$ into the symmetric spaces $\SL_{n+1}(\mathbb{R})/\textup{SO}_{n+1}(\mathbb{R})$ and $\textup{Sp}_{2n}(\mathbb{R})/\textup{U}_{n}(\mathbb{R})$ \cite{fisher2018quasi}.
\end{itemize}

\noindent We will also consider general CAT(0) spaces that are proper (i.e.\ where closed balls are compact) and cocompact, and mapping class groups $\mathcal{MCG}(S_{g,p})$ of orientable compact connected surfaces of genus
$g$ and $p$ boundary components. 
Recall that the mapping class group of a surface $S$, $\mathcal{MCG}(S)$, is defined to be the group of orientation-preserving 
homeomorphisms up to isotopy. It is finitely-generated, and for any finite
generating set one considers the word metric, whence yielding a metric space which is 
unique up to quasi-isometry. 
\\
\\ We define the \textit{rank} of a metric space as the maximal dimension of an isometrically embedded copy of a Euclidean space. Similarly, we define the \textit{geometric rank} (or the \textit{quasi-flat rank} or the \textit{quasi-rank}) of a metric space $X$ as the maximal dimension of a quasi-isometrically embedded copy of a Euclidean space. We will denote it by $\textup{grank}(X)$. We introduce the geometric rank especially for mapping class groups since they are only defined up to quasi-isometry. For all the other spaces, the geometric rank is actually equal to the rank as we will see below. Therefore, by abuse of notation, when we consider the rank of a metric space, it should be understood that it is the geometric rank when the metric space is a mapping class group.

We are interested in the following natural question.
\begin{question}
Is the rank monotonous under coarse embeddings? In other words, if $X$ and $Y$ are such spaces and there is a coarse embedding $f:X \to Y$, does it imply that  $\textup{rank}(X) \leq \textup{rank}(Y)$?
\end{question}
\noindent  The answer is positive and well-known if one replaces the word coarse by quasi-isometric. It was shown by Anderson--Schroeder \cite{anderson1986existence} that if $X$ is a symmetric space of non-compact type and $f : \mathbb{R}^n \to X$ is a quasi-isometric embedding, then $X$ contains an $n$-flat, i.e.\ there exists an isometric embedding of $\mathbb{R}^n$ into $X$. Later, Kleiner \cite{kleiner1999local} generalized this to all locally compact cocompact Hadamard spaces.
In particular, for proper cocompact CAT(0) spaces, the geometric rank is equal to the rank. This answers the question for quasi-isometric embeddings. Indeed, if $f:X \to Y$ is a quasi-isometric embedding between two proper cocompact CAT(0) and $\textup{rank}(X) = p$ then $\mathbb{R}^p$ quasi-isometrically embeds into $Y$, hence $\textup{rank}(Y) \geq p$. Moreover, it was shown by Kleiner--Leeb \cite{kleiner1997rigidity} and by Eskin--Farb \cite{eskin1997quasi} that maximal quasi-flats in higher rank symmetric spaces of non-compact type satisfy a generalization of the Morse Lemma: if $Y$ is a symmetric space of non-compact type of rank $n \geq 2$, and $f:\mathbb{R}^n \to Y $ is a quasi-isometric embedding then there exist $\delta >0$ and $k\in \mathbb{N}$ (that depends on the quasi-isometry constants) such that $f(\mathbb{R}^n)$ lies in the $\delta$-neighborhood of a union of $k$ flats in $Y$. Note that this is a general phenomenon that occurs in all asymphoric hierarchically hyperbolic spaces \cite{behrstock2021quasiflats}.
\\
\\The monotonicity of the rank is no longer satisfied in the setting of coarse embeddings in general. Counter-examples are given by horospherical embeddings. Let $\mathbb{H} ^2$ denote the upper half-plane model
for the hyperbolic plane.
We have
\begin{equation*}
d_{\mathbb{H} ^2}((x,1),(x+r,1)) \, = \,
\arg\cosh (1 + r^2 /2)
\, \underset{\ell \to\infty}{\sim} \, 2 \ln r .
\end{equation*}
So the map $f : \mathbb{R} \longrightarrow \mathbb{H} ^2, \hskip.1cm x \longmapsto (x,1)$,
which is a parametrization of a horocycle, is a coarse embedding that is exponentially distorted. This carries over to the hyperbolic space of any dimension $n \ge 1$,
and provides coarse embeddings $f : \mathbb{R}^{n} \longrightarrow \mathbb{H}^{n+1}$
whose images are horospheres. These embeddings do not respect the monotonicity of the rank since $\textup{rank}(\mathbb{R}^n) = n$ and $\textup{rank}(\mathbb{H}^{n+1}) = 1$. Therefore, as long as the domain $X$ has a Euclidean factor $\mathbb{R}^p$ with $p\geq 2$, it can be embedded into a space with lower rank. Indeed, if $X= \mathbb{R}^p \times Z$ and $\textup{rank}(X) = p+r$, then it coarsely embeds into $ \mathbb{H}^{p+1} \times Z$, whose rank is $1+r$. This naturally leads us to ask the following questions.

\begin{question}
When $X$ has no Euclidean factor, can we still have a coarse embedding from $X$ to $Y$ such that $\textup{rank}(X) > \textup{rank}(Y)$? What about when $X$ has a Euclidean factor of dimension $1$? i.e.\ if $X = \mathbb{R} \times Z$, can we have a coarse embedding from $X$ to $Y$ such that $\textup{rank}(X) > \textup{rank}(Y)$?
\end{question}
\noindent We can show that, in the absence of a Euclidean factor in the domain $X$, the rank is monotonous under coarse embeddings in many cases, especially between symmetric spaces, thus answering a question of Fisher and Whyte \cite{fisher2018quasi}, see section 5.
\\
\\ In all the following results, the Euclidean buildings in the model space in the target are not necessarily thick, nor have cocompact affine Weyl group. It is only required for the Euclidean buildings in the domain.
\begin{thm} [See Theorem \ref{Thm 1 v2}]\label{Thm 1 v1}
Let $X= X_1 \times \dots \times X_k$ be a product of geodesic metric spaces of exponential growth. If $Y$ is a proper cocompact CAT(0) space of rank $< k$, or a mapping class group of geometric rank $<k$, then there is no coarse embedding from $X$ to $Y$.
\end{thm}
\noindent This result has been originally announced by Gromov in \cite{gromov1993asymptotic} section $7.E_2$, where he outlined a strategy of proof. Our proof relies mainly on his sketch. Moreover, we can extend this strategy to prove a similar result when $X$ is a general model space. Namely, we show
\begin{thm}[see Theorem \ref{Thm 2 v2}]\label{Thm 2 v1}
Let $X= S \times B$ a model space of rank $k$. If $Y$ is a proper cocompact CAT(0) space of rank $< k$, or a mapping class group of geometric rank $<k$, then there is no coarse embedding from $X$ to $Y$.
\end{thm}
\noindent This implies for example that there is no coarse embedding from $T_3 \times T_3 \times T_3 $ into $\mathbb{H}^p \times \mathbb{H}^q$ for any integers $p,q \geq 1$, nor into the symmetric space $\SL_{3}(\mathbb{R})/\textup{SO}_{3}(\mathbb{R})$.
\\
\\ When $X$ has a Euclidean factor of dimension 1, we can show that the monotonicity of the rank still holds when the target is a model space, answering a question by Gromov \cite{gromov1993asymptotic}, section $7.E_2$ remark (a). 
\begin{thm}[see Theorem \ref{Thm 3 v2}]\label{Thm 3 v1}
Let $X$ be either $X_1 \times \dots \times X_k$ as in Theorem \ref{Thm 1 v1}, or a model space $S \times B$ of rank $k$ as in Theorem \ref{Thm 2 v1}, and let $Y = \mathbb{R}^n \times S' \times B'$ be of rank $\leq k$. 
\\Then there is no coarse embedding from $X \times \mathbb{R}$ to $Y$.
\end{thm}
\noindent We also give a result about coarse embeddings of Euclidean spaces into symmetric spaces of lower rank, which uses the generalization of the Morse lemma for higher-rank symmetric spaces: 
\begin{thm}\label{Thm 4}
Let $p> k \geq 1$ be integers. Let $Y$ be a symmetric space of non-compact type such that $\textup{rank}(Y) = k$, and let $f : \mathbb{R}^{p} \to Y $ be a coarse embedding.
Then no Euclidean subspace $E \simeq \mathbb{R}^{k} \subset \mathbb{R}^{p} $ is sent quasi-isometrically by $f$. 
\end{thm}
\noindent In the rank one case, we can extend the result to Gromov-hyperbolic geodesic metric spaces with bounded geometry thanks to a result by Bonk--Schramm \cite{bonk2011embeddings}. This result was also suggested by Gromov in \cite{gromov1993asymptotic}, section $7.E_2$.
\begin{cor} \label{cor of Thm 4}
Let $Y$ be a Gromov-hyperbolic geodesic metric space with bounded geometry, and let $p \geq 2$ be an integer.
A coarse embedding $f : \mathbb{R}^p \to Y $ is uniformly compressing, i.e.\ it admits a sublinear upper control function $\rho^+$.
\end{cor}
\noindent Now, one can naturally ask: If $X$ does not coarsely embed into $Y$, does adding a Euclidean factor in the target makes the embedding possible?
\\ We can answer this question when $\textup{rank}(X) > \textup{rank}(Y)+1 $. In fact, when $\textup{rank}(X) > \textup{rank}(Y)+1 $, even adding a nilpotent factor in the target does not make this embedding possible:
\begin{cor}\label{cor 2}
Let $X$ be either $X_1 \times \dots \times X_k$ as in Theorem \ref{Thm 1 v1}, or a model space $S \times B$ of rank $k$ as in Theorem \ref{Thm 2 v1}, and let $Y = \mathbb{R}^n \times S' \times B'$ be a model space of rank $< k-1$. 
\\Then for any nilpotent connected Lie group $P$, there is no coarse embedding from $X$ to $Y \times P$.
\end{cor}
\begin{proof}
By Assouad's embedding theorem \cite{assouad1982distance}, $P$ coarsely embeds into some $\mathbb{Z}^d$, which in turn coarsely embeds into $\mathbb{H}^{d+1}$. Therefore, we get a coarse embedding from $X$ to $Y \times \mathbb{H}^{d+1}$, but $\textup{rank}(Y\times \mathbb{H}^{d+1}) = \textup{rank}(Y) +1 < k$. Such coarse embedding is not possible by Theorems \ref{Thm 1 v1} and \ref{Thm 2 v1}.
\end{proof}
\noindent The main tools that we are going to use are the homological filling functions, using Lipschitz chains. They measure the difficulty of filling cycles of a given size. The key point is that filling functions detect the rank of proper cocompact CAT(0) spaces, and the geometric rank of mapping class groups. In other words, they behave differently whether the dimension of the cycle is smaller or greater than the rank of the space. In fact, if $X$ is either a proper cocompact CAT(0) space or a mapping class group, a cycle $\Sigma$ in $X$ with dimension $k < \textup{rank}(X)$ and $\textup{Vol}_k^X(\Sigma) = \ell$ can be filled similarly as in a Euclidean space, i.e.\ it has a filling with $(k+1)$-volume $\precsim \ell^{\frac{k+1}{k}} $. On the other hand, if $k \geq \textup{rank}(X)$, then it can be filled in a sub-Euclidean way, i.e.\ it admits a filling with $(k+1)$-volume $= o\left(\ell^{\frac{k+1}{k}}\right)$. This result is due to Wenger \cite{wenger2011asymptotic} for proper cocompact CAT(0) spaces, and to Behrstock--Drutu \cite{behrstock2019higher} for mapping class groups, and is the main tool for the proof of Theorem \ref{Thm 1 v1} and Theorem \ref{Thm 2 v1}. For Theorem \ref{Thm 3 v1} and Theorem \ref{Thm 4}, we need a stronger condition on the filling of the target space, that is a linear filling above the rank. This is a result of Leuzinger \cite{leuzinger2014optimal} who proved that if $X$ is a model space, and $\Sigma$ is a $k$-cycle such that $k \geq \textup{rank}(X)$ and $\textup{Vol}_k^X(\Sigma) = \ell$, then it can be filled similarly as in a hyperbolic space, i.e.\ it has a filling with $(k+1)$-volume $\precsim \ell $.
\\
\\Our proof will also rely on the co-area formula which requires some slicing techniques. However, Lipschitz chains do not behave well under slicing, that is why we will consider metric currents in complete metric spaces that were introduced by Ambrosio--Kirchheim \cite{ambrosio2000currents}. They can be seen as a generalization of Lipschitz chains since any Lipschitz chain induces naturally a current, which behaves well under slicing.
\\
\\We denote by $\textup{FV}_{k}^{X}$ the $k$-filling function of a metric space $X$. We can restate our results in terms of filling functions, and in a more general framework, allowing the target space $Y$ to be any complete Lipschitz-connected metric space with at most exponential growth. A metric space $Y$ is \textit{Lipschitz-connected} if there is a $c\geq 1$ such that for any $d\in \mathbb{N}$ and any
$K$–Lipschitz map $f : S^d \to Y$ , there is a $cK$–Lipschitz extension $\Tilde{f} : D^{d+1} \to Y$ (see Section \ref{sect connect dots}). In particular, CAT(0) spaces and mapping class groups (up to a quasi-isometry) are Lipschitz-connected.
\begin{thm}\label{Thm 1 v2}
Let $X= X_1 \times \dots \times X_k$ be a product of geodesic metric spaces with exponential growth, $k\geq 2$ an integer, and let $Y$ be a Lipschitz-connected complete metric space with at most exponential growth. If $\textup{FV}_{k}^{Y}(\ell) = o\left(\ell^{\frac{k}{k-1}}\right)$, then there is no coarse embedding from $X$ to $Y$.
\end{thm}
\begin{thm}\label{Thm 2 v2}
Let $X= S \times B$ be a model space of rank $k \geq 2$, and let $Y$ be a Lipschitz-connected complete metric space with at most exponential growth.
\\If $\textup{FV}_{k}^{Y}(\ell) = o\left(\ell^{\frac{k}{k-1}}\right)$, then there is no coarse embedding from $X$ to $Y$.
\end{thm}
\begin{thm}\label{Thm 3 v2}
Let $X$ be either $X_1 \times \dots \times X_{k}$ as in Theorem \ref{Thm 1 v1}, or a model space $S \times B$ of rank $k$, and let $Y$ be a Lipschitz-connected complete metric space with at most exponential growth. If $\textup{FV}_{k+1}^{Y}(\ell)\sim \ell$, then there is no coarse embedding from $X \times \mathbb{R}$ to $Y$.
\end{thm}
\noindent Theorem \ref{Thm 1 v1} and Theorem \ref{Thm 2 v1}  immediately follow from Theorem \ref{Thm 1 v2} and Theorem \ref{Thm 2 v2} by results of Wenger \cite{wenger2011asymptotic} and Behrstock-Drutu \cite{behrstock2019higher}. Theorem \ref{Thm 3 v2} follows from Theorem \ref{Thm 3 v1} by Leuzinger's result \cite{leuzinger2014optimal}.
\subsection{Idea of the proof} 
\noindent Let us first look at Theorem \ref{Thm 1 v2}. The case $k=1$ is a simple volume obstruction observation: a space with exponential growth cannot be coarsely embedded with sublinear upper control $\rho_+$ in a space with at most exponential growth. Let us give a sketch of the proof when $k=2$, which already contains most of the conceptual difficulties of the general case. When $X=X_1 \times X_2$ is a product of two geodesic metric spaces of exponential growth and the 2-dimensional filling function of $Y$ is sub-Euclidean, suppose that a coarse embedding $f : X \to Y$ exists. By our observation, no copy of $X_1$ inside of $X$ is sent sublinearly, which means that there exist ``undistorted directions'', more precisely  there exist pairs of points $(a_n)_n$, $(b_n)_n$ in $X_1$ such that $d_n := d_X(a_n,b_n) \to \infty$ and which are mapped in $Y$ quasi-isometrically. Now consider the rectangles $R_n$ with one geodesic segment from $a_n$ to $b_n$ in $X_1$ as one side, and a geodesic segment in a copy of $X_2$ of length at most $\varphi \left(d_n\right)$ as the orthogonal side, where $\varphi$ is a sublinear function that depends on $\textup{FV}_{2}^{Y}$ and such that $f(\partial R_n)$ can be filled in $Y$ by a chain of volume $= o\left( d_n \, \varphi(d_n)  \right)$. Since $(a_n)$ and $(b_n)$ are sent quasi-isometrically, the height of $f(\partial R_n)$ is comparable to that of $\partial R_n$, however the filling of $f(\partial R_n)$ in $Y$ is much smaller than that of $\partial R_n$ in $X$. This implies, by the co-area formula for currents, that the width of $f(\partial R_n)$ is highly compressed. From this observation, we construct a sequence of sets in $X$ whose ``coarse''-volume, that will be defined in the background material, is not coarsely preserved by $f$, which is not possible for coarse embeddings. For $k\geq 3$, we proceed by induction on $k$. Suppose that we have proved the result for some $k \geq 2$, and that there exists a coarse embedding $f$ from $X= X_1 \times \dots \times X_{k+1}$ to $Y$ that satisfies a sub-Euclidean $(k+1)$-filling. By the induction hypothesis, there exists a sequence of $k$-dimensional rectangles in a copy of $X_1 \times \dots \times X_{k}$, such that the image of their boundaries have fillings comparable to that in the domain. The sub-Euclidean $(k+1)$-filling implies that there exist, orthogonally to these $k$-rectangles, sets that are highly compressed by $f$, which contradicts the coarse volume preservation.
\\
\\To prove Theorem \ref{Thm 2 v2}, which implies Theorem \ref{Thm 2 v1}, we will adapt the previous proof. Since the domain is no longer a product space, we will consider parallelograms instead of rectangles, and adapt the size of its sides to make the proof work. More importantly, instead of looking for ``undistorted directions'', we need to find ``maximally singular undistorted directions''. Maximally singular directions are directions following which a space of rank $k$ contains a subset that factorizes as a product $\mathbb{R} \times X'$, where $X'$ has rank $k-1$. To find these directions, we will decompose the cycle along the walls of a suitable Weyl sector, and use a pigeonhole-like reasoning. We refer to section \ref{sect 4.3}, after the statement of Theorem \ref{CE of model space}, to a more detailed description of the idea of proof.
\\
\\Theorem \ref{Thm 3 v2} is a consequence of the first two theorems, using the linear filling above the rank and Theorem \ref{CE of euclidean spaces}, which is a result about coarse embeddings of Euclidean spaces into symmetric spaces of lower rank. This result allows us to derive an upper bound for fillings of images of $(k-1)$-parallelograms, starting from the linear filling of dimension $k+1$ in $Y$. 
\\
\\The proof of Theorem \ref{Thm 4} is done by contradiction. If there exists such quasi-isometric embedding, then by the quasi-flats Theorem of Kleiner--Leeb \cite{kleiner1997rigidity} and Eskin--Farb \cite{eskin1997quasi}, such a quasi-flat is a bounded distance from a finite union of maximal flats. This implies that there exist arbitrarily big balls that are sent uniformly close to a maximal flat. By considering parallelograms inside these balls and the filling of the projection of their images in this flat, we get a contradiction. Namely, we get a sequence of images of parallelograms that are sent quasi-isometrically in a Euclidean space, and that are filled linearly, which gives a contradiction by applying the co-area formula. Finally, a result of Bonk--Schramm \cite{bonk2011embeddings} says that any Gromov hyperbolic geodesic metric space with bounded geometry quasi-isometrically embeds into a hyperbolic space $\mathbb{H}^n$. By composing by this quasi-isometry and applying the previous theorem we get the Corollary \ref{cor of Thm 4}.
\subsection{Organization of the paper}
We will start in section 2 by giving some background material that will be used throughout the paper. More precisely, section 2.1 is dedicated to coarse embeddings. In section 2.2 we collect definitions of the model spaces and mapping class groups, and give some large-scale properties, like parallel sets and cross sections in CAT(0) spaces. In 2.3 we give a brief overview of the theory of metric currents introduced by Ambrosio--Kirchheim, and state the Slicing Theorem that will play a crucial role in the proofs. We define the filling functions in section 2.4 and state results of Wenger \cite{wenger2011asymptotic}, Behrstock-Drutu \cite{behrstock2019higher} and Leuzinger \cite{leuzinger2014optimal} that shows that the filling functions can detect the rank of the spaces we are considering. Finally, connect the dots argument in 2.5 will allow us to always assume that our coarse embeddings are Lipschitz, which is very convenient when working with Lipschitz chains or currents. Before proving Theorem \ref{Thm 1 v2} in section 3.2, which implies Theorem \ref{Thm 1 v1}, we will give some useful lemmas about coarse embeddings and volume preservation in section 3.1. In section 4, we give a preliminary result about Weyl sectors in either symmetric spaces of non-compact type or Euclidean buildings, and we prove a decomposition result for parallelograms that will be used in the proof of Theorem \ref{Thm 2 v2} right after. The goal of section 5 is to give a proof of Theorem \ref{Thm 3 v2}, thus implying Theorem \ref{Thm 3 v1}. It starts first by giving a result about coarse embeddings of Euclidean spaces into symmetric spaces of lower rank, which also implies Theorem \ref{Thm 4} and Corollary \ref{cor of Thm 4}. Finally, we collect some further questions in section 6.
\subsection{Acknowledgements}

I am extremely grateful to my advisor Romain Tessera for introducing me to this subject, for all the discussions that we had and all his support. I would like to thank Jason Behrstock and Yves Benoist for reading an earlier version and for several valuable comments. I wish also to thank Charles Frances for explaining to me the motivations and applications of coarse embeddings in Lorentzian geometry. I have benefited from many useful discussions with David Hume, John Mackay, Pierre Pansu, Stefan Wenger, Yves de Cornulier and Linus Kramer. I also thank them for their feedback on the previous version. Finally, I would like to thank Amandine Escalier and Mingkun Liu for many discussions and for helping me with the drawings.

\section{Background}
\subsection{Coarse embeddings}
A map $f :(X,d_X)\to (Y,d_Y)$ is a coarse embedding if there exist functions $\rho_{\pm}:[0,\infty)\to [0,\infty)$ such that $\rho_-(r)\to\infty$ as $r\to\infty$ and for all $x,y\in X$
$$\rho_-(d_X(x,y)) \leq d_Y(f(x),f(y)) \leq \rho_+(d_X(x,y)).$$
Equivalently, for every pair $(x_n)_{n \ge 0}$, $(x'_n)_{n \ge 0}$ 
of sequences of points in $X$,
$$\lim_{n \to \infty} d_X(x_n, x'_n) = \infty \iff \lim_{n \to \infty} d_Y(f(x_n), f(x'_n)) = \infty.$$
When the control functions $\rho_-, \rho_+$ are affine, $f$ is said to be a \textit{quasi-isometric embedding}.
\\The map $f$ is said to be \textit{large-scale Lipschitz} if it admits an affine upper control, without necessarily having a lower control.
\begin{defn}
 A metric space $(X,d)$ is \textit{large-scale geodesic} if there exist constants $\lambda,c > 0$ and $b\geq 0$ such that,
for every pair $(x,x')$ of points of $X$,
there exists a sequence $x_0=x, x_1, \dots, x_n=x'$ of points in $X$
such that $d(x_{i-1},x_i) \le c$ for $i = 1, \dots, n$ and $n \le \lambda d(x,x') + b$.
\end{defn}
If $X$ is large-scale geodesic, then $f$ admits a control function $\rho_+$ which is affine.
\begin{defn}\label{bounded geom}
A metric space $X$ has \textit{bounded geometry} 
if there exists $R_0 \ge 0$ such that, for every $R \ge 0$, 
there exists an integer $N$ such that every ball of radius $R$ in $X$ 
can be covered by $N$ balls of radius $R_0$.
\end{defn}
\begin{rem}
An example of a metric space that is not of bounded geometry is an infinite-dimensional Hilbert space, with the metric given by the norm. 
\end{rem}
\begin{defn}
Let $X$ be a metric space with bounded geometry and let $\varepsilon >0$. We define the \textit{$\varepsilon$-volume} of a subset $A \subset X$, that we will denote $\textup{Vol}_X^{\varepsilon}(A)$, as the minimal number of balls of radius $\varepsilon$ needed to cover $A$.
\end{defn}
We will denote by $B(x,r)$ the closed ball $B(x,r):= \{y\in X : d(y,x)\leq r\}$.
\begin{defn}
Let $\varepsilon >0$. We define the \textit{$\varepsilon$-growth function} of a metric space $(X,d)$ with bounded geometry to be
$$ \beta_X^{\varepsilon}(r) = \sup \{ \textup{Vol}_X^{\varepsilon}\left(B(x,r)   \right) \mid x \in X  \}.  $$
We say that $X$ has \textit{polynomial growth} if there exist $\varepsilon >0$ and $D\geq 0$ such that $\beta_X^{\varepsilon}(r) \preccurlyeq  r^D $.
\\We say that $X$ has \textit{exponential growth} if there exists $\varepsilon >0$ such that $\beta_X^{\varepsilon}(r) \approx e^r $.
\\The comparison on functions $\mathbb{N}\to \mathbb{R}$ is defined as follows: $f \preccurlyeq g$ if there exists a constant $c>0$ such that $f(n) \leq c g(cn+c)+c$ for all $n \in \mathbb{N}$, and $f\approx g$ if $f \preccurlyeq g$ and $g \preccurlyeq f$.
\end{defn}
\begin{rem}
If $X$ is large-scale geodesic and has bounded geometry, then it has at most exponential growth $\beta_X^{\varepsilon}(r) \preccurlyeq  e^r $.
\end{rem}
\subsection{Coarse geometry of model spaces and mapping class groups}
This section is devoted to recall the definition of the spaces we will be considering. We will define model spaces as products of symmetric spaces and Euclidean buildings with a Euclidean factor, and recall the definition of mapping class groups. We will give some properties of these spaces that will be useful.
\\Symmetric spaces and Euclidean buildings can be seen as leading examples of CAT(0) spaces. In fact, they seem to be the most rigid among all proper CAT(0) spaces \cite{caprace2009isometry}\cite{leeb326characterization}.
\subsubsection{Symmetric spaces}
A \textit{symmetric space} is a connected riemannian manifold $M$ such that for all $x\in M$, there exists an isometry $\sigma_x \in \textup{Isom}(M)$ such that $\sigma_x(x) = x$ and $T_x \sigma_x = -\textup{Id}_{T_xM}$. It is said to be \textit{of nonpositive curvature} if it has non-positive sectional curvature. If moreover it has no non-trivial Euclidean factor, $M$ is called a \textit{symmetric space of non-compact type}. All symmetric spaces of non-compact type can be obtained as coset spaces $G/K$, where $G$ is a connected semi-simple Lie group with trivial center and no compact factors, and $K$ is a maximal compact subgroup of $G$. The metric on $G/K$ comes from the Killing form of $\textup{Lie}(G)$. When $M$ is a symmetric space of non-compact type, $G$ is the identity component of the isometry group $\textup{Isom}_0(M)$. \\An important example of symmetric spaces of non-compact type is given by $\SL_n(\mathbb{R})/\textup{SO}_n(\mathbb{R})$. When $n=2$, it corresponds to the hyperbolic plane $\mathbb{H}^2$. More generally, if $M$ is a symmetric space of non-compact type such that $\textup{dim}(\textup{Isom}_0(M)) = n \geq 2$, then after rescaling the metrics of the de Rham factors of $M$ by positive constants, $M$ can be isometrically embedded in $\SL_n(\mathbb{R})/\textup{SO}_n(\mathbb{R})$ \cite{eberlein1996geometry}. The symmetric space $\SL_n(\mathbb{R})/\textup{SO}_n(\mathbb{R})$ can also be seen as the collection of scalar products on $\mathbb{R}^n$, for which the unit ball has volume $1$.
\subsubsection{Euclidean buildings}
Euclidean buildings are non-Archimedean analogues of symmetric spaces of non-compact type. In fact, to any semi-simple algebraic group over a local field, $\textup{SL}_n(\mathbb{Q}_p)$ for example, we can associate a Euclidean building, called its Bruhat-Tits Building, on which the group acts isometrically and co-compactly. Let us define Euclidean buildings in general. A general reference for buildings is \cite{abramenko2008buildings}, and for an introductory course see \cite{caprace2014lectures}.
\\
\\ Let $W < \textup{Isom}(\mathbb{R}^n) $ be a discrete  reflection group, i.e.\ a discrete subgroup of $\textup{Isom}(\mathbb{R}^n)$ generated by orthogonal reflections through a collection $\mathcal{H}$ of hyperplanes, called $\textit{walls}$, and such that this collection is locally finite. The pattern determined by the set of walls defines a cellular decomposition $\Sigma$ of $\mathbb{R}^n$, called a $\textit{Euclidean Coxeter complex}$. The group $W$ is called the $\textit{affine Weyl group}$ (or the $\textit{affine Coxeter group}$) of the Euclidean Coxeter complex $\Sigma$. A \textit{chamber} (or \textit{alcove}) in that complex is defined as a connected component of $\mathbb{R}^n - \cup_{h \in \mathcal{H}} h$. The affine Weyl group $W$ acts transitively on the set of chambers. The top-dimensional cells of $\Sigma$ are the closures of chambers, also called \textit{closed chambers}, which are compact if and only if the group $W$ acts cocompactly. A lower dimensional cell is the intersection of a closed chamber  with a set of walls.

\begin{defn}
Let $W$ be an affine Weyl group of $\mathbb{R}^n$, and $\Sigma$ its associated Euclidean Coxeter complex. A \textit{(discrete) Euclidean building} modeled on ($\Sigma,W)$ is a cell complex $B$, which is covered by subcomplexes all isomorphic to $\Sigma$, called the \textit{apartments} of $B$, and such that the following incidences properties hold:
\begin{enumerate}
    \item Any two cells of $B$ lie in some apartment.
    \item For any two apartments, there is an isomorphism between them fixing their intersection pointwise.
\end{enumerate}
\end{defn}
Euclidean buildings are frequently called \textit{Affine buildings} in the literature.
\\A Euclidean  building is \textit{thick} if each wall belongs to
at least 3 half-apartments with disjoint interiors. If the affine Weyl group acts cocompactly on the apartments, then the chambers of the Euclidean Coxeter complex are polysimplices. This induces
on the Euclidean building a structure of a polysimplicial complex. If the building is moreover irreducible, then this complex is a simplicial complex. 
\\
A Euclidean building always possesses a geometric realization, which is a complete CAT(0) metric space, and such that the restriction of its distance to each apartment is the Euclidean metric.
\subsubsection{Model spaces}
 We call \textit{model spaces} the spaces of the form $X = \mathbb{R}^n \times S \times B$, where $S$ is a product of symmetric spaces of non-compact type and $B$ a product of thick Euclidean buildings with cocompact affine Weyl group, with bounded geometry and no Euclidean factor. Note that such spaces are CAT(0), thus Lipschitz-connected.
 \\
 \subsubsection{Mapping class groups} References for this subsection are \cite{farb2011primer},\cite{hamenstaedt2005geometry}, \cite{behrstock2008dimension}.
\\We consider orientable compact connected surfaces $S=S_{g,p}$ of genus
$g$ and $p$ boundary components. 
The mapping class group, $\mathcal{MCG}(S)$, is defined to be 
$Homeo^{+}(S)/Homeo_0(S)$, the group of isotopy classes of homeomorphisms of $S$. Boundary components are not required to be fixed by the mapping class group, so each boundary should be considered as a puncture. This group is finitely-generated and for any finite
generating set one considers the word metric in the usual way, yielding a metric space which is 
unique up to quasi-isometry. It is known that mapping class groups are not CAT(0) because they cannot act geometrically on a CAT(0) space when the surface has genus $\geq 3$, or genus $2$ and at least one puncture \cite{kapovich1996actions}. Nonetheless, their filling functions behave like those of CAT(0) spaces, as we will see in section 2.4.
\\Throughout the remainder, we assume that $3g-3+p \geq 2$, i.e. we exclude the sphere with at most 4 punctures and the torus with at most 1 puncture. The mapping class group of these 7 surfaces is either finite or virtually free, so quasi-isometric to a point or a locally finite regular tree, which is already covered by the Euclidean buildings that we consider.
\\By a result of \cite{mosher1995mapping}, mapping class groups are automatic, so they are combable in the sens that their Cayley graphs admit a bounded quasi-geodesic combing, see \cite{bridson2013metric}. Therefore, they are of type $\mathcal{F}_{\infty}$ \cite{epstein1991word}. In particular, for every $n \in \mathbb{N}$, there exists an $n$-connected CW-complex $X_S$ on which $\mathcal{MCG}(S)$ acts freely, properly discontinuously and cocompactly. 
  \subsubsection{Rank and geometric rank}
\begin{defn}
Let $X$ be a metric space. We define its \textit{rank} as  
\\$\textup{rank}(X) = \max \{k \in \mathbb{N} \mid \exists \, g : \mathbb{R}^k \to X \textup{ an isometric embedding}     \}$.
\end{defn}
Similarly, we define the geometric rank, that we also call quasi-flat rank or quasi-rank.
\begin{defn}
Let $X$ be a metric space. We define its \textit{geometric rank} as  
\\$\textup{grank}(X) = \max \{k \in \mathbb{N} \mid \exists \, g : \mathbb{R}^k \to X \textup{ a quasi-isometric embedding}     \}$.
\end{defn}
We recall that, by a result of \cite{anderson1986existence} and \cite{kleiner1999local}, the geometric rank is equal to the rank for proper cocompact CAT(0) spaces. For mapping class groups, it was shown by \cite{behrstock2008dimension} and by \cite{hamenstaedt2005geometry} that the geometric rank of a mapping class group $\mathcal{MCG}(S_{g,p})$ is equal to the maximal rank of its free abelian subgroups, which is equal to $3g-3+p$ by \cite{birman1983abelian}. In particular, this rank is realized by 
any subgroup generated by Dehn twists on a maximal set of disjoint essential
simple closed curves.
\begin{ex}
For all $p\geq 1$, $\textup{rank}(\mathbb{R}^p) = p$. 
\\For all $n\geq 2$, $\textup{rank}(\mathbb{H}^n) = 1$ and $\textup{rank}(\SL_n(\mathbb{R})/\textup{SO}_n(\mathbb{R})) =\textup{rank}(\textup{BT}(\SL_n(\mathbb{Q}_p)))= n-1$, where $\textup{BT}(\SL_n(\mathbb{Q}_p))$ is the Bruhat-Tits building of $\SL_n(\mathbb{Q}_p)$.
\end{ex}
\begin{rem}
Note that the rank of a Euclidean building is equal to its dimension, and $\textup{rank}(\mathbb{R}^n \times S \times B) = n+\textup{rank}(S)+\textup{dim}(B)$.
\end{rem}
\subsubsection{Parallel sets in model spaces} \label{parallel sets}
References for this subsection are
\cite{eberlein1996geometry},\cite{helgason2001differential} for symmetric spaces, and \cite{kleiner1997rigidity}, \cite{leeb326characterization} for Euclidean buildings.
\\
\\Let $X$ be a model space of rank $k$. A subspace of $X$ is a \textit{flat} if it is
isometric to some Euclidean space. A \textit{maximal flat} is a flat of dimension $k$, and flats of dimension 1 are the geodesics. A geodesic $\gamma$ is said to be \textit{regular} if it is contained in a unique maximal flat. Otherwise it is called \textit{singular}. It is said to be \textit{maximally singular} if it belongs to $k$ half-flats of maximal dimension, with disjoint interiors. A \textit{singular flat} is a flat which is the intersection of maximal flats.
\\
\\
Let $X$ be a model space with no Euclidean factor, and let $F \subset X$ be a flat. If another flat $F'$ is at finite Hausdorff distance from $F$ then by the Flat Strip Theorem \cite{bridson2013metric}, since $X$ is CAT(0), there exists a segment $I$ such that the convex hull of the union $F \cup F'$ is isometric to $F \times I$. In that case, the flats $F$ and $F'$ are called \textit{parallel}. The flat $F$ is a closed convex subset with extendible geodesics (i.e.\ each geodesic segment is contained in a bi-infinite geodesic in $F$), therefore we can consider its \textit{parallel set} $P_X(F)$, which is the union of flats that are parallel to $F$. It is a closed convex subset of $X$ isometric to a product 
$$  P_X(F) = F \times CS_X(F)                   ,$$
where $CS_X(F)$ is a model space. It is called the \textit{cross section} of $F$. Moreover, if $F$ is a singular flat then its cross section has no Euclidean factor and $\textup{rank}(CS_X(F))= \textup{rank}(X)-\textup{dim}(F)$. The flats in $X$ being just products of flats in its factors, and singular flats being products of singular flats in the factors, the cross section of a product is the product of cross sections (see \cite{leeb326characterization} and \cite{kleiner1997rigidity}).
\begin{ex}Let $X$ be the Bruhat-Tits building associated to $\SL_3( \mathbb{Q}_p)$ and let $\gamma$ be a maximally singular geodesic. We then have that $CS_X(\gamma)$ is isometric to the $(p+1)$-regular simplicial tree $T_{p+1}$, and
\[P_X(\gamma) = \RR\times T_{p+1}.\]
\end{ex}
\begin{ex}
Let $Y=\SL(3,\RR)/\textup{SO}(3)$, and let $\gamma$ be a maximally singular geodesic. We then have that $CS_Y(\gamma)$ is isometric to the hyperbolic plane $\mathbb{H}^2$, and
\[P_Y(\gamma) = \RR\times \mathbb{H}^2.\]
\end{ex}
\hfill
\\In the proof of Theorem \ref{Thm 2 v2}, we will need a uniform lower bound on the volume growth of cross sections of singular flats. It is obvious for symmetric spaces since cross sections of singular flats are symmetric spaces of lower rank, and there are only finitely many isometry classes. For Euclidean buildings we need the following result.
\begin{prop} \label{prop cross section growth}
Let $X$ be a thick Euclidean building with cocompact affine Weyl group. Then for every singular flat $F$, there exists $D = D(F) >0$ such that the 3-regular tree of edge length $D$ isometrically embeds into the cross section of $F$. This implies that there exists a uniform lower bound on the volume growth of cross sections of singular flats in $X$.
\end{prop}
\begin{proof}
Let $X$ be such Euclidean building, and let $n = \dim X$. If $n = 1$, $X$ is a tree an the singular flats are the vertices. Thus the cross sections are the whole building. Let $D>0$ be the distance between vertices in an apartment. Since $X$ is thick, it contains a 3-regular tree of edge length $D$. 
\\Suppose that $n \geq 2$, and let $F$ be a singular flat of dimension $k \leq n-1$. Let $F'$ be a singular $(n-1)$-flat (i.e.\ a wall) that contains $F$. Any $(n-1)$-flat parallel to $F'$ contains a $k$-flat parallel to $F$, therefore $P_X(F') \subset P_X(F)$, which implies that $\mathbb{R}^{n-1-k} \times CS_X(F') \subset CS_X(F)$. So it is enough to show it for $k=n-1$. 
\\The affine Weyl group $W$ is a semi-direct product $W = W_r \ltimes T$, where $W_r$ is the associated finite Coxeter group and $T$ is a group of translations generated by vectors orthogonal to some walls. Since it is cocompact, $T$ is isomorphic to $\mathbb{Z}^n$. Let $F \subset X$ be a wall, and fix a generating set of $T$. Since the affine Weyl group is cocompact, there exists $D = D(F) >0$ such that the walls parallel to $F$ in an apartment are at distance an integer multiple of $D$, see also \cite{kleiner1997rigidity} Corollary 5.1.3. In particular, every wall parallel to $F$ admits at least three parallel walls in $X$ at distance $D$. Consider the graph $G$ whose vertex set is the set of walls parallel to $F$, and whose edges are pairs
of walls that differ by a translation by a generator of $T$ in an apartment. If moreover the length of the edges is the distance in $X$ between the corresponding walls, then the graph $G$ is isometric the the cross section of $F$, which is called the \textit{wall tree} of $F$. It is a thick tree, see \cite{weiss2008structure} Chapter 10 and \cite{kramer2014coarse}, and contains a subtree which is 3-regular of edge length D, obtained by only considering for each vertex two parallel walls at distance D.
\\Moreover, since there is only a finite number of singular flats up to translation, there is only a finite number of values $D(F)$. Therefore, there is a uniform lower bound on the volume growth of such wall trees: if we fix $\varepsilon >0$, there exists $\mu >0$ such that for any singular flat $F$ and for all $R>0$
$$        \textup{Vol}^{\varepsilon} \left( B_{CS_X(F)}(R)  \right) \geq \exp(\mu R)  .              $$
\end{proof}
\subsection{Metric currents}
References for this section are \cite{ambrosio2000currents},\cite{huang2022morse} and  \cite{wenger2003isoperimetric}.
\\Ambrosio and Kirchheim extended the classical theory of normal and integral currents developped by Federer and Fleming \cite{federer1960normal} to arbitrarily complete metric spaces.
\\Let $(X,d)$ be a complete metric space and $k\geq 0$ and let $\mathcal D^k(X)$ denote the set of $(k+1)$-tuples $(f,\pi_1,\dots,\pi_k)$ 
of Lipschitz functions on $X$ with $f$ bounded.
\\
\begin{defn}
A \textit{$k$-dimensional metric current}  $T$ on $X$ is a multi-linear functional on $\mathcal D^k(X)$ satisfying the following
properties:
\begin{enumerate}
 \item If $\pi^j_i$ converges point-wise to $\pi_i$ as $j\to\infty$ and if $\sup_{i,j}\textup{Lip}(\pi^j_i)<\infty$ then
       \begin{equation*}
         T(f,\pi^j_1,\dots,\pi^j_k) \to T(f,\pi_1,\dots,\pi_k).
       \end{equation*}
 \item If $\{x\in X:f(x)\not=0\}$ is contained in the union $\bigcup_{i=1}^kB_i$ of Borel sets $B_i$ and if $\pi_i$ is constant 
       on $B_i$ then
       \begin{equation*}
         T(f,\pi_1,\dots,\pi_k)=0.
       \end{equation*}
 \item There exists a finite Borel measure $\mu$ on $X$ such that
       \begin{equation}\label{equation:mass-def}
        |T(f,\pi_1,\dots,\pi_k)|\leq \prod_{i=1}^k\textup{Lip}(\pi_i)\int_X|f|d\mu
       \end{equation}
       for all $(f,\pi_1,\dots,\pi_k)\in\mathcal{D}^k(X)$.
\end{enumerate}\end{defn}
The space of $k$-dimensional metric currents on $X$ is denoted by $\mathbf M_k(X)$ and the minimal Borel measure $\mu$
satisfying \eqref{equation:mass-def} is called \textit{mass} of $T$ and written as $\|T\|$. Let us denote by $\mass{T} = \|T\|(X) $ the \textit{total mass} of $T$. 
\textit{The support of $T$} is, by definition, the closed set $\textup{spt} T$ of points $x\in X$ such that $\|T\|(B(x,r))>0$ for all $r>0$. 
\\Every function $g \in L_{loc}^1(\mathbb{R}^k)$ induces a current $[\![ g ]\!] \in \mathbf M_k(\mathbb{R}^k) $. Indeed, by Rademacher's theorem, Lipschitz function are differentiable almost everywhere, so if   $(f,\pi_1,\dots,\pi_k) \in \mathcal{D}^k(\mathbb{R}^k) $, $[\![ g ]\!]$ is defined by: :
$$ [\![ g ]\!](f,\pi_1,\dots,\pi_k) :=   \int_{\mathbb{R}^k} g f\det\left(\frac{\partial\pi_i}{\partial x_j}\right)d{\mathcal L}^k          $$
Which corresponds to the integration of the differential form $f d\pi_1 \wedge \dots \wedge d\pi_k $, weighted by $f$. Therefore, every borel set $A \subset \mathbb{R}^k$ induces a current $[\![ A ]\!] := [\![ \chi_A ]\!]$.
\\
\\The restriction of $T\in\mathbf{M}_k(X)$ to a Borel set $A\subset X$ is given by 
\begin{equation*}
  (T\rstr A)(f,\pi_1,\dots,\pi_k):= T(f\chi_A,\pi_1,\dots,\pi_k).
\end{equation*}
This expression is well-defined since $T$ can be extended to a functional on tuples for which the first argument lies in 
$L^\infty(X,\|T\|)$.
\\
\\\textit{The boundary of} $T\in\mathcal{D}_k(X)$ is defined by analogy with Stokes formula:
\begin{equation*}
 \partial T(f,\pi_1,\dots,\pi_{k-1}):= T(1,f,\pi_1,\dots,\pi_{k-1}).
\end{equation*}
It is clear that $\partial T$ satisfies conditions (1) and (2) in the above definition. If $\partial T$ also has 
finite mass (condition (3)) then $T$ is called a \textit{normal current}. The space of normal currents is denoted by $\mathbf{N}_k(X)$.
\\ The push-forward of $T\in\mathbf{M}_k(X)$ 
under a Lipschitz map $\varphi$ from $X$ to another complete metric space $Y$ is given by
\begin{equation*}
 \varphi_\# T(g,\tau_1,\dots,\tau_k):= T(g\circ\varphi, \tau_1\circ\varphi,\dots,\tau_k\circ\varphi),
\end{equation*}
for $(g,\tau_1,\dots,\tau_k)\in\mathcal{D}^k(Y)$. This defines a $k$-dimensional current on $Y$.
\\
\\Let $\mathcal{H}^k$ denote the Hausdorff $k$-dimensional 
measure. An $\mathcal{H}^k$-measurable set $A\subset X$
is said to be \textit{countably $\mathcal{H}^k$-rectifiable} if there exist Lipschitz maps $f_i :B_i\longrightarrow X$ from subsets
$B_i\subset \mathbb{R}^k$ such that
\begin{equation*}
\mathcal{H}^k(A - \bigcup f_i(B_i))=0.
\end{equation*}
\begin{defn}
A current $T\in\mathbf{M}_k(X)$ with $k\geq 1$ is said to be \textit{rectifiable} if
\begin{enumerate}
 \item $\|T\|$ is concentrated on a countably $\mathcal{H}^k$-rectifiable set and
 \item $\|T\|$ vanishes on $\mathcal{H}^k$-negligible sets.
\end{enumerate}
$T$ is called \textit{integer rectifiable} if, in addition, the following property holds:
\begin{enumerate}
 \item[(3)] For any Lipschitz map $\varphi\colon X\longrightarrow \mathbb{R}^k$ and any open set $U\subset X$ there exists 
       $\theta\in L^1(\mathbb{R}^k,\mathbb{Z})$ such that
       \begin{equation*}
        \varphi_\#(T\rstr U)=[\![ \theta ]\!].
       \end{equation*}
\end{enumerate}
\end{defn}
\hfill
\\We denote the space of rectifiable currents by $\mathcal{R}_k(X)$, and $\mathcal{I}_k(X)$ for integer rectifiable currents.
Integer rectifiable normal currents are called \textit{integral currents}, denoted by $\mathbf{I}_k(X)$.
An element $T\in\mathbf{I}_k(X)$ is called a \textit{cycle}  $\partial T=0$.
\\
\\The main motivation for using currents instead of Lipschitz chains is the following Slicing Theorem due to Ambrosio--Kirchheim \cite{ambrosio2000currents} that will play a crucial role in this paper.
\begin{thm}\label{slicing thm}
Let be $T\in\mathbf{I}_k(X)$ and $\pi$ a Lipschitz function on $X$. Then there exists for almost every $r\in\mathbb{R}$ an integral current
$\slice{T}{\pi}{r}\in\mathbf{I}_{k-1}(X)$ with the following properties:
\begin{enumerate}
 \item $\slice{T}{\pi}{r}= \partial(T\rstr\{\pi \leq r\}) - (\partial T)\rstr\{\pi\leq r\}$,
 \item $\|\slice{T}{\pi}{r}\|$ and $\|\partial\slice{T}{\pi}{r}\|$ are concentrated on $\pi^{-1}(\{r\})$,
 \item $\mass{\slice{T}{\pi}{r}}\leq\textup{Lip}(\pi)\frac{d}{dr}\mass{T\rstr\{\pi\leq r\}}$, which is just a reformulation of the co-area formula.
\end{enumerate}
\end{thm}
\subsection{Homological filling functions}
The basic idea of a filling function is to measure the difficulty of
filling a boundary of a given size.  There are several ways to make
this rigorous, depending on the type of boundary and the type of
filling. We will use the Homological filling, as
in \cite{gromov1993asymptotic}\cite{leuzinger2014optimal}, which consists on filling Lipschitz cycles by Lipschitz chains, instead of the homotopical filling where we fill spheres by balls.
\\
\\\textit{An integral Lipschitz $d$-chain} in a complete metric space $X$
is a finite linear combination $\Sigma=\sum_i a_i\sigma_i$, with $a_i\in \mathbb Z$, of
Lipschitz maps $\sigma_i:\Delta^k\to X$ from the Euclidean $k$-dimensional simplex $\Delta^k$
to $X$.  We will often call this simply a \textit{Lipschitz $k$-chain}. The boundary operator is defined as in the case of singular chains.
\\Note that by the push-forward of metric currents, every Lipschitz chain in X induces an integral current. Indeed, if $\sigma_i:\Delta^k\to X$ is Lipschitz, then ${\sigma_i}_\#( [\![ \Delta^k ]\!] )$ is an integral $k$-current in $X$. We define the \textit{$k$-volume of $\sigma_i$} as the total mass of ${\sigma_i}_\#( [\![ \Delta^k ]\!] )$: 
$$\textup{Vol}_k\sigma_i = \mass{{\sigma_i}_\#( [\![ \Delta^k ]\!] )}. $$
Note that if $X$ is a riemannian manifold then, by Rademacher's theorem, $\sigma_i$ is differentiable almost everywhere and  $\textup{Vol}_k\sigma_i$ is also equal to the integral of the magnitude of its Jacobian. Note also that it is not necessarily equal to the volume of the image (its hausdorff measure), unless the map is injective, because the volume is counted with multiplicity.
\\We then define the $k$-\emph{volume} of a Lipschitz $k$-chain as
$$\textup{Vol}_k^X\, \Sigma:=\sum_i|a_i|\textup{vol}_k\sigma_i.$$
We wish to  measure the difficulty to fill Lipschitz $k$-cycles by Lipschitz $(k+1)$-chains. 
\\To ensure that such a filling exists, the space considered must be $k$-connected, i.e.\ all  homotopy groups $\pi_k(X)$ are trivial (which is the case for CAT(0) spaces and for the corresponding CW-complexes $X_S$ associated to $\mathcal{MCG}(S)$) hence  the corresponding homology groups are also trivial.
\\More precisely,  for  an integral Lipschitz $k$-cycle $\Sigma$ in $k$-connected space we define its \emph{filling volume}
$$
\textup{FillVol}_{k+1}^X(\Sigma):=\inf\{\textup{Vol}_{k+1}\, (\Omega) \mid \Omega \textup{ is a Lipschitz}\  (k+1)\textup{-chain with}\ \partial \Omega=\Sigma\}.
$$
The \emph{ $(k+1)$-dimensional 
filling function} of $X$ is  then  given by
$$
\textup{FV}_{k+1}^{X}(\ell):=\sup\{\textup{FillVol}_{k+1}^X(\Sigma)\mid \Sigma \textup{ is a Lipschitz}\  k\textup{-cycle in X with}\ \textup{Vol}_k^X (\Sigma) \leq \ell\}.
$$
We are only interested in the asymptotic behaviour of these filling functions. We have that if $X$ and $Y$ are two $k$-connected manifolds or simplicial complexes which are quasi-isometric, then by \cite{pride1999higher}
$$\textup{FV}_{k}^{X}(\ell)\sim \textup{FV}_{k}^{Y}(\ell).$$
The equivalence relation
on functions $\mathbb R\to \mathbb R$ is define as follows: we write $f\precsim g$ if there is a constant $C>0$ such that 
$f(x)\leq Cg(Cx+C)+Cx+C.$
We write $f\sim g$ and say that they are \textit{asymptotically equivalent} if  $f\precsim g$ and $g \precsim f$.
\begin{ex} \hfill
\\• A hyperbolic space $X = \mathbb{H}^n$ satisfies for all $2 \leq k \leq n$, $\textup{FV}_k^{X}(\ell)\sim \ell$  \cite{gromov1987hyperbolic} \cite{lang2000higher}.
\\• A Euclidean space $X = \mathbb{R}^n$ satisfies for all $2 \leq k \leq n$, $\textup{FV}_k^{X}(\ell)\sim  \ell^\frac{k}{k-1}$  \cite{federer1960normal}.
\\• More generally, \cite{wenger2003isoperimetric} showed that a complete CAT(0) space satisfies $\textup{FV}_k^{X}(\ell)\precsim  \ell^\frac{k}{k-1}$ for all dimensions $k$.
\\• Mapping class groups are combable, so by \cite{behrstock2019combinatorial} they satisfy for every $k$,  $\textup{FV}^{k}(\ell)\precsim  \ell^\frac{k}{k-1}$. The filling functions are defined in the corresponding CW-complex associated to the mapping class group.
\end{ex}

\hfill
\\The following key theorems say that the filling functions can detect the rank of these spaces.
\begin{thm}[\cite{wenger2011asymptotic}]
\hfill\\
Let $X$ be a proper cocompact CAT(0) space. If $k> \textup{rank}(X)$, then $\textup{FV}_{k}^{X}(\ell) = o\left(\ell^{\frac{k}{k-1}}\right)$.
\end{thm}
Wenger Actually proved it for all complete quasi-geodesic metric spaces admitting cone-type inequalities up to dimension $k$. He also showed that the filling functions are asymptotically equal to the Euclidean filling functions below the rank.
\begin{thm}[\cite{behrstock2019higher}]
\hfill\\
Let $X=\mathcal{MCG}(S_{g,p})$. If $k> \textup{grank}(X) = 3g-3+p$, then $\textup{FV}_{k}^{X}(\ell) = o\left(\ell^{\frac{k}{k-1}}\right)$.
\end{thm}
It is known that below the geometric rank, the filling functions are asymptotically equal to the Euclidean filling functions.
\begin{thm}[\cite{leuzinger2014optimal}]
\hfill\\
Let $X= \mathbb{R}^d\times S \times B$ be a model space. Then

\noindent \textup{(i)} \ \  $X$ has  Euclidean  filling functions  below the rank:
$$\textup{FV}_k^{X}(\ell)\sim  \ell^\frac{k}{k-1}\ \ \textup{ if}\ \ 2\leq  k \leq \textup{rank}\ (X);$$
\textup{(ii)} \ \  $X$ has linear filling functions above the rank:
$$\textup{FV}_{k}^{X}(\ell)\sim \ell\ \ \textup{ if}\ \  \textup{rank}\ (X) < k\leq \dim (X).$$
\end{thm}
When $X$ is a symmetric space of non-compact type, this result has been correctly asserted by Gromov in \cite{gromov1993asymptotic}, section 5.D, and proposed a 
possible different proof for the upper bound by projecting the cycle to a maximal flat, see \cite{leuzinger2014optimal} for more details. The proof of Leuzinger uses a different approach by projecting to a suitable horosphere.
\\
\\Note that by the push-forward of metric currents, every Lipschitz chain in X induces an integral current. Moreover, if $\sigma$ is a Lipschitz $k$-chain $\textup{Vol}_X^k(\sigma) = \mass{\sigma}$. Therefore, the the chains that we will consider will be seen as currents, so it is convenient to define filling functions in this setting. 
\begin{defn}
Let $(X,d)$ be a complete metric space and $k\geq 0$ an integer. For a $k$-cycle $\Sigma \in \mathbf{I}_k(X)$, we define its \textit{$(k+1)$-dimensional current-filling volume} to be $$\textup{FillVol}_{k+1}^{X,\mathrm{cr}}(\Sigma):= \inf \{ \mass{\Omega}  \mid \Omega \in \textbf{I}_{k+1} \textup{ with } \partial \Omega = \Sigma            \} . $$ 
\end{defn}
Since every Lipschitz chain induces a metric current, it is clear that if $\Sigma$ is a Lipschitz $k$-cycle seen in $ \mathbf{I}_k(X)$, then $\textup{FillVol}_{k+1}^{X,\mathrm{cr}}(\Sigma) \leq \textup{FillVol}_{k+1}^{X}(\Sigma)$.
\subsection{Connect the dots argument}\label{sect connect dots}
\begin{defn}
A metric space $Y$ is \textit{Lipschitz-connected} if there exists $c\geq 1$ such that for any $d\in \mathbb{N}$ and any
$K$–Lipschitz map $f : S^d \to Y$ , there is a $cK$–Lipschitz extension $\Tilde{f} : D^{d+1} \to Y$.
\end{defn}
In particular, metric spaces admitting a convex bicombing are Lipschitz-connected \cite{schlichenmaier2005quasisymmetrically}, proposition 6.2.2. It is the case for CAT(0) spaces, because they are geodesic with convex distance function, and for mapping class groups because they are quasi-isometric to CAT(0) spaces \cite{petyt2021mapping}.
\begin{lem} \label{connect the dots}
Let $f:X\to Y$ be a large-scale Lipschitz map, $X$ a finite dimensional CW complex where the size of the cells is globally bounded, and $Y$ is Lipschitz-connected.
Then $f$ is a bounded distance from a Lipshitz map 
$g:X\to Y$. That is, $\sup_{x\in X}d_Y(f(x),g(x)) \leq C$ for 
some constant $C>0$. 
\end{lem}
\begin{proof}
Let us prove it by induction on the dimension of the cells of $X$, as suggested in \cite{block1997large}. Let $D = \sup_{\sigma} \textup{diam}(\sigma)$ over all the cells of $X$, let $(\lambda,C)$ be the large-scale Lipschitz constants of $f$, and let $c$ be the Lipschitz-connectivity constant of $Y$. Without loss of generality, we can assume that all the edges in $X^{(1)}$ have length $1$. For any vertex $v \in X^{(0)}$, define $g(v) := f(v)$. Given an edge $e =(u,v) \in X^{(1)} $, $d_Y(g(u),g(v)) \leq \lambda + C$. So $g$ is defined on $S^0 = \partial e$ and is $(\lambda + C)$-Lipschitz, therefore it can be extended to $e$ as a $(c(\lambda + C))$-Lipschitz map. $g$ is defined at this stage on $X^{(1)}$, and is uniformly close to $f$ on $X^{(1)}$. Indeed, if $x \in X^{(1)} $ is in an edge $e=(u,v)$,
\begin{equation*}
    \begin{split}
        d_Y(f(x),g(x)) & \leq d_Y(f(x),f(u)) + d_Y(g(u),g(x))
        \\ & \leq \lambda + C + c(\lambda + C).
    \end{split}
\end{equation*}
Suppose now that $g$ is defined on $X^{(k)}$ and is $K$-Lipschitz. Let $\sigma$ be a $(k+1)$-cell in $X$. The restriction of $g$ to $\partial \sigma$ can then be extended to $\sigma$ as $cK$-Lipschitz map. Same as before, $g$ that is now defined on $X^{(k+1)}$ is Lipschitz and a bounded distance from $f$. If $x\in \sigma$, and $u \in \partial \sigma$,
\begin{equation*}
    \begin{split}
        d_Y(f(x),g(x)) & \leq d_Y(f(x),f(u)) + d_Y(g(u),g(x))
        \\ & \leq \lambda D + C + cK D.
    \end{split}
\end{equation*}
Since $X$ is finite dimensional, the induction stops.
\end{proof}
This lemma applies to all the coarse embeddings that we will be considering, so we can always assume that they are actually Lipschitz.
\section{The domain is a product of spaces with exponential growth}
The goal of this section is to prove Theorem \ref{Thm 1 v2}, which implies Theorem \ref{Thm 1 v1}.
\subsection{Preliminary lemmas}
Let us give some useful lemmas first. One crucial property that we are going to use is that the $\varepsilon$-volume is coarsely preserved by coarse embeddings:
\begin{lem}\label{coarse volume}
Let $X$ and $Y$ be metric spaces with bounded geometry, $f : X \to Y$ a coarse embedding, and let $\varepsilon >0$ big enough. There exist $\alpha , \beta >0 $ such that for all $A \subset X$ with $\textup{Vol}_X^{\varepsilon}(A) < \infty $,
$$  \alpha \, \textup{Vol}_X^{\varepsilon} ( A)  \leq      \textup{Vol}_Y^{\varepsilon} ( f( A ) ) \leq \beta \, \textup{Vol}_X^{\varepsilon} ( A )  .      $$
\end{lem}
\begin{proof}
Let $f$ be such coarse embedding with control functions $\rho_-$ and $\rho_+$. $X$ and $Y$ have bounded geometry, so there exist $R_0$ and $R_0'$ as in definition \ref{bounded geom}. Take $\varepsilon \geq \max(R_0,R_0')$. 
\\ • Let $n \in \mathbb{N}$ and $x_1,\dots, x_n \in X$ such that 
$$ A \subset B_X ( x_1 , \varepsilon  ) \cup ... \cup B_X ( x_n , \varepsilon  )        .$$
Then 
$$ f(A)  \subset f(B_X ( x_1 , \varepsilon  )) \cup ... \cup f(B_X ( x_n , \varepsilon  ))  \subset B_Y ( f(x_1) , \rho^+ (\varepsilon)  ) \cup ... \cup B_Y ( f(x_n) , \rho^+ (\varepsilon)  )      .$$
Since $Y$ has bounded geometry, there exists $
\beta = \textup{Vol}_X^{\varepsilon}(B(\rho^+ (\varepsilon) ) ) \in \mathbb{N}$ such that every ball of radius $\rho^+ (\varepsilon)$ is in the union of $p$ balls of radius $\varepsilon$. Therefore $f(A)$ is in the union of $\beta \times n$ balls of radius $\varepsilon$. Since its volume is the minimum among such number of balls, we have
$$    \textup{Vol}_Y^{\varepsilon} ( f( A  ) ) \leq \beta n .  $$
By taking $n$ to be minimal, we have 
$$   \textup{Vol}_Y^{\varepsilon} ( f( B_X (x,R)  ) ) \leq  \beta \, \textup{Vol}_X^{\varepsilon} ( B_X ( x , R      ) ) .    $$
\\ • Now let $n \in \mathbb{N}$ and $y_1,\dots, y_n \in Y$ such that 
$$ f(A)  \subset B_Y ( y_1 , \varepsilon  ) \cup \dots \cup B_Y ( y_n , \varepsilon  )        .$$
Then 
$$ A \subset f^{-1}\left(f( A)\right) \subset f^{-1}\left( B_Y ( y_1 , \varepsilon  )) \cup \dots \cup f^{-1} ( B_Y ( y_n , \varepsilon  ) \right)     .$$
Note that for any $i = 1, \dots, n$, $f^{-1} ( B_Y ( y_i , \varepsilon  ) )$ has diameter less than $\rho_-^{-1}(2 \varepsilon)$ : $\forall a,b \in f^{-1} ( B_Y ( y_i , \varepsilon  ) ) $, 
$ d_Y(f(a), f(b)) \leq 2 \varepsilon$, so $ d_X(a,b) \leq  \rho_-^{-1}(2 \varepsilon) $. For any $i = 1,\dots,n$ such that  $f^{-1} ( B_Y ( y_i , \varepsilon  ) )$ is non-empty, take some $x_i$ in it. So $f^{-1} ( B_Y ( y_i , \varepsilon  ) ) \subset    B_X(x_i, \rho_-^{-1}(2 \varepsilon))         $. 
\\Let us denote $D = \rho_-^{-1}(2 \varepsilon)$. Therefore
$$ A \subset B_X ( x_1 , D  ) \cup ... \cup B_X ( x_n , D  )        .$$
$X$ has bounded geometry, so there exists $
\gamma = \textup{Vol}_X^{\varepsilon}(B( D ) ) \in \mathbb{N}$ such that every ball of radius $D$ is in the union of $\gamma$ balls of radius $D$. So $A$ is in the union of $\gamma \times n$ balls of radius $\varepsilon$. Since its volume is the minimum among such number of balls, we have 
$$    \textup{Vol}_X^{\varepsilon} (A) \leq \gamma n.   $$
By taking $n$ to be minimal and by denoting $\alpha = \gamma^{-1}$, we have 
$$   \textup{Vol}_Y^{\varepsilon} ( f(A) ) \geq  \alpha \, \textup{Vol}_X^{\varepsilon} (A) .  $$ 
\end{proof}
Let us also prove the following lemma. 
\begin{lem}\label{volume neighb}
Let $X$ be a metric space with bounded geometry, $\varepsilon>0$ big enough, and let $\delta>0$. If $A \subset X$ with $\textup{Vol}_X^{\varepsilon}(A) < \infty $, then
 $$\textup{Vol}_X^{\varepsilon}( \textup{N}_{\delta}(A)   ) \leq \beta_X^{\varepsilon}(\delta+\varepsilon) \times \textup{Vol}_X^{\varepsilon}(A).  $$
\end{lem}
\begin{proof}
Let $\varepsilon \geq R_0$ as in the previous lemma. $X$ has bounded geometry, so there exist $n \in \mathbb{N}$ and $ x_1,\dots,x_n \in X$ such that $A \subset B_X(x_1,\varepsilon) \cup \dots \cup B_X(x_n,\varepsilon) $.
\\Let $z \in N_{\delta}(A) $, i.e.\ there exists $ a \in A$ such that $d_X(z,a) \leq \delta$. So there exists $i \in \{1,\dots,n\}$ such that $d_X(z,x_i) \leq \delta + \varepsilon$, therefore $N_{\delta}(A) \subset B_X(x_1,\varepsilon + \delta) \cup .. \cup B_X(x_n,\varepsilon + \delta)   $.
\\There exists $p= \beta_X^{\varepsilon}(\delta+\varepsilon) $ such that every ball of radius $\varepsilon + \delta$ can be covered by $p$ balls of radius $\varepsilon$. Therefore $N_{\delta}(A)$ can be covered by $p \times n$ balls of radius $\varepsilon$, so
$$\textup{Vol}_X^{\varepsilon}( \textup{N}_{\delta}(A)   ) \leq \beta_X^{\varepsilon}(\delta +\varepsilon) \times n . $$
By taking $n$ to be minimal, we have the result.
\end{proof}
We will also need the following result.
\begin{lem} \label{filling current distance}
Let $X$ be a complete metric space, and $x,y \in X$. Let $T \in \mathbf{I}_1(X)$ such that $\partial T = [\![ x ]\!]-[\![ y ]\!]$, then
$$ \mass{T} \geq d_X(x,y) ,          $$
and $$   \textup{FillVol}^{X,\mathrm{cr}}_{1}([\![ x ]\!]-[\![ y ]\!]) \geq d_X(x,y). $$
If moreover there is a geodesic segment from $x$ to $y$, then 
$$   \textup{FillVol}^{X,\mathrm{cr}}_{1}([\![ x ]\!]-[\![ y ]\!]) = d_X(x,y).            $$
\end{lem}
\begin{proof}
$\partial T = [\![ x ]\!]-[\![ y ]\!]$ means that for all $ f : X \to \mathbb{R}$ Lipschitz and bounded, we have 
$$ \partial T (f) = f(x) - f(y).           $$
By definition, we have $\partial T (f) = T(1,f)$. So $ T(1,f) = f(x) - f(y). $
\\T being a 1-current, there exists a finite Borel measure $\mu$ on $X$ such that for all $g : X \to \mathbb{R}$ Lipschitz and bounded and for all $f : X \to \mathbb{R}$ Lipschitz
$$    | T(g,f)| \, \leq \, \textup{Lip}(f) \int_X | g(x) | \, d\mu (x).         $$
Let $\mu = \|T\|$ be the minimal Borel measure. In particular, we have that $\mu (X) = \|T\|(X) = \mass{T} $.
\\For $g=1$ and for all $f : X \to \mathbb{R}$ Lipschitz, we have 
$$  | f(x) - f(y)   |   =  | T(1,f)   | \leq    \textup{Lip}(f) \int_X | 1(x) | \, d\mu (x)   =   \textup{Lip}(f) \mass{T}.                   $$
So $\mass{T}$ satisfies : $\forall f : X \to \mathbb{R}$ Lipschitz, 
  $| f(x) - f(y)   |   \leq       \textup{Lip}(f) \mass{T}                   $. This implies that
  $$ \sup_{f \, 1-\textup{Lip}}   \frac{| f(x) - f(y)   |}{\textup{Lip}(f)} \leq \mass{T}.                      $$
Consider the map $d_y : X \to \mathbb{R}$, $\forall z \in X$, $d_y(z) := d_X(y,z) $. It is a $1$-Lipschitz map that satisfies:
$$   \frac{| d_y(x) - d_y(y)   |}{\textup{Lip}(f)}    = d_X(x,y)  .                    $$
Therefore
$$\sup_{f \, 1-\textup{Lip}}   \frac{| f(x) - f(y)   |}{\textup{Lip}(f)} = d_X(x,y),  $$
and 
$$ \mass{T} \geq d_X(x,y).           $$
By taking the infimum of the mass of all such $1$-currents, 
$$   \textup{FillVol}^{X,\mathrm{cr}}_{1}([\![ x ]\!]-[\![ y ]\!]) \geq d_X(x,y). $$
If there is a geodesic segment $\gamma$ from $y$ to $x$, consider the Lipschitz $1$-chain realized by $\gamma$. Then $\partial \gamma = [\![ x ]\!]-[\![ y ]\!] $ and $\mass{\gamma} = d_X(x,y)$, and the infimum is attained.
\end{proof}
\subsection{Proof of Theorem \ref{Thm 1 v2}}
Let us prove a more general result where we do not actually require the $k$-filling function of the target space to be sub-Euclidean, but we only need a weaker condition on the filling of the boundaries of some $k$-rectangles.
\\For all that follows, if $f$ and $g$ are functions from $\mathbb{R}$ to $\mathbb{R}$, we will denote $f \ll g$ if $f=o(g)$ at $+\infty$, and $f \gg g$ if $g=o(f)$ at $+\infty$.
\begin{thm} \label{CE of product}
Let $(\varphi_i)_{i\in \mathbb{N}^*}$ be a sequence of functions from $\mathbb{R}_+$ to $\mathbb{R}_+$ such that $\varphi_1(d)=d$ for all $d$, and for all $i\in \mathbb{N}^*$ $\varphi_i \gg \varphi_{i+1} $, and $\varphi_i(d)$ tends to $+\infty$ at $+\infty$.
\\Let $k\geq 1$ be an integer, and let $X = X_1 \times ... \times X_k$ be a product of geodesic metric spaces with bounded geometry, and let $Y$ be a Lipschitz-connected complete metric space with at most exponential growth. Let $f : X \to Y$ be a large-scale Lipschitz map. If
\\ • for all $i = 1,\dots,k$, $X_i$ has exponential growth,
\\ • there exists a sublinear function $\phi$ such that for every $d>0$ big enough and every $k$-dimensional rectangle $R_{d}$, whose sides are geodesic segments in the $X_i \,'s$, with side lengths $l_1, \dots, l_k$ that satisfy $l_1 = \varphi_1(d), l_2 \leq \varphi_2(d),\dots,l_k \leq \varphi_k(d)$, we have
$$
\textup{FillVol}_{k}^{Y, \mathrm{cr}}\left(f(\partial R_{d})\right) \leq \phi ( \varphi_1(d) \times \dots\times\varphi_k(d)),
$$
then $f$ is not a coarse embedding.
\end{thm}
\begin{rem}
This condition means that the filling of the image of the boundary of such rectangles in $Y$ is sublinear when compared to their filling in $X$.
\end{rem}
\begin{rem}
This result can be seen as a generalization of the fact, for $k=1$, that a space with exponential growth cannot be coarsely embedded with a sublinear $\rho^+$ in a space with at most exponential growth, because the second condition means that $f$ admits a sublinear $\rho^+$. Indeed, $R_d$ is just a geodesic segment in $X$. Let $x$ and $y$ be its extremities. So the second condition can be rewritten as
$$
\textup{FillVol}_{1}^{Y, \mathrm{cr}}([\![ f(x) ]\!]-[\![ f(y) ]\!]) \leq \phi ( d_X(x,y)).
$$
By lemma \ref{filling current distance}, we have
$$
d_Y(f(x),f(y)) \leq   \textup{FillVol}_{1}^{Y, \mathrm{cr}}([\![ f(x) ]\!]-[\![ f(y) ]\!]). 
$$
Therefore
$$
d_Y(f(x),f(y)) \leq  \phi ( d_X(x,y)).
$$
\end{rem}
\begin{proof}
Let us prove the theorem by induction on $k$.
\\• For $k=1$, let $X$ be a metric space with exponential growth, and suppose that such a coarse embedding $f : X \to Y$ exists for a sublinear $\phi$.
\\Let $x\in X$ and $R>0$. Then
 $$f( B_X (x,R)  ) \subset B_Y ( f(x) , \phi(R)      ) . $$
Let $\varepsilon >0$. By taking the  $\varepsilon$-volumes :
\begin{center}
    $  \textup{Vol}_Y^{\varepsilon} ( f( B_X (x,R)  ) ) \leq  \textup{Vol}_Y^{\varepsilon} ( B_Y ( f(x) , \phi(R)      ) )   .        $
\end{center}
$Y$ has at most exponential growth so there exists $\lambda \geq 0$ such that for all $ y \in Y$ and for all $R>0$,
\begin{center}
    $  \textup{Vol}_Y^{\varepsilon} (  B_Y (y,R)   ) \leq  e^{ \lambda R  }   .        $
\end{center}
But, on one hand $f$ is a coarse embedding, so by lemma \ref{coarse volume} it preserves the volume coarsely. In particular there exists $\alpha >0$ such that $\forall x \in X$
$$  \alpha \, \textup{Vol}_X^{\varepsilon} ( B_X ( x , R      ) )  \leq      \textup{Vol}_Y^{\varepsilon} ( f( B_X (x,R)  ) )       .$$
On the other hand, $X$ has exponential growth, so there exists $\gamma >0$ such that for all $x \in X$ and for all $R>0$
$$\textup{Vol}_X^{\varepsilon} (  B_X (x,R)   ) \geq  e^{ \gamma R  } .          $$
So we finally get
$$ \alpha e^{ \gamma R  } \leq \alpha \, \textup{Vol}_X^{\varepsilon} ( B_X ( x , R      ) )  \leq \textup{Vol}_Y^{\varepsilon} ( f( B_X (x,R)  ) ) \leq \textup{Vol}_Y^{\varepsilon} ( B_Y ( f(x) , \phi(R)      ) )  \leq  e^{ \lambda \phi(R)  } .   $$
Which is not possible for a sublinear function $\phi$.
\\Note that for $k=1$ we did not require $X$ to be geodesic.
\\
\\Before proving it for all $k\geq 2$, and just for the sake of clarity, let us first prove it for $k=2$. Note that it already contains most of the conceptual difficulties of the general case.
\\
\\• Let $X = X_1 \times X_2$ be a product of two geodesic metric spaces of exponential growth and $Y$ be a metric space with at most exponential growth. Suppose that $f : X \to Y$ is a coarse embedding that satisfies the second condition, i.e.\ there exists a sublinear function $\phi$ such that for every $d>0$ and every $2$-dimensional rectangle $R_{d}$, whose sides are geodesic segments in the $X_i \,'s$, of side lengths $l_1,l_2$ with $l_1 = d$ and $l_2 \leq \varphi_2(d)$, we have
$$
\textup{FillVol}_{2}^{Y, \mathrm{cr}}(f(\partial R_{d})) \leq \phi \left( d \times \varphi_2(d)\right).
$$
Following the case $k=1$, for all $x_2 \in X_2$, the copy $X_1 \times \{x_2\} $ cannot be sent sublinearly. Therefore, if we fix some $x_2 \in  X_2$, there exist $\lambda >0$ and two sequences $(a_n)_{n \in \mathbb{N}}$ and $(b_n)_{n \in \mathbb{N}}$ in $X_1 \times \{x_2\} $ such that
$$
d_n := d_X(a_n,b_n) \underset{n\to +\infty}{\longrightarrow} +\infty
\qquad
\text{and}
\qquad
d_Y(f(a_n),f(b_n)) \geq \lambda \, d_X(a_n,b_n).
$$
Let $n \in \mathbb{N}$ and take $x_2' \in X_2$ such that $d_{X_2}(x_2,x_2') \leq \varphi_2(d_n)$. Consider $a_n' = (\textup{proj}_{X_1}(a_n), x_2')$, and $b_n' = (\textup{proj}_{X_1}(b_n), x_2')$.
\\ Let $c_1$ be a geodesic segment in $X_1$ going from $\textup{proj}_{X_1}(a_n)$ to $\textup{proj}_{X_1}(b_n)$, and $c_2$ be a geodesic segment in $X_2$ from $x_2$ to $x_2'$.
\\Now consider the four following geodesic segments in $X$ : 
$$
\gamma_1 = ( c_1, x_2 ),
\qquad
\gamma_2 = ( \textup{proj}_{X_1}(b_n), c_2 ),
\qquad
\gamma_3 = ( -c_1, x_2 ),
\qquad
\gamma_4 = ( \textup{proj}_{X_1}(a_n), -c_2 ).
$$
By concatenating them, i.e.\ taking their formal sum as Lipschitz chains, we get a Lipschitz 1-cycle that we will denote $R_n$.
\\The second condition on the filling volume implies that 
$$
\textup{FillVol}_{2}^{Y, \mathrm{cr}}(f(R_n)) \leq \phi \left( d_n \,  \varphi_2(d_n)\right) .
$$
This implies that there exists $S_n \in \mathbf{I}_2(Y)$ in $Y$ such that $\partial S_n = f(R_n)$ and 
$$
\mass{S_n} \leq \phi \left( d_n \, \varphi_2(d_n)\right).
$$
Now consider the 1-Lipschitz map $\pi : Y \to \mathbb{R}$, $\pi(y) = d_Y(y,f(\gamma_1)) $.
\\By the Slicing Theorem, we have that for a.e.\ $t \in \mathbb{R}$, there exists $<S_n,\pi,t> \, \in \mathbf{I}_1(Y)$ such that $\slice{S_n}{\pi}{t}= \partial(S_n\rstr\{\pi \leq t\}) - (\partial S_n)\rstr\{\pi\leq t\}$, and by integrating the co-area formula over the distance $t$, we have
$$\mass{\slice{S_n}{\pi}{t}}\leq\frac{d}{dt}\mass{S_n\rstr\{\pi\leq t\}},$$
$$ \int_0^{+\infty} \mass{\slice{S_n}{\pi}{t}} dt \leq \int_0^{+\infty} \frac{d}{dt}\mass{S_n\rstr\{\pi\leq t\}},$$
$$ \int_0^{+\infty} \mass{\slice{S_n}{\pi}{t}} dt \leq \mass{S_n}.$$
Since $\mass{S_n} \leq \phi \left( d_n \, \varphi_2(d_n)\right) $, we get
\begin{equation}\label{equation 1}
\int_0^{D} \mass{\slice{S_n}{\pi}{t}} dt \leq \int_0^{+\infty} \mass{\slice{S_n}{\pi}{t}} dt \leq \phi \left( d_n \, \varphi_2(d_n)\right),
\end{equation}
where $D=d_Y(f(\gamma_1),f(\gamma_3))$. However, for a.e.\ $t \in \left]0,D\right[$, $\mass{\slice{S_n}{\pi}{t}}$ cannot be too small since the current $\slice{S_n}{\pi}{t}$ almost gives a filling for the $0$-cycle $[\![ f(b_n) ]\!] - [\![ f(a_n) ]\!]$.
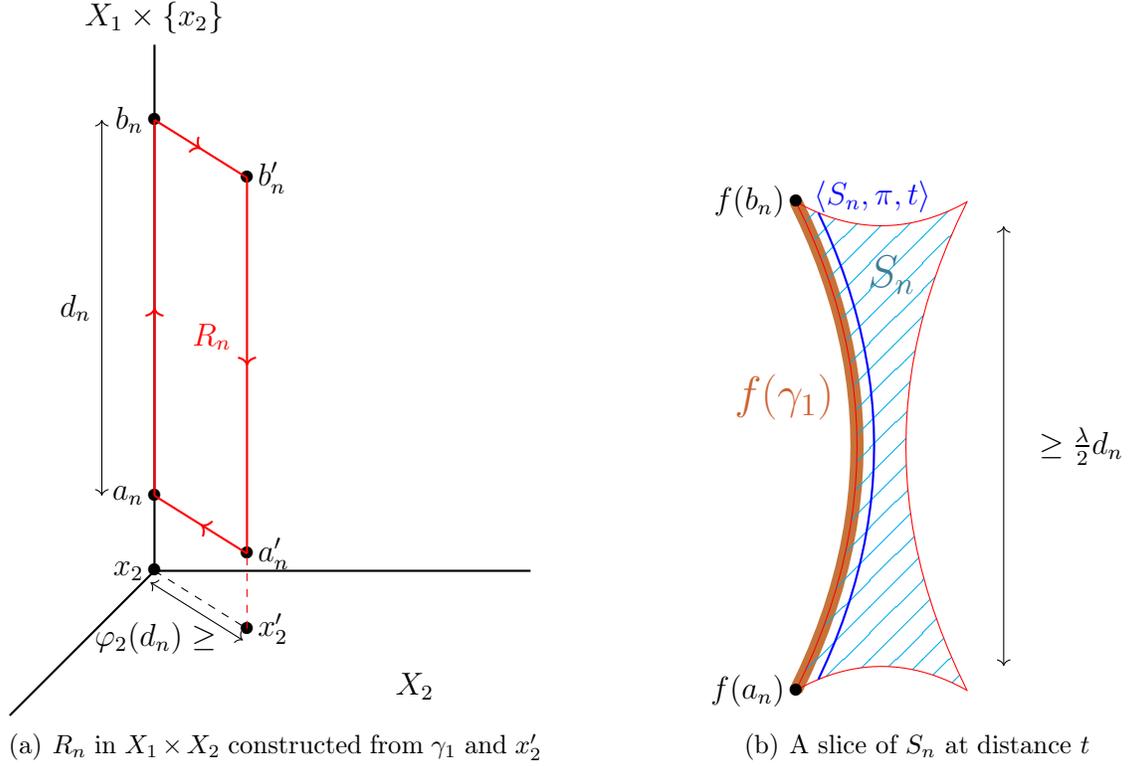
\begin{figure}%
\begin{center}%
\subfigure[$R_n$ in $X_1 \times X_2$ constructed from $\gamma_1$ and $x_2'$]{%
\label{fig:first}%
\begin{tikzpicture}
\draw[thick] (0,0,0)--(0,0,5); 
\draw[thick] (0,0,0)--(0,7,0);
\draw[thick] (0,0,0)--(5,0,0);

\coordinate (a_n) at (0,1,0);
\node at (a_n) [left] {$a_n$};
\node at (a_n) {\textbullet};

\coordinate (a_n') at (2,1,2);
\node at (a_n') [right] {$a_n'$};
\node at (a_n') {\textbullet};

\coordinate (b_n) at (0,6,0);
\node at (b_n) [left] {$b_n$};
\node at (b_n) {\textbullet};

\coordinate (b_n') at (2,6,2);
\node at (b_n') [right] {$b_n'$};
\node at (b_n') {\textbullet};

\coordinate (x_2) at (0,0,0);
\node at (x_2) [left] {$x_2$};
\node at (x_2) {\textbullet};

\coordinate (x_2') at (2,0,2);
\node at (x_2') [right] {$x_2'$};
\node at (x_2') {\textbullet};

\node[above] at (0,7,0) {$X_1 \times \{x_2\}$};

\node at (5,0,4) {$X_2$};

\node[right, color=red] at (0.75,3.5,1) {$R_n$};
\node[below, color=black] at (0.4,-0.15,1) {$\varphi_2(d_n) \geq$};

\draw[dashed] (x_2) -- (x_2');
\draw[dashed, color=red] (x_2') -- (a_n');
\begin{scope}[thick, decoration={markings, mark=at position 0.5 with {\arrow{>}}}]

    \draw[postaction={decorate}, color=red] (a_n) -- (b_n);
    \draw[postaction={decorate}, color=red] (b_n) -- (b_n');
    \draw[postaction={decorate}, color=red] (b_n') -- (a_n');
    \draw[postaction={decorate}, color=red] (a_n') -- (a_n);
\end{scope}

 \draw[<->] (-0.7,1,0) to (-0.7,6,0);
  \node[left] at (-0.7,3.5,0) {$d_n$};
  
   \draw[<->] (0.1,0,0.4) to (2.1,0,2.4);
\end{tikzpicture}
}%
\goodgap
\goodgap
\goodgap
\goodgap
\goodgap
\goodgap
\subfigure[A slice of $S_n$ at distance $t$]{%
\label{fig:second}%
\begin{tikzpicture}[scale=0.65]
    \draw[domain = 5:-5, color=red!10!brown, line width=5pt] plot ({-0.5-0.05*\x*\x}, {\x});
    \fill[pattern=my north east lines, pattern color=cyan, draw=red]
        plot [domain = 5:-5] ({-0.5-0.05*\x*\x}, {\x})
        --
        plot [domain = -1.75:1.75] (\x,{-4.51-0.16*\x*\x})
        --
        plot [domain = -5:5] ({0.5+0.05*\x*\x}, {\x})
        --
        plot [domain = 1.75:-1.75] (\x,{4.51+0.16*\x*\x});

        \draw [domain = 4.77658:-4.77658, thick, color=blue] plot ({-0.15-0.05*\x*\x}, {\x});

        \coordinate (fb_n) at (-1.75,5);
        \node at (fb_n) [left] {$f(b_n)$};
        \node at (fb_n) {\textbullet};

        \coordinate (fa_n) at (-1.75,-5);
        \node at (fa_n) [left] {$f(a_n)$};
        \node at (fa_n) {\textbullet};

        \node[color=cyan!50!black, right] at (-0.5,3.5) {\Large{$S_n$}};

        \draw[<->] (2.5,-4.5) to (2.5,4.5);
        \node[right] at (3,0) {$\geq \frac{\lambda}{2} d_n$};

        \node[right,color=blue] at (-1.6,5.1) {{$\langle S_n,\pi,t \rangle$}};

        \node[left,color=red!25!brown] at (-0.75,1) {\Large{$f(\gamma_1)$}};
\end{tikzpicture}}%
\end{center}
\caption{The cycle $R_n$ in $X$, and a slice of the filling of its image in $Y$}
\end{figure}
\begin{claim}
For $n$ big enough and for a.e.\ $t \in \left]0,D\right[$, $\mass{\slice{S_n}{\pi}{t}} \geq \frac{\lambda}{2} d_n$.
\end{claim}
\begin{proof}
For a.e.\ $t \in \left]0,D\right[$, $$\partial(S_n)\rstr\{\pi \leq t\} = \partial(S_n)\rstr\{\pi =0\} + \partial(S_n)\rstr\{ 0 < \pi \leq t\}$$
$$ = f(\gamma_1) + H_t .$$
Where $H_t$ is the 1-current $\partial (S_n) \rstr\{ 0 < \pi \leq t\}$.
\\Since $\|\slice{S_n}{\pi}{t}\|$ is concentrated on $\pi^{-1}(\{t\})$,
$$\partial(S_n\rstr\{\pi \leq t\}) = \slice{S_n}{\pi}{t} + f(\gamma_1) + H_t.$$
Which means that $\slice{S_n}{\pi}{t} + f(\gamma_1) + H_t$ is a 1-current that is actually a cycle.
\\So $-(\slice{S_n}{\pi}{t} + H_t)$ is a 1-chain that fills the 0-cycle $\partial f(\gamma_1) = f(\partial \gamma_1) = [\![ f(b_n) ]\!] - [\![ f(a_n) ]\!]$.
\\Therefore
$$\mass{ -(\slice{S_n}{\pi}{t} + H_t) } \geq \textup{FillVol}^{Y,\mathrm{cr}}_{1}([\![ f(b_n) ]\!] - [\![ f(a_n) ]\!]).$$
\\Since, by the lemma \ref{filling current distance} , $\textup{FillVol}^{Y,\mathrm{cr}}_{1}([\![ f(b_n) ]\!] - [\![ f(a_n) ]\!]) \geq d_Y(f(a_n),f(b_n)) $, we have 
$$\mass{ -(\slice{S_n}{\pi}{t} + H_t) } \geq \lambda \, d_n.$$
So 
$$\mass{\slice{S_n}{\pi}{t}} + \mass{H_t} \geq \lambda \, d_n.$$
Note that
\begin{equation*} 
\begin{split}
H_t & =  \partial S_n\rstr\{ 0 < \pi \leq t\} \\
 & =   \big(f(\gamma_2)+f(\gamma_4)\big) \rstr\{ 0<\pi \leq t \},                   \\
\end{split}
\end{equation*}
because $0< \pi < D$. So
\begin{equation*} \label{eq2}
\begin{split}
H_t & =  f(\gamma_2)\rstr\{ 0<\pi \leq t \}   + f(\gamma_4)\rstr\{ 0<\pi \leq t \}   .               \\
\end{split}
\end{equation*}
By taking the mass
\begin{equation*} 
\begin{split}
\mass{H_t} & \leq  \mass{f(\gamma_2)\rstr\{ 0<\pi \leq t \}} + \mass{f(\gamma_4)\rstr\{ 0<\pi \leq t \}}  \\
& \leq \mass{f(\gamma_2)} + \mass{f(\gamma_4)}    \\
 & \leq \textup{Lip}(f) \big( \mass{\gamma_2} + \mass{\gamma_4}\big) .    \\
\end{split}
\end{equation*}
Since the $\gamma_i' \, s$ are Lipschitz chains, their mass is equal to their volume. 
\\So $\mass{\gamma_2} = \mass{\gamma_4} \leq \varphi_2(d_n)$, and $\mass{H_t} \leq 2\textup{Lip}(f) \varphi_2(d_n)$.
\\
\\Therefore, since $\varphi_2(d) \ll d$, there exists $N \in \mathbb{N}$ such that for all $n\geq N$ and for a.e.\ $t \in \left]0,D\right[$, $$\mass{\slice{S_n}{\pi}{t}} \geq \lambda d_n - 2 \textup{Lip}(f) \varphi_2(d_n) \geq  \frac{\lambda}{2} d_n. \qquad\qedhere $$
\end{proof}
Thus, by \eqref{equation 1},
\begin{equation*}
\begin{split}
    \phi \left( d_n \, \varphi_2(d_n)\right)  & \geq \int_0^{D} \mass{\slice{S_n}{\pi}{t}} \, dt\\ 
& \geq D \, \frac{\lambda}{2} d_n.
\end{split}
\end{equation*}
Which implies that
\begin{equation*}
 D \leq \frac{2}{\lambda}  \frac{\phi \left( d_n \, \varphi_2(d_n)\right)}{d_n}.
\end{equation*}
Let us denote $\psi(\varphi_2(d_n)) = \frac{2}{\lambda}  \frac{\phi \left( d_n \, \varphi_2(d_n)\right)}{d_n}$. Note that $\psi$ is sublinear. 
\\
\\Since $D=d_Y(f(\gamma_1),f(\gamma_3))$, the last inequality implies that there exists $z \in X$ such that $z \in \gamma_3$ and $d_Y(f(\gamma_1),f(z)) \leq \psi(\varphi_2(d_n))$.
\\ But $z \in \gamma_3 $ implies that $d_X(z,\gamma_1)=d_X(x_2,x_2') $.
\\
\\By doing this process for all $n\geq N$ and all $x_2' \in B_{X_2}(x_2,\varphi_2(d_n))$, we get subsets $C_n \subset X$ that projects onto $B_{X_2}(x_2,\varphi_2(d_n))$, i.e.\ 
\begin{equation*} \label{eq7}
    B_{X_2}(x_2,\varphi_2(d_n)) \subset \textup{proj}_{X_2}(C_n) ,
\end{equation*}
and such that
\begin{equation} \label{equation 2}
    f(C_n) \subset N_{\psi(\varphi_2(d_n))}   (f(\gamma_1)),
\end{equation} 
since $\forall z \in C_n$, $ d_Y(f(z),f(\gamma_1)) \leq \psi(\varphi_2(d_n))$.
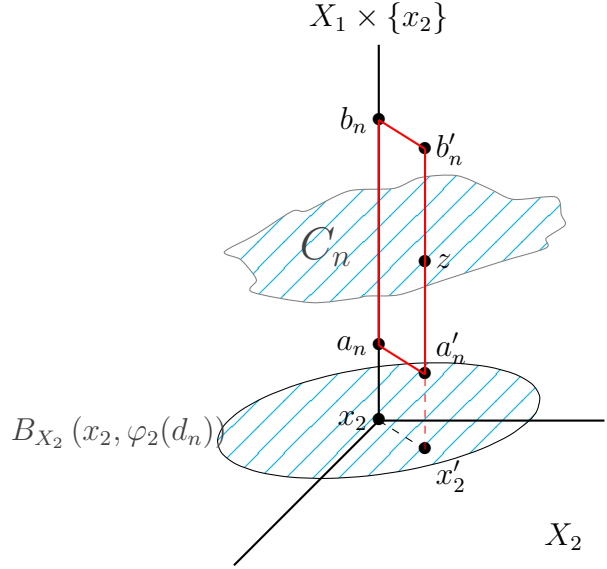
\begin{figure}[htbp]
  \centering
\begin{tikzpicture}
    \fill[pattern=my north east lines, pattern color=gray, draw=gray]
        plot[domain=0:350, smooth cycle] ({2*cos(\x)+rnd*0.5},2.5,{2*sin(\x)+rnd*0.5});

    \fill[pattern=my north east lines, pattern color=green, draw=black]
        plot[domain=0:350, smooth cycle] ({2*cos(\x)},0,{2*sin(\x)});

    \draw[thick] (0,0,0)--(0,0,5); 
    \draw[thick] (0,0,0)--(0,5,0);
    \draw[thick] (0,0,0)--(3,0,0);

    \coordinate (a_n) at (0,1,0);
    \node at (a_n) [left] {$a_n$};
    \node at (a_n) {\textbullet};

    \coordinate (a_n') at (1,1,1);
    \node at (a_n') [above right] {$a_n'$};
    \node at (a_n') {\textbullet};

    \coordinate (b_n) at (0,4,0);
    \node at (b_n) [left] {$b_n$};
    \node at (b_n) {\textbullet};

    \coordinate (b_n') at (1,4,1);
    \node at (b_n') [right] {$b_n'$};
    \node at (b_n') {\textbullet};

    \coordinate (x_2) at (0,0,0);
    \node at (x_2) [left] {$x_2$};
    \node at (x_2) {\textbullet};

    \coordinate (x_2') at (1,0,1);
    \node at (x_2') [below right] {$x_2'$};
    \node at (x_2') {\textbullet};

    \node[above] at (0,5,0) {$X_1 \times \{x_2\}$};

    \node at (4,0,4) {$X_2$};
    \node[color=gray!50!black] at (-2.3,1,3)  {$B_{X_2}\left(x_2,\varphi_2(d_n)\right)$};

    \node[right, color=gray!50!black] at (-1,2.5,0.5) {\Large{$C_n$}};

    \node[right] at (1,2.5,1) {$z$};
    \node at (1,2.5,1) {\textbullet};

    \draw[dashed] (x_2) -- (x_2');
    \draw[dashed, color=red] (x_2') -- (a_n');
    \begin{scope}[thick, decoration={markings, mark=at position 0.5 with }]

        \draw[postaction={decorate}, color=red] (a_n) -- (b_n);
        \draw[postaction={decorate}, color=red] (b_n) -- (b_n');
        \draw[postaction={decorate}, color=red] (b_n') -- (a_n');
        \draw[postaction={decorate}, color=red] (a_n') -- (a_n);
    \end{scope}
\end{tikzpicture}
  \caption{The subsets $C_n$ in $X$ that are highly compressed by $f$ }
  \label{fig:CubeNormal}
\end{figure}
\\The projection onto $X_2$ is 1-Lipschitz, so if we fix $\varepsilon >0$,
\begin{equation*} 
    \textup{Vol}_X^{\varepsilon}(C_n) \geq \textup{Vol}_X^{\varepsilon}(\textup{proj}_{X_2}(C_n)) \geq \textup{Vol}_X^{\varepsilon}\left(  B_{X_2}\left(x_2,\varphi_2(d_n)\right) \right).
\end{equation*}
$X_2$ has exponential growth, so there exists $\alpha >0$ such that $\forall R >0$, $\textup{Vol}_{X_2}^{\varepsilon}(  B_{X_2}(R)) \geq e^{\alpha R}$. So 
\begin{equation*} 
    \textup{Vol}_X^{\varepsilon}(C_n) \geq e^{\alpha \, \varphi_2(d_n)}.
\end{equation*}
While the equation \eqref{equation 2} implies that
\begin{equation*} 
    \textup{Vol}_Y^{\varepsilon}(f(C_n)) \leq \textup{Vol}_Y^{\varepsilon}( N_{\psi(\varphi_2(d_n))}   (f(\gamma_1)) ).
\end{equation*}
But, like we saw for the case $k=1$, $f$ coarsely preserves volumes, i.e.\ there exist $\delta, \delta' >0$ such that 
$$  \delta \, \textup{Vol}_X^{\varepsilon} ( C_n )  \leq      \textup{Vol}_Y^{\varepsilon} ( f( C_n ) ) \leq  \delta' \, \textup{Vol}_X^{\varepsilon} ( C_n )   .    $$
Also, by lemma \ref{volume neighb}
\begin{equation*}
    \textup{Vol}_Y^{\varepsilon}\left( N_{\psi(\varphi_2(d_n))}   (f(\gamma_1)) \right) \leq \beta_Y^{\varepsilon}\left( \varepsilon + \psi(\varphi_2(d_n) \right) \times \textup{Vol}_Y^{\varepsilon}(  f(\gamma_1) ).
\end{equation*}
$Y$ has at most exponential growth, so there exists $\beta >0$ such that for all $R>0$, 
\\$  \beta_Y^{\varepsilon}\left( R \right)  \leq      e^{\beta R}    $.
Therefore, we have on one hand
$$   \beta_Y^{\varepsilon}\left( \varepsilon + \psi(\varphi_2(d_n) \right) \leq \textup{exp} \left( \varepsilon + \psi(\varphi_2(d_n))   \right).  $$
On the other hand, by taking a partition of the geodesic segment $\gamma_1$ into sub-intervals of length $\varepsilon$, we get :
$$ \textup{Vol}_X^{\varepsilon} (\gamma_1)         \leq   \frac{ d_n }{\varepsilon} +1    \leq   \frac{ 2 d_n}{\varepsilon}.            $$
Hence
$$  \textup{Vol}_Y^{\varepsilon} (f(\gamma_1))    \leq \delta' \, \textup{Vol}_X^{\varepsilon} (\gamma_1)        \leq   \frac{ 2 \delta'}{\varepsilon} d_n .   $$           
So we have 
$$   \textup{Vol}_Y^{\varepsilon}\left( N_{\psi(\varphi_2(d_n))}   (f(\gamma_1)) \right) \leq \textup{exp}\left( \varepsilon + \psi(\varphi_2(d_n))   \right) \times  \frac{ 2 \delta'}{\varepsilon} d_n.$$
We conclude from the previous inequalities that 
\begin{equation*}
   \delta e^{\alpha \, \varphi_2(d_n)} \leq \delta \, \textup{Vol}_X^{\varepsilon} ( C_n ) \leq  \textup{Vol}_Y^{\varepsilon} ( f( C_n ) ) \leq \textup{exp}\left( \varepsilon + \psi(\varphi_2(d_n))   \right) \times  \frac{ 2 \delta'}{\varepsilon} d_n.
\end{equation*}
Which implies finally that for all $n\geq N$
\begin{equation*}
   \delta \, \textup{exp}\left(\alpha \, \varphi_2(d_n)\right) \leq \textup{exp}\left( \varepsilon + \psi(\varphi_2(d_n))   \right) \times  \frac{ 2 \delta'}{\varepsilon} d_n.
\end{equation*}
Which is not possible when $d_n \to \infty$ because $\psi$ is sublinear. This completes the proof for $k=2$.
\\
\\• Suppose now that the result holds for some $k \geq 2$, and let us prove it for $k+1$.
\\Suppose that there is a coarse embedding $f : X = X_1 \times \cdots \times X_{k+1}  \to Y$ that satisfies the second condition of the theorem, and let us fix some $x_{k+1} \in X_{k+1}$.
\\ The induction hypothesis implies that the restriction of $f$ to the copy $X_1 \times \dots \times X_{k} \times \{x_{k+1}\}$ of $X_1 \times \dots \times X_{k}$ does not satisfy the second condition of the theorem. Which means that there exists a sequence $(d_n)_n$ such that $d_n \to \infty$ and a sequence of  
$k$-dimensional rectangles $R_{d_n}$,whose sides are geodesic segments in the $X_i \,'$s, of lengths $l_1, \dots, l_k$, such that $l_1 = d_n, l_2 \leq \varphi_2(d_n),\dots,l_k \leq \varphi_k(d_n)$ and there exists $\lambda >0$ such that for all $n$ 
$$
\textup{FillVol}_{k}^{Y, \mathrm{cr}}(f(\partial R_{d_n})) \geq \lambda \, d_n \times \varphi_2(d_n) \times \dots \times \varphi_k(d_n).
$$
Let $n \in \mathbb{N}$ and let us take $x_{k+1}' \in X_{k+1}$ such that $d_{X_{k+1}}(x_{k+1},x_{k+1}') \leq \varphi_{k+1}(d_n)$, and consider the $(k+1)$-dimensional rectangle $R_{d_n}' = R_{d_n} \times [x_{k+1},x_{k+1}']  $.
\\Since $f$ satisfies the second condition of the theorem, there exists a sublinear function $\phi$ such that
$$
\textup{FillVol}_{k+1}^{Y, \mathrm{cr}}(f(\partial R_{d_n}')) \leq \phi (d_n \times \varphi_2(d_n) \times \dots \times \varphi_{k+1}(d_n)).
$$
This implies that there exists $V_n \in \mathbf{I}_{k+1}(Y)$ in $Y$ such that $\partial V_n = f(\partial R_{d_n}')$ and 
$$
\mass{V_n} \leq \phi \left(\prod_{i=1}^{k+1}  \varphi_i(d_n)\right).
$$
Now consider the $1$-Lipschitz map $\pi : Y \to \mathbb{R}$, $\pi(y) = d_Y\big(y,f( R_{d_n} \times \{ x_{k+1}  \}      )\big) $, i.e.\ the map that gives the distance from the image of the basis of the $(k+1)$-dimensional rectangle $R_{d_n}'$.
\\Again, by the Slicing Theorem, we have that for a.e.\ $t \in \mathbb{R}$, there exists $\langle V_n,\pi,t \rangle \, \in \mathbf{I}_k(Y)$ such that $\slice{V_n}{\pi}{t}= \partial(V_n\rstr\{\pi \leq t\}) - (\partial V_n)\rstr\{\pi\leq t\}$, and by integrating the co-area formula over the distance $t$, we have
$$ \int_0^{+\infty} \mass{\slice{V_n}{\pi}{t}} dt \leq \mass{V_n}.$$
Since $\mass{V_n} \leq \phi \left(\prod_{i=1}^{k+1}  \varphi_i(d_n) \right) $, we get
\begin{equation}\label{equation 3}
\int_0^{D} \mass{\slice{V_n}{\pi}{t}} dt \leq \int_0^{+\infty} \mass{\slice{V_n}{\pi}{t}} dt \leq \phi \left(\prod_{i=1}^{k+1}  \varphi_i(d_n) \right),
\end{equation}
where $D=d_Y \big(f( R_{d_n} \times \{ x_{k+1}  \}      ),f( R_{d_n} \times \{ x_{k+1}'  \}      ) \big)$.
\\
\\However, for a.e.\ $t \in \left]0,D\right[$, $\mass{\slice{V_n}{\pi}{t}}$ cannot be too small since $\slice{V_n}{\pi}{t}$ almost gives a filling of the $(k-1)$-cycle $ \partial f( R_{d_n} \times \{ x_{k+1}  \}      ) =  f( \partial( R_{d_n} \times \{ x_{k+1}  \} )     )$, that we will just denote by $f(\partial R_{d_n})$. We will also denote $ R_{d_n} \times \{ x_{k+1}  \}$ by $R_{d_n}$.
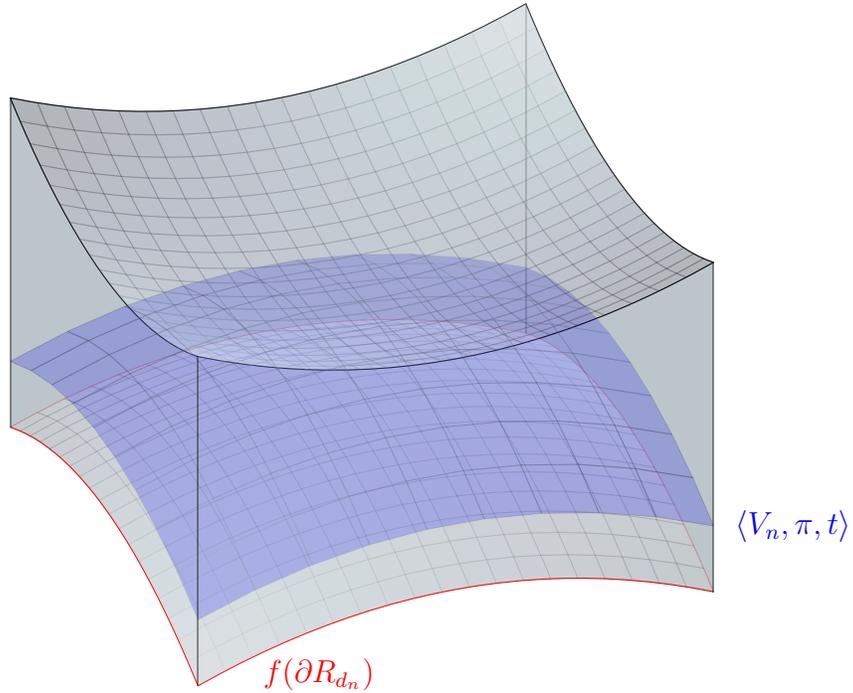
\begin{figure}[htbp]
  \centering
\begin{tikzpicture}[scale = 2]
  \begin{axis}[
    view={160}{50}, 
    hide axis, 
    xmin=-2,xmax=2, 
    ymin=-2,ymax=2, 
    zmin=-2,zmax=5, 
    enlargelimits=upper,
    colormap/bone, 
    trig format plots=rad,
    ]
  \addplot3[name path=A, domain=-1.5:1.5, variable=\u, samples y=1]
  ({1.5},{u},{3+(u^2+2.25)/2});
  \addplot3[name path=B, domain=-1.5:1.5, variable=\u, samples y=1]
  ({-1.5},{u},{3+(u^2+2.25)/2});
  \addplot3[name path=C, domain=-1.5:1.5, variable=\u, samples y=1]
  ({u},{-1.5},{3+(u^2+2.25)/2});
  \addplot3[name path=D, domain=-1.5:1.5, variable=\u, samples y=1]
  ({u},{1.5},{3+(u^2+2.25)/2});
  \addplot3[name path=AA, domain=-1.5:1.5, variable=\u, samples y=1,color=red]
  ({1.5},{u},{-(u^2+2.25)/2});
  \addplot3[name path=BB, domain=-1.5:1.5,opacity=0.2, variable=\u, samples y=1,color=red] 
  ({-1.5},{u},{-(u^2+2.25)/2});
  \addplot3[name path=CC, domain=-1.5:1.5,opacity=0.2, variable=\u, samples y=1,color=red]
  ({u},{-1.5},{-(u^2+2.25)/2});
  \addplot3[name path=DD, domain=-1.5:1.5, variable=\u, samples y=1,color=red]
  ({u},{1.5},{-(u^2+2.25)/2});
  \addplot3 [domain=-2.25:5.25, samples=20, opacity=0.2, 
  variable=\u, point meta=0]
  ({-1.5} , {-1.5},{u});
  \addplot3 [surf,domain=-1.5:1.5, domain y=-1.5:1.5,
  samples=20, samples y=20, opacity=0.1, 
  variable=\u, variable y=\v,
  point meta=u*v 
  ]
  ( {u}, {v}, {-(u^2+v^2)/2)} );
  \addplot3 [surf,domain=-1.5:1.5, domain y=-1.5:1.5,
  samples=10, samples y=10, opacity=0.5,
  variable=\u, variable y=\v,
  point meta=u*v,
  opacity=0.2,blue
  ]
  ( {u}, {v}, {1.5-(u^2+v^2)/2)} ); 
  \addplot3 [surf,domain=-1.5:1.5, domain y=-1.5:1.5,
  samples=20, samples y=20, opacity=0.2,
  variable=\u, variable y=\v,
  point meta=u*v
  ]
  ({u}, {v}, {3+(u^2+v^2)/2)} );
  \addplot3[opacity=0.2,cyan!30!black] fill between[of=A and AA];
  \addplot3[opacity=0.2,cyan!30!black] fill between[of=B and BB];
  \addplot3[opacity=0.2,cyan!30!black] fill between[of=C and CC];
  \addplot3[opacity=0.2,cyan!30!black] fill between[of=D and DD];
  \addplot3 [domain=-2.25:5.25, samples=20,  opacity=0.4,
  variable=\u, point meta=0]
  ({1.5} , {1.5},{u});
  \addplot3 [domain=-2.25:5.25, samples=20,  opacity=0.4,
  variable=\u, point meta=0]
  ({-1.5} , {1.5},{u});
  \addplot3 [domain=-2.25:5.25, samples=20,  opacity=0.4,
  variable=\u, point meta=0]
  ({1.5} , {-1.5},{u});
  
  \node[label={0:{\textcolor{blue}{$\slice{V_n}{\pi}{t}$}}},
  ] at (axis cs:-1.5,1.5,-0.75) {};
  \node[label={0:{\textcolor{red}{$f(\partial R_{d_n})$}}},
  ] at (axis cs:0.7,0,-5.5) {};
\end{axis}
\end{tikzpicture}
  \caption{A slice of $V_n$ at distance $t$ from $f(R_{d_n})$}
  \label{fig:CubeNormal}
\end{figure}
\\
\begin{claim}
For $n$ big enough and for a.e.\ $t \in \left]0,D\right[$, $$\mass{\slice{V_n}{\pi}{t}} \geq \frac{\lambda}{2} \prod_{i=1}^{k}  \varphi_i(d_n).$$
\end{claim}
\begin{proof}
For a.e.\ $t \in \left]0,D\right[$,  $$\partial(V_n) \rstr\{\pi \leq t\} = \partial(V_n) \rstr\{\pi =0\} + \partial(V_n) \rstr\{ 0 < \pi \leq t\}$$
$$ = f(R_{d_n}) + H_t. $$
Where $H_t$ is the $k$-current $\partial (V_n) \rstr\{ 0 < \pi \leq t\}$.
\\Since $\|\slice{V_n}{\pi}{t}\|$ is concentrated on $\pi^{-1}(\{t\})$,
$$\partial(V_n\rstr\{\pi \leq t\}) = \slice{V_n}{\pi}{t} + f(R_{d_n}) + H_t. $$
Which means that $\slice{V_n}{\pi}{t} + f(R_{d_n}) + H_t$ is actually a cycle.
\\So $-(\slice{V_n}{\pi}{t} + H_t)$ is a $k$-chain that fills the $(k-1)$-cycle $\partial f(R_{d_n}) = f(\partial R_{d_n})$. 
\\Therefore, by the induction hypothesis
$$\mass{ -(\slice{V_n}{\pi}{t} + H_t) } \geq \textup{FillVol}_{k}^{Y, \mathrm{cr}}(f(\partial R_{d_n}))\geq \lambda \, \prod_{i=1}^{k}  \varphi_i(d_n).$$
So 
$$\mass{\slice{V_n}{\pi}{t}} + \mass{H_t} \geq \lambda \, \prod_{i=1}^{k}  \varphi_i(d_n).$$
Note that
\begin{equation*} 
\begin{split}
H_t & =  \partial (V_n) \rstr\{ 0 < \pi \leq t\} 
     \\ & =\big( \sum_{F \in \Delta} f(F ) \big) \rstr\{ 0 <\pi \leq t \} ,
\end{split}
\end{equation*}
where $\Delta$ is the set of side faces of $ R_{d_n}'$, i.e.\ faces of $ R_{d_n}'$ except $R_{d_n} \times \{ x_{k+1}  \}$ and $R_{d_n} \times \{ x_{k+1}'  \}$, because $0< \pi < D$. So by taking the mass
\begin{equation*} 
\begin{split}
\mass{H_t} & = \mass{\big( \sum_{F \in \Delta} f(F ) \big) \rstr\{ 0 <\pi \leq t \} }  
\\ & \leq \sum_{F \in \Delta} \mass{ ( f(F) ) \rstr\{ 0<\pi \leq t \}}      \\
& \leq \sum_{F \in \Delta} \mass{  f(F) }    \\
& \leq \textup{Lip}(f) \sum_{F \in \Delta} \mass{ F }.
\end{split}
\end{equation*}
Since every side face $F$ satisfies $\mass{ F } \leq \frac{\prod_{i=1}^{k+1}  \varphi_i(d_n)}{\varphi_{s}(d_n)}$, where $s \in \{ 1,\dots,k\}$, so every $F \in \Delta$ satisfies $\mass{F} \leq \frac{\prod_{i=1}^{k+1}  \varphi_i(d_n)}{\varphi_{k}(d_n)} $. There are $2k$ side faces,
So
$$\mass{H_t} \leq   2 k \, \textup{Lip}(f) \prod_{\substack{i=1 \\ i \ne k}}^{k+1}  \varphi_i(d_n). $$
Hence
\begin{equation*}
    \begin{split}
        \mass{\slice{V_n}{\pi}{t}} &\geq \lambda \, \prod_{i=1}^{k}  \varphi_i(d_n) - \mass{H_t} 
        \\ &\geq \lambda \, \prod_{i=1}^{k}  \varphi_i(d_n) - 2 k \, \textup{Lip}(f) \prod_{\substack{i=1 \\ i \ne k}}^{k+1}  \varphi_i(d_n).  
    \end{split}
\end{equation*}
Since $\varphi_{k+1} \ll \varphi_k$, it implies that 
$$  \prod_{\substack{i=1 \\ i \ne k}}^{k+1}  \varphi_i(d_n) \ll    \prod_{i=1}^{k}  \varphi_i(d) .       $$
\\So by taking $n$ big enough we have 
$$\lambda \, \prod_{i=1}^{k}  \varphi_i(d_n) - 2 k \, \textup{Lip}(f) \prod_{\substack{i=1 \\ i \ne k}}^{k+1}  \varphi_i(d_n) \geq \frac{\lambda}{2} \prod_{i=1}^{k}  \varphi_i(d_n). $$
We conclude that for sufficiently big $n$ and for a.e.\ $t \in \left]0,D\right[$ $$\mass{\slice{V_n}{\pi}{t}} \geq \frac{\lambda}{2} \prod_{i=1}^{k}  \varphi_i(d_n). \qquad\qedhere $$
\end{proof}
Therefore, by \eqref{equation 3}
\begin{equation*}
\begin{split}
    \phi \left(\prod_{i=1}^{k+1}  \varphi_i(d_n) \right) & \geq \int_0^{D} \mass{\slice{V_n}{\pi}{t}} dt\\ 
& \geq D \,  \frac{\lambda}{2} \prod_{i=1}^{k}  \varphi_i(d_n).
\end{split}
\end{equation*}
This implies that
\begin{equation*}
  \frac{2}{\lambda}  \frac{\phi \left(\prod_{i=1}^{k+1}  \varphi_i(d_n) \right)}{\prod_{i=1}^{k}  \varphi_i(d_n)} \geq D.
\end{equation*}
\hfill
\\Denote $\psi(\varphi_{k+1}(d_n)) = \frac{2}{\lambda}  \frac{\phi \left(\prod_{i=1}^{k+1}  \varphi_i(d_n) \right)}{\prod_{i=1}^{k}  \varphi_i(d_n)}$. Note that $\psi$ is sublinear : $\psi(\varphi_{k+1}(d_n)) = o(\varphi_{k+1}(d_n))$ because $\phi$ is sublinear.
\\
\\Since $D=d_Y(f( R_{d_n} \times \{ x_{k+1}  \}      ),f( R_{d_n} \times \{ x_{k+1}'  \}      ))$, the last inequality implies that $\exists z \in X$ such that $z \in R_{d_n} \times \{ x_{k+1}'  \} $ and $d_Y(f(R_{d_n} \times \{ x_{k+1}  \} ),f(z)) \leq \psi(\varphi_{k+1}(d_n))$.
\\But $z \in R_{d_n} \times \{ x_{k+1}'  \} $ implies that $d_X(z,R_{d_n} \times \{ x_{k+1}  \})=d_{X_{k+1}}(x_{k+1},x_{k+1}') $.
\\
\\By doing this process for all $x_{k+1}' \in B_{X_{k+1}}\big(x_{k+1},\varphi_{k+1}(d_n)\big)$, we get a subset $C_n \subset X$ that projects onto $B_{X_{k+1}}\big(x_{k+1},\varphi_{k+1}(d_n)\big)$, i.e.\ 
\begin{equation}\label{equation 4}
   B_{X_{k+1}}\big(x_{k+1},\varphi_{k+1}(d_n)\big) \subset \textup{proj}_{X_{k+1}}(C_n), 
\end{equation}
and such that
\begin{equation} \label{equation 5}
    f(C_n) \subset N_{\psi(\varphi_{k+1}(d_n))}   \big(f(R_{d_n} \times \{ x_{k+1}  \}  )\big),
\end{equation} 
since $\forall z \in C_n$, $d_Y(f(z) , f(R_{d_n} \times \{ x_{k+1}  \} )) \leq \psi(\varphi_{k+1}(d_n))$.
\\The projection onto $X_{k+1}$ is 1-Lipschitz, so \eqref{equation 4} implies that, if we fix $\varepsilon >0$,
\begin{equation*} 
    \textup{Vol}_X^{\varepsilon}(C_n) \geq \textup{Vol}_X^{\varepsilon}(\textup{proj}_{X_{k+1}}(C_n)) \geq \textup{Vol}_X^{\varepsilon}( B_{X_{k+1}}(x_{k+1},\varphi_{k+1}(d_n)) ).
\end{equation*}
$X_{k+1}$ has exponential growth, so  $\exists \alpha >0$ such that $\forall R >0$, 
$\textup{Vol}_{X_{k+1}}^{\varepsilon}(  B_{X_{k+1}}(R)) \geq e^{\alpha R}.$ So 
\begin{equation*} 
    \textup{Vol}_X^{\varepsilon}(C_n) \geq \textup{exp}(\alpha \, \varphi_{k+1}(d_n)).
\end{equation*}
While \eqref{equation 5} implies that
\begin{equation*}
    \textup{Vol}_Y^{\varepsilon}(f(C_n)) \leq \textup{Vol}_Y^{\varepsilon}\left(N_{\psi \left(\varphi_{k+1}(d_n)\right)}   \big(f(R_{d_n} \times \{ x_{k+1}  \}  )\big) \right).
\end{equation*}
Since $f$ coarsely preserves volumes, there exist $\delta, \delta' >0$ such that 
$$  \delta \, \textup{Vol}_X^{\varepsilon} ( C_n )  \leq      \textup{Vol}_Y^{\varepsilon} ( f( C_n ) ) \leq   \delta' \, \textup{Vol}_X^{\varepsilon} ( C_n )    . $$
By lemma \ref{volume neighb}, we have
\begin{equation*}
    \textup{Vol}_Y^{\varepsilon}\left(N_{\psi \left(\varphi_{k+1}(d_n)\right)}   \big(f(R_{d_n} \times \{ x_{k+1}  \}  )\big) \right) \leq \beta_Y^{\varepsilon}\left( \varepsilon + \psi(\varphi_{k+1}(d_n) \right) \times \{ x_{k+1}  \}  ) ).
\end{equation*}
$Y$ has at most exponential growth, so there exists $\beta >0$ such that for all $ R>0$,
$$ \beta_Y^{\varepsilon}\left( R \right)  \leq      e^{\beta R} .   $$
In particular, we have on hand that
$$\beta_Y^{\varepsilon}\left( \varepsilon + \psi(\varphi_{k+1}(d_n) \right)  \leq   \textup{exp}\left(  \beta ( \varepsilon +    \psi ( \varphi_{k+1}(d_n))        )    \right)                        \leq    \textup{exp} \left( 2 \beta      \psi ( \varphi_{k+1}(d_n))            \right) . $$ 
On the other hand, by taking a partition of each side vector of $R_{d_n}$ into sub-intervals of length $\varepsilon$, we get :
$$ \textup{Vol}_X^{\varepsilon} \big( R_{d_n} \times \{ x_{k+1}  \}   \big)         \leq \prod_{i=1}^{k} \left( \frac{ \varphi_{i}(d_n)}{\varepsilon} +1    \right) \leq  \prod_{i=1}^{k} \left( \frac{2 \varphi_{i}(d_n)}{\varepsilon}     \right)  \leq \big(2/\varepsilon \big)^{k} \prod_{i=1}^{k}  \varphi_i(d_n)   .              $$
So, by denoting $A =  \big(2/\varepsilon \big)^{k} $, we have
\begin{equation*}
    \begin{split}
       \textup{Vol}_Y^{\varepsilon}(  f(R_{d_n} \times \{ x_{k+1}  \}  ) )
         &  \leq \delta' \, \textup{Vol}_X^{\varepsilon} \big( R_{d_n} \times \{ x_{k+1}  \}  \big)   
         \\ &\leq \delta' A \, \prod_{i=1}^{k}  \varphi_i(d_n).
    \end{split}
\end{equation*}
Therefore
\begin{equation*}
    \textup{Vol}_Y^{\varepsilon}\left(N_{\psi\left(\varphi_{k+1}(d_n) \right)}   (f(R_{d_n} \times \{ x_{k+1}  \}  )) \right)  \leq \textup{exp}( 2 \beta      \psi ( \varphi_{k+1}(d_n)  ) \times \delta' A \,  \prod_{i=1}^{k}  \varphi_i(d_n).
\end{equation*}
We conclude from all the previous inequalities that 
\begin{equation*}
  \delta \, \textup{exp}(\alpha \, \varphi_{k+1}(d_n)) \leq \delta \, \textup{Vol}_X^{\varepsilon} ( C_n ) \leq  \textup{Vol}_Y^{\varepsilon} ( f( C_n ) ) \leq \textup{exp}( 2 \beta      \psi ( \varphi_{k+1}(d_n)  ) \times \delta' A \,  \prod_{i=1}^{k}  \varphi_i(d_n).
\end{equation*}
Which implies finally that 
\begin{equation*}
  \delta \, \textup{exp}(\alpha \, \varphi_{k+1}(d_n)) \leq \textup{exp}( 2 \beta      \psi ( \varphi_{k+1}(d_n))   ) \times \delta' A \, \prod_{i=1}^{k}  \varphi_i(d_n) .
\end{equation*}
Which is not possible when $d_n \to \infty$ because $\psi$ is sublinear. This completes the proof.
\end{proof}
Let us show that Theorem \ref{Thm 1 v2} follows from Theorem \ref{CE of product}. 
\\Let $X= X_1 \times \dots \times X_k$ be a product of geodesic metric spaces with exponential growth, and let $Y$ be a Lipschitz-connected complete metric space with at most exponential growth such that $\textup{FV}_{k}^{Y}(\ell) = o\left(\ell^{\frac{k}{k-1}}\right)$. Let us show that there exists a sequence $(\varphi_i)_i$ as in Theorem \ref{CE of product} such that for every $k$-dimensional rectangle $R_{d}$, whose sides are geodesic segments in the $X_i \,'s$, with side lengths $l_1, \dots, l_k$ that satisfy $l_1 = \varphi_1(d), l_2 \leq \varphi_2(d),\dots,l_k \leq \varphi_k(d)$, we have
$$
\textup{FillVol}_{k}^{Y, \mathrm{cr}}\left(f(\partial R_{d})\right) \ll  \varphi_1(d) \times \dots\times\varphi_k(d).
$$
Actually, we do not need an infinite sequence $(\varphi_i)_{i \in \mathbb{N}^*}$ but only $\varphi_1, \dots , \varphi_k$ that satisfy the same conditions, because we proceed by induction from $p=1$ to $p=k$.  
\\
\\ $\textup{FV}_{k}^{Y}(\ell) = o\left(\ell^{\frac{k}{k-1}}\right)$ implies that there exists a function $a : \mathbb{R}_+ \to \mathbb{R}_+$ such that $a=o(1)$, and $\textup{FV}_{k}^{Y}(\ell) \leq  \ell^{\frac{k}{k-1}} a(\ell)$ for all $\ell>0$. Without loss of generality, we can assume that the function $a$ is slowly decreasing: it is deceasing and that $a(d) \geq \frac{1}{d} $ for all $d>0$.
\\Let us consider the following sequence: for all $i \in \mathbb{N}^*$, $\alpha_i = 1- \frac{1}{i}$. Let us denote its partial sums by $S_n = \sum_{i=1}^{n}\alpha_i = n - H_n$, where $H_n$ is the partial sum of the harmonic series. 
\\
\\We define the sequence $(\varphi_i)_{i \in \mathbb{N}^*}$ as follows: for all $i\in \mathbb{N}^*$, and for all $d>0$, $\varphi_i(d) = a(d)^{\alpha_i}d$.
\\Therefore $\varphi_1(d) = d$ for all $d>0$. It is clear that $\varphi_i \gg \varphi_{i+1}$ for all $i\in \mathbb{N}^*$. We also have that $\varphi_i(d)$ tends to $+ \infty$ for all $i\in \mathbb{N}^*$. Indeed, for all $d>0$, $a(d) \geq \frac{1}{d}$, so $a(d)^{\alpha_i}\geq \frac{1}{d^{\alpha_i}}$ . Therefore $\varphi_i(d) \geq d^{\frac{1}{i}}$ that tends to $+\infty$.
\\
\\Let us now show the inequality.
\begin{equation*}
\begin{split}
    \textup{FillVol}_{k}^{Y, \mathrm{cr}}\left(f(\partial R_{d})\right) &\leq \textup{FV}_{k}^{Y}\left( \textup{Vol}_{k-1}^{Y}  \left( f(\partial R_{d})  \right)      \right)
    \\&\leq \textup{FV}_{k}^{Y}\left( 2k \textup{Lip}(f) l_1 \times \dots \times l_{k-1}       \right)
    .
\end{split}
\end{equation*}
Without loss of generality, we can assume that $\textup{Lip}(f)= \frac{1}{2k}$. So
\begin{equation*}
\begin{split}
    \textup{FillVol}_{k}^{Y, \mathrm{cr}}\left(f(\partial R_{d})\right) &\leq \textup{FV}_{k}^{Y}\left( \varphi_1(d) \times \dots \times \varphi_{k-1}(d)     \right)
    \\&\leq \left( \varphi_1(d) \times \dots \times \varphi_{k-1}(d)  \right)^{\frac{k}{k-1}} a\left( \varphi_1(d) \times \dots \times \varphi_{k-1}(d)   \right).
\end{split}
\end{equation*}
\\ Note that $\varphi_1(d) \times \dots \times \varphi_{k-1}(d) = d^{k-1} a(d)^{S_{k-1}}$. So, we have on one hand
$$     \textup{FillVol}_{k}^{Y, \mathrm{cr}}\left(f(\partial R_{d})\right) \leq  d^k a(d)^{\frac{k}{k-1}S_{k-1}} a\left(  d^{k-1}a(d)^{S_{k-1}}  \right).             
$$
On the other hand
$$          \varphi_1(d) \times \dots \times \varphi_{k}(d) = d^{k} a(d)^{S_{k}}.    
$$
\\Therefore, to get the desired inequality, let us show that
$$ a\left(  d^{k-1}a(d)^{S_{k-1}}  \right) \ll a(d)^{S_{k}-\frac{k}{k-1}S_{k-1}} .
$$
\\ • Let us first consider the left hand side. For all $d>0$, $a(d)\geq \frac{1}{d}$, so $d^{k-1}a(d)^{S_{k-1}} \geq d^{k-1-S_{k-1}} = d^{H_{k-1}}$. For $k \geq 2$, $H_{k-1} \geq 1$. So $d^{k-1}a(d)^{S_{k-1}} \geq d$. Since $a$ is decreasing, we get for all $k\geq 2$ and for all $d>0$
$$          a\left( d^{k-1}a(d)^{S_{k-1}} \right) \leq a(d) .
$$
\\ • Let us now consider the right hand side. We have
\begin{equation*}
    \begin{split}
        S_{k}-\frac{k}{k-1}S_{k-1} &= k - H_k - \frac{k}{k-1}( (k-1)-H_{k-1}  )
        \\&= -H_k +\frac{k}{k-1}(H_k - \frac{1}{k})
        \\&= -H_k + \frac{k}{k-1}H_k - \frac{1}{k-1}
        \\&= \frac{1}{k-1}(H_k - 1).
    \end{split}
\end{equation*}
\\ For all $k\geq 2$, $H_k < k$. So $\frac{1}{k-1}(H_k - 1) < 1$. Therefore, for all $k \geq 2$, $S_{k}-\frac{k}{k-1}S_{k-1} <1$. This implies that 
$$ a(d) \ll  a(d)^{S_{k}-\frac{k}{k-1}S_{k-1}}.
$$
\\We conclude that 
$$ a\left(  d^{k-1}a(d)^{S_{k-1}}  \right) \ll a(d)^{S_{k}-\frac{k}{k-1}S_{k-1}} .
$$
This concludes the proof of the theorem.
\section{The domain is a model space with no Euclidean factor}
\subsection{Preliminary result on Weyl sectors}
The goal of this subsection is to show a structure result for Weyl sectors in a model space. In both symmetric spaces and Euclidean buildings, singular hyperplanes in a maximal flat give rise to a root system. Let us first define the Weyl sectors associated to a root system.
\begin{defn}
Let $\Phi$ be a root system in a Euclidean space $\left( E, \langle,\rangle  \right)$ (see \cite{bourbaki81} Chap.6), and let $\Delta$ be a basis of $\Phi$. Consider the hyperplane perpendicular to each root in $\Delta$. The complement of this finite set of hyperplanes is disconnected, and each connected component is called a \textit{Weyl sector}. Equivalently, we may define Weyl sectors as the equivalence classes of the following relation: 
\begin{center}
   $ u\sim v $ \, if \,  $\langle \alpha ,u \rangle \cdot \langle \alpha ,v \rangle > 0$ \, \, for every $\alpha \in \Delta$.
\end{center}
\end{defn}
\begin{rem}
They are also called Weyl chambers in the literature. We prefer Weyl sectors to avoid any confusion with chambers in Euclidean buildings.
\end{rem}
Let us now define the root system associated to a maximal flat in a model space.
\\
\\• Let $X$ be a Euclidean building of rank $p$. Let $x_0 \in X$ be a vertex, and let $A$ be an apartment containing it. Let $\mathcal{H}$ denote the set of walls in $A$ that contains $x_0$. Recall that $\mathcal{H}$ is stable under reflections by walls in $\mathcal{H}$. So if we take for each wall in $\mathcal{H}$ an orthogonal vector with a suitable size, we get a root system. The \textit{Weyl sectors with tip at $x_0$ of the apartment $A$} are defined as the Weyl sectors of this root system. They are also the connected components of $A - \cup_{h \in \mathcal{H}} h $.
\\
\\• Let $M$ is a symmetric space of non-compact type of rank $p$. A reference for what follows is \cite{eberlein1996geometry}, Chap.2. 
\\Let $x_0 \in M$ and let $F$ be a a maximal flat containing $x_0$. Let $G= \textup{Isom}_0(M)$ be the identity component of the isometry group of $M$, $\mathfrak{g}$ its Lie algebra and $K = \textup{Stab}_G(x_0)$. Let $\mathfrak{g} = \mathfrak{k} \oplus \mathfrak{p} $ be the Cartan decomposition determined by $x_0$, and $\theta = \theta_{x_0}$ be the Cartan involution of $\mathfrak{g}$ induced by $x_0$. \\Consider $\mathfrak{a} \subset \mathfrak{p} $ a maximal abelian subspace such that $F = \textup{exp}(\mathfrak{a}).x_0$, and the root space decomposition of $\mathfrak{g}$ determined by $\mathfrak{a}$ $$\mathfrak{g} = \mathfrak{g_0} \oplus \sum_{\alpha \in \Lambda} \mathfrak{g_{\alpha}}, $$
where $\Lambda \subset \mathfrak{a}^* $ is a finite subset called the restricted root system determined by $\mathfrak{a}$, and 
$$\mathfrak{g}_{\alpha} = \left\{ X \in \mathfrak{g} \mid [A,X] = \alpha(A)X,\ \forall A \in \mathfrak{a}          \right\} .$$
The bilinear form $\langle X,Y \rangle := -B(\theta(X),Y)$, where $B$ is the Killing form of $\mathfrak{g}$, defines a positive definite inner product on $\mathfrak{g}$. This inner product induces an isomorphism between $\mathfrak{a}$ and $\mathfrak{a}^*$, so for all $\alpha \in \Lambda$, let us denote $h_{\alpha}$ the unique vector in $\mathfrak{a}$ such that we have 
$$   \alpha(A) = \langle  h_{\alpha}, A \rangle , \qquad \forall A \in \mathfrak{a} .              $$
A \textit{wall} in $F$, or equivalently in  $\mathfrak{a}$, is a hyperplane of the form $\textup{ker}(\alpha)$. Note that we can identify $\mathfrak{a}$ with $F$ by the exponential map composed with the orbital map at $x_0$, which is an isometry when restricted to $\mathfrak{a}$. For every $\alpha \in \Lambda$, consider the reflection $S_\alpha$ in the hyperplane $\mathfrak{a}_\alpha = \textup{ker}(\alpha) $, that is orthogonal to $h_\alpha$. In particular, we have $S_\alpha (A) = A - 2 \frac{\langle  h_{\alpha}, A \rangle}{\langle  h_{\alpha}, h_{\alpha} \rangle} h_\alpha $, $\forall A \in \mathfrak{a}$.
\\The collection $\{h_\alpha, \alpha \in \Lambda       \}$ form a root system. We define the \textit{Weyl sectors with tip at $x_0$ of the maximal flat $F$} as the Weyl sectors of this root system. They are also the connected components of $\mathfrak{a} \, \textbackslash\cup_{\alpha \in \Lambda} \textup{ker}(\alpha) $.
\begin{prop} \label{Weyl cone acute}
Let X be a model space. Then its Weyl sectors are acute open simplicial cones generated by maximally singular vectors. In other words, if $x_0 \in X$ and $F$ is a maximal flat containing $x_0$, and if $C$ is a Weyl sector of $F$ with tip at $x_0$, then after identifying $x_0$ with the origin of $F$:
\\• There exist $u_1,\dots,u_p$ maximally singular vectors in $\overline{C}$ such that $\forall x \in C$, $\exists \, x_1,\dots,x_p \geq 0 $ such that $x = \sum_{i=1}^{p} x_i u_i  $, where $p = \textup{rank }X$.
\\ • For all $ x,y \in C, \quad \langle x,y \rangle > 0 $.
\end{prop}
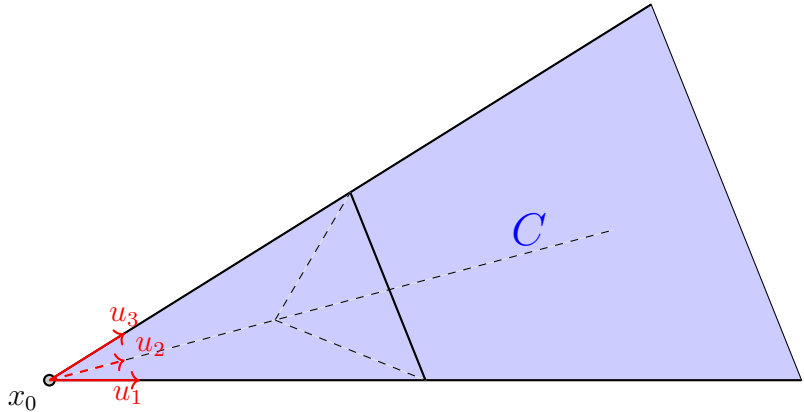
\begin{figure}[htbp]
  \centering
\begin{tikzpicture}

    \coordinate (a) at (4,2.5);
    \coordinate (a") at (1,0.625);
    \coordinate (a') at (8,5);
    
    \coordinate (b) at (3,.8);
    \coordinate (b') at (7.5,2);
    \coordinate (b") at (1,.266);

    \coordinate (c) at (5,0);
    \coordinate (c') at (10,0);
    \coordinate (c") at (1.2,0);
    
    \coordinate (e) at (0,0);

   \draw[fill=blue!20] (a') -- (c') -- (e) -- cycle;

    \draw[thick, fill=black!10] (e) -- (c');
    \draw[dashed, fill=black!10] (e) -- (b');
    \draw[thick, fill=black!10] (e) -- (a');

    \draw[dashed, fill=black!10] (b) -- (a);
    \draw[dashed, fill=black!10] (b) -- (c);
    
    \draw[thick, fill=black!50] (a) -- (c);

    \fill[black!20, draw=black, thick] (a) circle (0pt) node[black, above right] {};
    \fill[black!20, draw=black, thick] (b) circle (0pt) node[black, above left] {};
    \fill[black!20, draw=black, thick] (c) circle (0pt) node[black, below right] {};
    \fill[black!20, draw=black, thick] (e) circle (2pt) node[black, below left] {$x_0$};

    \begin{scope}[thick, decoration={markings, mark=at position 1 with {\arrow{>}}}]

    \draw[postaction={decorate}, color=red] (e) -- (a");
    \draw[dashed,postaction={decorate}, color=red] (e) -- (b");
    \draw[postaction={decorate}, color=red] (e) -- (c");
\end{scope}

\node[right, color=blue] at (6,2) {\Large{$C$}};

\node[right,color=red] at (0.7,-0.2) {$u_1$};
\node[right,color=red] at (1,0.45) {$u_2$};
\node[right,color=red] at (0.65,0.85) {$u_3$};

\end{tikzpicture}
  \caption{A Weyl sector $C$ generated by the maximal singular vectors $u_1,u_2,u_3$.}
  \label{fig:CubeNormal}
\end{figure}
\begin{proof}
Since maximal flats in a product space are products of maximal flats in the factors, Weyl sectors of a product space are products of Weyl sectors of the factors. Therefore, it is enough to show the result for a irreducible space, i.e.\ for a general root system in a Euclidean space.
\\• Let $\Phi$ be a root system in a Euclidean space $E = \mathbb{R}^p$. Let $W_r$ be its Weyl group, i.e.\ the group generated by reflections through the walls associated to the roots of $\Phi$.
\\Let $C$ be a Weyl sector with tip at the origin. By \cite{bourbaki81} Theorem 2 p.153 and section 6 p.64, there exist $\alpha_1,...,\alpha_p$ a basis of $\Phi^*$ (equivalently there exist $h_1,\dots,h_p$ a basis of $\Phi$) such that $C$ is an open simplicial cone, whose walls are exactly the hyperplanes $\{ \alpha_i = 0     \}$ for $i = 1 \dots p$ (i.e.\ the subsets $\{ x \in \mathbb{R}^p , \langle x,h_i \rangle = 0       \}$ for $i = 1 \dots p$). Which means that $C = \{ x= \sum_{i=1}^p x_i e_i \mid x_i \geq 0    \}$, where 
$$e_i \in   \bigcap_{\substack{k=1 \\ k \ne i}}^p \textup{ker} (\alpha_k).  $$
Since $\{ \alpha_1, \dots, \alpha_p  \}$ form a basis of $\Phi^*$, they form a basis of $\left( \mathbb{R}^p \right)^*$, and the intersection is one dimensional. So the $e_i$'s are maximally singular vectors.
\\
\\• Now let us show that $C$ is acute. Again by \cite{bourbaki81} Theorem 2 p.153, we can choose the basis $\{h_1,\dots,h_p\}$ in a unique way, up to permutation, such that the walls of $C$ are again $\textup{Ker}(\alpha_i)$ and such that $h_i$ and $C$ are both in the same side of the wall $\textup{Ker}(\alpha_i)$. In that case, we have 
$$ C = \cap_{i=1}^p ( \alpha_i > 0 ) =  \{ x \in E \mid \forall i \in \{ 1,\dots,p \}  , \quad  \langle h_i,x \rangle >0          \} .            $$
\\Since $\{h_1,\dots ,h_p\}$ is a basis of the root system $\Lambda$, then $\langle h_i , h_j \rangle \leq 0$ for all $ i \ne j$.
\\Therefore if $x,y \in C$, i.e.\ for every $ i \in \{ 1,\dots,p \}$, $\langle x , h_i \rangle >0$ and $\langle y , h_i \rangle >0$ , then by \cite{bourbaki81} p.79 lemma 6, we have $\langle x , y \rangle >0$. So $C$ is acute.
\end{proof}
\begin{cor} \label{corollary of acute}
Let $C$ be a Weyl sector in a model space of rank $p$, and $u_1,...,u_p \in \overline{C}$ maximally singular vectors generating it. Then if $x = \sum_{i=1}^{p} x_i u_i  \in C $, we have $\forall J \subset \{ 1,\dots,p \}$ $$\| x \| ^2 \geq \sum_{j \in J} \| x_j u_j \|^2.  $$
\end{cor}
\begin{proof}
Since $u_1,...,u_p \in \overline{C}$, by the proposition and the continuity of the inner product, we have $\forall i,j \in \{ 1,...,p \} $ $ \langle u_i, u_j \rangle \geq 0$.
\\If $x = \sum_{i=1}^{p} x_i u_i \in C $, then 
$$     \| x \| ^2 = \sum_{i=1}^p \| x_i u_i \|^2         + 2  \sum_{i,j=1}^p  x_i x_j \langle u_i,u_j \rangle   \geq  \sum_{i=1}^p \| x_i u_i \|^2,             $$
since for any $1\leq i \leq p$, $x_i \geq 0$.
\end{proof}
\subsection{Parallelepipeds and Parallelograms}
In the proof of Theorem \ref{Thm 1 v1}, we considered rectangles in the product space. However, in the case of irreducible symmetric space of non-compact type or Euclidean buildings, the space does not factorize and we can no longer consider rectangles. More precisely, maximally singular geodesics are not necessarily orthogonal in a maximal flat. Therefore, we will have to work with cycles that are higher dimensional parallelograms in a maximal flat. We will define them as boundaries of higher dimensional parallelepipeds by induction, as sum of simplices.
\\
\\Let $E$ be a Euclidean space and $x_1, \dots, x_n \in E$. We denote by $[x_1 \dots x_n] $ the $(n-1)$-simplex which is the convex hull of these points and the orientation is given by the order in which the vertices are listed. Its boundary $\partial [x_1 \dots x_n]$ is 
$$ \partial [x_1 \dots x_n] = \sum_{k=1}^n (-1)^{k+1}  [x_1, \dots, \hat{x_k}, \dots, x_n]   ,             $$ 
where $[x_1, \dots, \hat{x_k}, \dots, x_n]$ is the $(n-2)$-simplex given by omitting $x_k$.
\\Let us define the \textit{$0$-dimensional parallelepiped} in $x$, $C^0(x)$, as the $0$-simplex $[x]$.
\\A \textit{1-dimensional parallelepiped} is defined from a point $x$ and a vector $u \in E$ by
$$ C^1(x;u) := [x,x+u]     .       $$
Suppose that we defined $C^{n-1}(x;u_1, \dots, u_{n-1})$ as a sum of $(n-1)$-simplices. To define the $n$-dimensional parallelepiped $C^n(x;u_1,\dots,u_n)$, we use its standard triangulation: $x$ has $n$ hyper-faces opposite to it. Each one of them is an $(n-1)$-dimensional parallelepiped written as $C^{n-1}(x+u_k;u_1, \dots,\hat{u_k},\dots u_{n})$, for $k \in \{ 1, \dots, n   \}$, which is a sum of $(n-1)$-simplices. We take its convex hull with $x$, that we can write as $[x; C^{n-1}(x+u_k;u_1, \dots,\hat{u_k},\dots u_{n}) ]$, which is just the sum of the $n$-simplices that we get by taking the convex hull of $x$ with the $(n-1)$-simplices in $C^{n-1}(x+u_k;u_1, \dots,\hat{u_k},\dots u_{n})$.
\\
\\We define the \textit{$n$-dimensional parallelepiped} $C^n(x;u_1,\dots,u_n)$ as the alternating sum of these terms: 
$$  C^n(x;u_1,\dots,u_n) :=  \sum_{k=1}^n (-1)^{k+1}  [x, C^{n-1}(x+u_k;u_1, \dots,\hat{u_k},\dots u_{n})]   .                                $$
We define the \textit{$(n-1)$-dimensional parallelogram} as its boundary $$P^{n-1}(x;u_1,\dots,u_n) := \partial C^n(x;u_1,\dots,u_n).$$
\begin{figure}[htbp]
  \centering
\begin{tikzpicture}[scale=1.62]
    \tkzInit[xmin=0,xmax=6,ymin=0,ymax=3] 
    \tkzClip[space=.5] 
    \tkzDefPoint(0,0){x} 
    \tkzDefPoint(4,0){x+u_1} 
    \tkzDefPoint(1,3){x+u_2} 
    \tkzDefPoint(5,3){x+u_1+u_2} 
    
    \coordinate (x') at (0.1,0);

    \coordinate (u_3') at (5.1,3);

\node[left, color=red] at (4.1,0.4) {\large{$[x,x+u_1,x+u_1+u_2]$}};
\node[right, color=blue] at (0.95,2.67) {\large{$[x,x+u_1+u_2,x+u_2]$}};

    \tkzLabelPoints(x,x+u_1)
    \tkzLabelPoints[above right](x+u_1+u_2,x+u_2)
    
    \begin{scope}[thick, decoration={markings, mark=at position 0.5 with {\arrow{>}}}]

    \draw[postaction={decorate}, color=red] (x') -- (x+u_1);
        \draw[postaction={decorate}, color=red] (x+u_1) -- (u_3');
        \draw[postaction={decorate}, color=red] (u_3') -- (x');

      \draw[postaction={decorate}, color=blue] (x) -- (x+u_1+u_2);
      \draw[postaction={decorate}, color=blue] (x+u_1+u_2) -- (x+u_2);
       \draw[postaction={decorate}, color=blue] (x+u_2) -- (x);

\end{scope}

    \fill[black, draw=black, thick] (x+u_1) circle (2pt) node[black, above right]{};
    \fill[black, draw=black, thick] (0.05,0) circle (2pt) node[black, above right]{};
    \fill[black, draw=black, thick] (5.05,3) circle (2pt) node[black, above right]{};
    \fill[black, draw=black, thick] (x+u_2) circle (2pt) node[black, above right]{};

\end{tikzpicture}
  \caption{$P^1(x;u_1,u_2)$ as sum of two simplices}
  \label{fig:CubeNormal}
\end{figure}
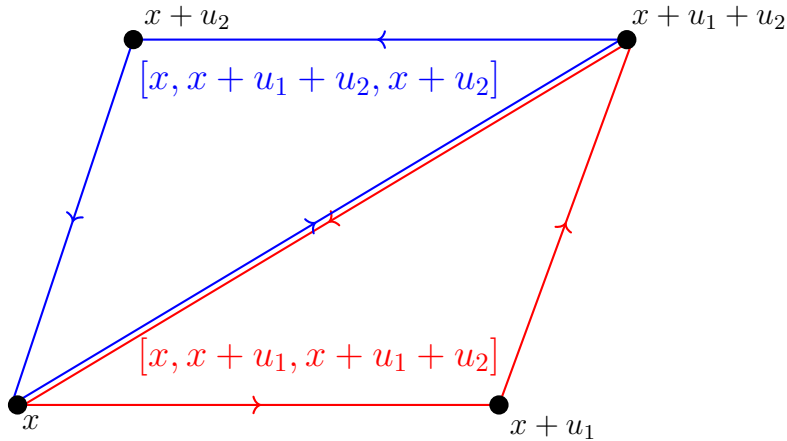
\begin{lem}\label{boundary or parallelepiped}
Let $y,u_1,\dots,u_n \in E$. 
\\• We have
$$ C^n(y;u_1,\dots, u_k ,\dots,u_n) = -C^n(y+u_k;u_1,\dots, -u_k ,\dots,u_n), $$
therefore
$$ C^n(y;u_1,\dots,u_n) = (-1)^n C^n(y+\sum_{k=1}^n u_k;-u_1,\dots,-u_n). $$
• For all $x \in E$,
$$ \partial [x,C^n(y;u_1,\dots,u_{n-1})] =    C^n(y;u_1,\dots,u_{n-1}) - [x,\partial C^n(y;u_1,\dots,u_{n-1}) ]                 $$
\end{lem}
\begin{proof}
By the boundary formula and by linearity of the boundary operator.
\end{proof}
\begin{lem}\label{paralello sum}
Let $x,u_1,\dots,u_n \in E$. We can write 
the $(n-1)$-dimensional parallelogram as a sum of $(n-1)$-dimensional parallelepipeds : 
\begin{equation*}
    \begin{split}
    P^{n-1}(x;u_1,\dots,u_n) &=  \sum_{k=1}^n (-1)^{k}   C^{n-1}(x;u_1, \dots,\hat{u_k},\dots ,u_{n})
    \\ &+ (-1)^n  \sum_{k=1}^n (-1)^{k}  C^{n-1}(x+\sum_{i=1}^n u_i;-u_1 \dots,-\hat{u_k},\dots, -u_{n}). 
    \end{split}
\end{equation*}
\end{lem}
\begin{proof}
By the preceding lemma, we have
\begin{equation*}
    \begin{split}
         P^{n-1}(x;u_1,\dots,u_n) &= \partial C^n(x;u_1,\dots,u_n)  \\ &=  \partial \Big( \sum_{k=1}^n (-1)^{k+1}  [x, C^{n-1}(x+u_k;u_1, \dots,\hat{u_k},\dots, u_{n})] \Big)  
         \\ &= \sum_{k=1}^n (-1)^{k+1} C^{n-1}(x+u_k;u_1, \dots,\hat{u_k},\dots, u_{n})
         \\ &- \sum_{k=1}^n (-1)^{k+1} [x, \partial C^{n-1}(x+u_k;u_1, \dots,\hat{u_k},\dots, u_{n})].
    \end{split}
\end{equation*}
Also by lemma \ref{boundary or parallelepiped}, we have 
$$  C^{n-1}(x+u_k;u_1, \dots,\hat{u_k},\dots, u_{n}) = (-1)^{n-1} C^{n-1}(x+\sum_{i=1}^n u_i;-u_1,\dots,-\hat{u_k},\dots,,-u_n).                       $$
So
\begin{equation*}
    \begin{split}
         P^{n-1}(x;u_1,\dots,u_n) &= \sum_{k=1}^n (-1)^{k} (-1)^{n} C^{n-1}(x+\sum_{i=1}^n u_i;-u_1,\dots,-\hat{u_k},\dots,,-u_n)
         \\ &- \sum_{k=1}^n (-1)^{k+1} [x, \partial C^{n-1}(x+u_k;u_1, \dots,\hat{u_k},\dots, u_{n})].
    \end{split}
\end{equation*}
Let us denote 
\begin{equation*}
    \begin{split}
       & S_1 = (-1)^{n} \sum_{k=1}^n (-1)^{k}  C^{n-1}(x+\sum_{i=1}^n u_i;-u_1,\dots,-\hat{u_k},\dots,,-u_n),
        \\& S_2 = \sum_{k=1}^n (-1)^{k+1} [x, \partial C^{n-1}(x+u_k;u_1, \dots,\hat{u_k},\dots, u_{n})],
    \end{split}
\end{equation*}
so that $ P^{n-1}(x;u_1,\dots,u_n) = S_1 -S_2 $. Let us first develop $S_2$. We have 
\begin{equation*}
    \begin{split}
         &\partial C^{n-1}(x+u_k;u_1, \dots,\hat{u_k},\dots, u_{n})  =  \sum_{\substack{p=1 \\ p \ne k}}^n (-1)^{p} \varepsilon(p,k) C^{n-2}(x+u_k;u_1,\dots,\hat{u_p},\dots,\hat{u_k},\dots, u_n     )               
         \\ &+      (-1)^{n-1}  \sum_{\substack{p=1 \\ p \ne k}}^n (-1)^{p} \varepsilon(p,k) C^{n-2}(x+\sum_{i=1}^n u_i;-u_1,\dots,-\hat{u_p},\dots,-\hat{u_k},\dots, -u_n     ),        
    \end{split}
\end{equation*}
where $\varepsilon(p,k) = 1$ if $p < k $ and $\varepsilon(p,k) = -1$ if $p > k $. Therefore
\begin{equation*}
    \begin{split}
    S_2 &= \sum_{k=1}^n (-1)^{k+1} [x, \partial C^{n-1}(x+u_k;u_1, \dots,\hat{u_k},\dots, u_{n})]
    \\ &= \sum_{k=1}^n (-1)^{k+1} \Big[x, \sum_{\substack{p=1 \\ p \ne k}}^n (-1)^{p} \varepsilon(p,k) C^{n-2}(x+u_k;u_1,\dots,\hat{u_p},\dots,\hat{u_k},\dots, u_n     )                 \Big]
    \\ &+ \sum_{k=1}^n (-1)^{k+1} \Big[x, (-1)^{n-1}  \sum_{\substack{p=1 \\ p \ne k}}^n (-1)^{p} \varepsilon(p,k) C^{n-2}(x+\sum_{i=1}^n u_i;-u_1,\dots,-\hat{u_p},\dots,-\hat{u_k},\dots, -u_n     )      \Big]
    \\ &= \sum_{\substack{p=1,k=1 \\ p \ne k}}^n (-1)^{k+1} (-1)^{p} \varepsilon(p,k) \Big[x, C^{n-2}(x+u_k;u_1,\dots,\hat{u_p},\dots,\hat{u_k},\dots, u_n     )                 \Big]
    \\ &+ (-1)^{n-1}  \sum_{\substack{p=1,k=1 \\ p \ne k}}^n  (-1)^{k+1} (-1)^{p} \varepsilon(p,k)    \Big[x,  C^{n-2}(x+\sum_{i=1}^n u_i;-u_1,\dots,-\hat{u_p},\dots,-\hat{u_k},\dots, -u_n     )      \Big]
     \\ &= \sum_{p=1}^n (-1)^{p} \sum_{\substack{k=1 \\ k \ne p}}^n (-1)^{k+1}  \varepsilon(p,k) \Big[x, C^{n-2}(x+u_k;u_1,\dots,\hat{u_p},\dots,\hat{u_k},\dots, u_n     )                 \Big]
    \\ &+ (-1)^{n-1}  \sum_{\substack{p=1,k=1 \\ p \ne k}}^n  (-1)^{p+k+1} \varepsilon(p,k)    \Big[x,  C^{n-2}(x+\sum_{i=1}^n u_i;-u_1,\dots,-\hat{u_p},\dots,-\hat{u_k},\dots, -u_n     )      \Big].
    \end{split}
\end{equation*}
Remark that 
\begin{equation*}
    \begin{split}
        \sum_{\substack{k=1 \\ k \ne p}}^n (-1)^{k+1}  \varepsilon(p,k) \Big[x, C^{n-2}(x+u_k;u_1,\dots,\hat{u_p},\dots,\hat{u_k},\dots, u_n     )                 \Big] = -C^{n-1}(x;u_1,\dots,\hat{u_p},\dots,u_n).
    \end{split}
\end{equation*}
So the first sum is just equal to $  - \sum_{p=1}^n (-1)^{p}C^{n-1}(x;u_1,\dots,\hat{u_p},\dots,u_n)$.
\\By denoting $E(p,k) := (-1)^{p+k+1} \Big[x,  C^{n-2}(x+\sum_{i=1}^n u_i;-u_1,\dots,-\hat{u_p},\dots,-\hat{u_k},\dots, -u_n     )      \Big] $, the second sum can be written as
$$  (-1)^{n-1}  \sum_{\substack{p=1,k=1 \\ p \ne k}}^n  \varepsilon(p,k)    E(p,k).               $$
Since $E(p,k) = E(k,p)$, and $\varepsilon(p,k) = - \varepsilon(k,p)$, this second sum is equal to zero. So
\begin{equation*}
    \begin{split}
        P^{n-1}(x;u_1,\dots,u_n) &= S_1 -S_2
        \\ &= (-1)^{n} \sum_{k=1}^n (-1)^{k}  C^{n-1}(x+\sum_{i=1}^n u_i;-u_1,\dots,-\hat{u_k},\dots,,-u_n) 
        \\ &+ \sum_{k=1}^n (-1)^{k}C^{n-1}(x;u_1,\dots,\hat{u_k},\dots,u_n). \qquad\qedhere
    \end{split}
\end{equation*}
\end{proof}
\hfill
\\Now, let us show the following decomposition result for parallelograms:
\begin{prop}\label{decomp result}
Let $x,u_1,\dots,u_{n+1} \in E$. If $u_{n+1} = \sum_{i=1}^p a_i$, with $a_1, \dots, a_p \in E$, Then 
$       P^n(x;u_1,\dots,u_{n+1}) - \sum_{i=1}^p P^n(x+\sum_{k=1}^{i-1}a_k;u_1,\dots,u_n,a_i)$ is an $n$-cycle, sum of $2n(p+1)$ $n$-parallelepipeds which are of the form $C^n(y;w_1,\dots,w_n)$, where $w_n \in \{u_{n+1},a_1,\dots,a_p   \}$ and $(w_1,\dots,w_{n-1}) \in \{(u_1,\dots,\hat{u_k},\dots,u_n)   \mid k \in \{1,\dots,n   \}      \}$.
\end{prop}
In particular, if $\| u_1  \| > \| u_2  \| > \dots > \| u_{n+1}  \| $, which is the case that we will be interested in, then the cycle obtained has small volume since all its parallelepipeds have vectors $\ne (u_1,\dots,u_n)$. In other words, we decompose the cycle $P^n(x;u_1,\dots,u_{n+1})$ along the vectors $a_1,\dots,a_p$, plus a residual cycle with a very small volume compared to the other cycles. 

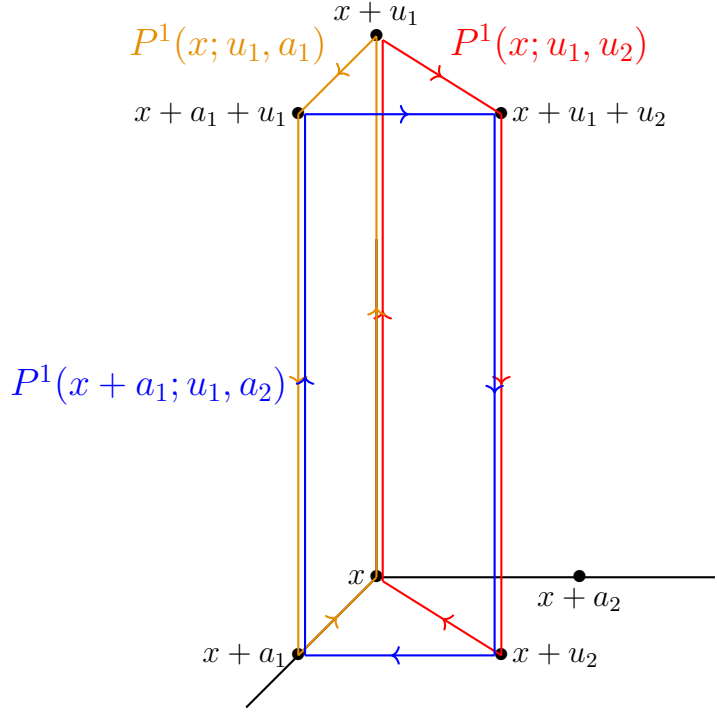
\begin{figure}[htbp]
  \centering
\begin{tikzpicture}[scale=0.9]
\draw[thick] (0,0,0)--(0,0,5); 
\draw[thick] (0,0,0)--(0,5,0);
\draw[thick] (0,0,0)--(5,0,0);

\coordinate (a_n) at (0,0,0);
\node at (a_n) [left] {$x$};
\node at (a_n) {\textbullet};

\coordinate (a_n") at (0.1,0,0.2);
\coordinate (b_n") at (0.1,8,0.2);

\coordinate (a_n') at (3,0,3);
\node at (a_n') [right] {$x+u_2$};
\node at (a_n') {\textbullet};

\coordinate (b_n) at (0,8,0);
\node at (b_n) [above] {$x+u_1$};
\node at (b_n) {\textbullet};

\coordinate (b_n') at (3,8,3);
\node at (b_n') [right] {$x+u_1+u_2$};
\node at (b_n') {\textbullet};

\coordinate (c_n) at (0,0,3);
\node at (c_n) [left] {$x+a_1$};
\node at (c_n) {\textbullet};

\coordinate (d_n) at (0,8,3);
\node at (d_n) [left] {$x+a_1+u_1$};
\node at (d_n) {\textbullet};

\coordinate (e) at (3,0,0);
\node at (e) [below] {$x+a_2$};
\node at (e) {\textbullet};

\node[left, color=blue] at (0,4,3)  {\large{$P^1(x+a_1;u_1,a_2)$}};
\node[above left, color=green!10!orange] at (0,8,1.5)  {\large{$P^1(x;u_1,a_1)$}};
\node[above right, color=red] at (1.5,8,1.5)  {\large{$P^1(x;u_1,u_2)$}};

\begin{scope}[thick, decoration={markings, mark=at position 0.5 with {\arrow{>}}}]

    \draw[postaction={decorate}, color=red] (0.15,0,0.15) -- (0.15,8,0.15);
    \draw[postaction={decorate}, color=red] (0.15,8,0.15) -- (b_n');
    \draw[postaction={decorate}, color=red] (b_n') -- (a_n');
    \draw[postaction={decorate}, color=red] (a_n') -- (0.15,0,0.15);
\end{scope}
\begin{scope}[thick, decoration={markings, mark=at position 0.5 with {\arrow{>}}}]

    \draw[postaction={decorate}, color=green!10!orange] (a_n) -- (b_n);
    \draw[postaction={decorate}, color=green!10!orange] (b_n) -- (d_n);
    \draw[postaction={decorate}, color=green!10!orange] (d_n) -- (c_n);
    \draw[postaction={decorate}, color=green!10!orange] (c_n) -- (a_n);
\end{scope}
\begin{scope}[thick, decoration={markings, mark=at position 0.5 with {\arrow{<}}}]

    \draw[postaction={decorate}, color=blue] (2.9,0,3) -- (2.9,8,3);
    \draw[postaction={decorate}, color=blue] (2.9,8,3) -- (0.1,8,3);
    \draw[postaction={decorate}, color=blue] (0.1,8,3) -- (0.1,0,3);
    \draw[postaction={decorate}, color=blue] (0.1,0,3) -- (2.9,0,3);
\end{scope}
\end{tikzpicture}
  \caption{A decomposition of $P^1(x;u_1,u_2)$ along the vectors $a_1$ and $a_2$. \\The residual cycle is the sum of the boundaries of the upper and lower triangles.}
  \label{fig:CubeNormal}
\end{figure}

\begin{proof}
Let $i \in \{1,\dots,p   \}  $, 
\begin{equation*}
    \begin{split}
    P^n(x+\sum_{k=1}^{i-1}a_k;u_1,\dots,u_n,a_i) &=  \sum_{s=1}^{n+1} (-1)^{s}   C^{n}(x+\sum_{k=1}^{i-1}a_k;u_1, \dots,\hat{u_s},\dots ,u_{n},a_i)
    \\ &+ (-1)^{n+1}  \sum_{s=1}^{n+1} (-1)^{s}  C^{n}(x+\sum_{k=1}^{i}a_k+\sum_{i=1}^n u_i;-u_1 \dots,-\hat{u_s},\dots, -u_{n},-a_i)
    \\ & =  \sum_{s=1}^{n} (-1)^{s}   C^{n}(x+\sum_{k=1}^{i-1}a_k;u_1, \dots,\hat{u_s},\dots ,u_{n},a_i)
    \\ &+ (-1)^{n+1}  \sum_{s=1}^{n} (-1)^{s}  C^{n}(x+\sum_{k=1}^{i}a_k+\sum_{i=1}^n u_i;-u_1 \dots,-\hat{u_s},\dots, -u_{n},-a_i)
    \\ & +(-1)^{n+1}   C^{n}(x+\sum_{k=1}^{i-1}a_k;u_1,\dots ,u_{n})
    \\ &+ (-1)^{n+1}   (-1)^{n+1}  C^{n}(x+\sum_{k=1}^{i}a_k+\sum_{i=1}^n u_i;-u_1,\dots, -u_{n}).
    \end{split}
\end{equation*}
By lemma \ref{boundary or parallelepiped}, we have
$$ C^{n}(x+\sum_{k=1}^{i}a_k+\sum_{i=1}^n u_i;-u_1,\dots, -u_{n}) =  (-1)^n C^{n}(x+\sum_{k=1}^{i}a_k;u_1 ,\dots, u_{n})   .        $$ 
So
\begin{equation*}
    \begin{split}
       & (-1)^{n+1}   C^{n}(x+\sum_{k=1}^{i-1}a_k;u_1,\dots ,u_{n})
    \\ &+ (-1)^{n+1}   (-1)^{n+1}  C^{n}(x+\sum_{k=1}^{i}a_k+\sum_{i=1}^n u_i;-u_1,\dots, -u_{n})
    \\  =      & (-1)^{n+1}   C^{n}(x+\sum_{k=1}^{i-1}a_k;u_1,\dots ,u_{n})
    \\ &+ (-1)^{n}  C^{n}(x+\sum_{k=1}^{i}a_k;u_1,\dots, u_{n}).
    \end{split}
\end{equation*}

By adding them, for $i \in \{1,\dots,p   \}  $, we get 
\begin{equation*}
    \begin{split}
      &  \sum_{i=1}^p  \Big( (-1)^{n+1}   C^{n}(x+\sum_{k=1}^{i-1}a_k;u_1,\dots ,u_{n})
    + (-1)^{n}  C^{n}(x+\sum_{k=1}^{i}a_k;u_1, \dots, u_{n}) \Big)
    \\& = (-1)^n \Big(   C^{n}(x+u_{n+1};u_1, \dots,u_{n}) -  C^{n}(x;u_1, \dots,u_{n})      \Big).
    \end{split}
\end{equation*}
So the sum of the $p$ parallelograms $ \sum_{i=1}^p  P^n(x+\sum_{k=1}^{i-1}a_k;u_1,\dots,u_n,a_i)$ gives
\begin{equation*}
    \begin{split}
     & \sum_{i=1}^p       \sum_{s=1}^{n} (-1)^{s}   C^{n}(x+\sum_{k=1}^{i-1}a_k;u_1, \dots,\hat{u_s},\dots ,u_{n},a_i) 
    \\ &+ (-1)^{n+1} \sum_{i=1}^p    \sum_{s=1}^{n} (-1)^{s}  C^{n}(x+\sum_{k=1}^{i}a_k+\sum_{i=1}^n u_i;-u_1 \dots,-\hat{u_s},\dots, -u_{n},-a_i)          
    \\&+  (-1)^n \Big(   C^{n}(x+u_{n+1};u_1, \dots,u_{n}) -  C^{n}(x;u_1, \dots,u_{n})      \Big).
    \end{split}
\end{equation*}
On the other hand, we have
\begin{equation*}
    \begin{split}
        P^n(x;u_1,\dots,u_{n+1}) &= \sum_{k=1}^{n+1} (-1)^{k}   C^{n}(x;u_1, \dots,\hat{u_k},\dots ,u_{n+1})
    \\ &+ (-1)^{n+1}  \sum_{k=1}^{n+1} (-1)^{k}  C^{n}(x+\sum_{i=1}^{n+1} u_i;-u_1 \dots,-\hat{u_k},\dots, -u_{n+1})
    \\ &= \sum_{k=1}^{n} (-1)^{k}   C^{n}(x;u_1, \dots,\hat{u_k},\dots ,u_{n+1})
    \\ &+ (-1)^{n+1}  \sum_{k=1}^{n} (-1)^{k}  C^{n}(x+\sum_{i=1}^{n+1} u_i;-u_1 \dots,-\hat{u_k},\dots, -u_{n+1})
    \\ &+ (-1)^{n+1}   C^{n}(x;u_1, \dots,u_{n})
    \\ &+(-1)^{n} C^{n}(x+u_{n+1};u_1 \dots, u_{n}).
    \end{split}
\end{equation*}
Since $(-1)^{n+1} (-1)^{n+1} C^{n}(x+\sum_{i=1}^{n+1} u_i;-u_1 \dots, -u_{n}) = (-1)^{n} C^{n}(x+u_{n+1};u_1 \dots, u_{n}) $.
\\Finally, by taking the difference, all the chains whose vectors are $(u_1, \dots, u_{n})$ simplifies and we get:
\begin{equation*}
  \begin{split}
       &  P^n(x;u_1,\dots,u_{n+1}) - \sum_{i=1}^p P^n(x+\sum_{k=1}^{i-1}a_k;u_1,\dots,u_n,a_i) 
     \\ & = \sum_{k=1}^{n} (-1)^{k}   C^{n}(x;u_1, \dots,\hat{u_k},\dots ,u_{n+1})
    \\ & + (-1)^{n+1}  \sum_{k=1}^{n} (-1)^{k}  C^{n}(x+\sum_{i=1}^{n+1} u_i;-u_1 \dots,-\hat{u_k},\dots, -u_{n+1})
    \\ & - \sum_{i=1}^p       \sum_{s=1}^{n} (-1)^{s}   C^{n}(x+\sum_{k=1}^{i-1}a_k;u_1, \dots,\hat{u_s},\dots ,u_{n},a_i) 
    \\ &- (-1)^{n+1} \sum_{i=1}^p    \sum_{s=1}^{n} (-1)^{s}  C^{n}(x+\sum_{k=1}^{i}a_k+\sum_{i=1}^n u_i;-u_1 \dots,-\hat{u_s},\dots, -u_{n},-a_i)       .   
  \end{split}
\end{equation*}
This is a sum of $2n(p+1)$ $n$-parallelepipeds, whose last vector is either $u_{n+1}$ or an $a_i$, and whose first $(n-1)$-vectors are in $\{(u_1,\dots,\hat{u_k},\dots,u_n)   \mid k \in \{1,\dots,n   \}      \}$.
\end{proof}
\subsection{Proof of Theorem \ref{Thm 2 v2}}\label{sect 4.3}
We saw in subsection \ref{parallel sets} that the cross section of a product is the product of cross sections. Therefore, to prove Theorem \ref{Thm 2 v2}, it is enough to prove it when $X$ is an irreducible symmetric space of non-compact type or a Euclidean building with no Euclidean factor. Let us prove the following more general result.
\begin{thm} \label{CE of model space}
Let $r\geq 2$ be an integer. Let $\varphi_1,\varphi_2,\dots,\varphi_r$ be functions from $\mathbb{R}_+$ to $\mathbb{R}_+$ such that $\varphi_1(d)=d$ for all $d$, and for all $i \,$  $\varphi_i \gg \varphi_{i+1} $, and $\varphi_i(d)$ tends to $+\infty$ at $+\infty$. Suppose moreover that for all $p=2, \dots , r$, $\left( \prod_{\substack{i=1 \\ i \ne p-1}}^{p}  \varphi_i(d) \right)^{\frac{p}{p-1}} \ll \prod_{i=1}^{p}  \varphi_i(d) $. 
\\Let $X$ be a symmetric space of non-compact type of rank $\geq r$, or a thick Euclidean building with cocompact affine Weyl group, with bounded geometry and no Euclidean factor, of dimension $\geq r$. Let $Y$ be a Lipschitz-connected complete metric space with  at most exponential growth, and let $f : X \to Y$ a large-scale Lipschitz map. 
If there exists a sublinear map $\phi$ such that 
\\$\forall x \in X$, for every $F$ maximal flat that contains $x$, $\forall d >0$, $\forall u_1,\dots,u_r \in T_x F \simeq F$ that satisfy
\\• $u_1,\dots,u_{r-1}$ are maximally singular,
\\• $\| u_1  \| = d$ and $\forall i = 2, \dots, r$, $\| u_i  \| \leq \varphi_i(d)$,
\\we have 
$$     \textup{FillVol}_{r}^{Y, \mathrm{cr}}(f(P^{r-1}(x;u_1,\dots,u_r)))            \leq \phi \left(\prod_{i=1}^{r}  \varphi_i(d)\right),         $$
\\then $f$ is not a coarse embedding.
\end{thm}
Before going into details, let us recall the main idea of the proof. The proof of this theorem is an adaptation of that of Theorem \ref{CE of product}. Since the domain is no longer a product space, we consider parallelograms instead of rectangles, and adapt the size of its sides to make the proof work. Similarly, the proof will be done by induction. When $r=2$, like in Theorem \ref{CE of product}, we start by looking for pairs of points $(a_n)_n$, $(b_n)_n$ in the domain such that $d_n := d_X(a_n,b_n) \to +\infty$ and which are mapped in $Y$ quasi-isometrically. Moreover, they must also be lying on a maximally singular geodesic in order to get the splitting of the domain. To find these pairs of points, we start by the pairs of points $(a_n)_n$, $(b_n)_n$ obtained by the exponential growth of the domain. Then we decompose the geodesic segment $[a_n,b_n]$ along the maximally singular geodesics generating a Weyl sector that contains $[a_n,b_n]$. We obtain the desired pairs of points by a pigeonhole-like argument. The rest of the proof is similar to that of Theorem \ref{CE of product}. When $r\geq2$, instead of looking for maximally singular geodesic segments which are undistorted, we look for maximally singular parallelograms (i.e.\ whose sides are maximally singular segments) with undistorted filling. To do that, we use the induction hypothesis to find undistorted, almost singular, parallelograms with undistorted filling. Then we decompose them along the walls of a Weyl sector using Proposition \ref{decomp result}, and we apply a pigeonhole-like argument.
\begin{proof}
Let $r \geq 2$ be fixed, and let $\varphi_1,\varphi_2,\dots,\varphi_r$ be functions satisfying the four conditions. We will prove the theorem by induction on $k$, from $k=2$ to $k=r$.
\\• Let us start by the case $k =2$.
\\Let $X$ be a symmetric space of non-compact type of rank $p \geq 2$, or a Euclidean building of rank $p \geq 2$. Suppose there exists a sublinear map $\phi$ such that $\forall x \in X$, for every $F$ maximal flat that contains $x$, $\forall d >0$, $\forall u_1,u_2 \in T_x F \simeq F$ that satisfy: $u_1$ is maximally singular, $\| u_1  \| = \varphi_1(d) = d$ and $\| u_2  \| \leq \varphi_2(d)$, we have 
$$     \textup{FillVol}_{2}^{Y, \mathrm{cr}}(f(P^{1}(x;u_1,u_2)))            \leq \phi (d \, \varphi_2(d)).  $$
Since $X$ has exponential growth and $Y$ has at most exponential growth, the case $k=1$ in Theorem \ref{CE of product} implies that $X$ is not sent sublinearly, i.e.\ there exist $\lambda >0$ and two sequences $(a_n)_{n \in \mathbb{N} }$ and $(b_n)_{n \in \mathbb{N} }$ in $X$ such that $d_n := d_X(a_n, b_n) \to \infty$, and $\forall n \in \mathbb{N}$
$$  d_Y(f(a_n),f(b_n)) \geq \lambda d_X(a_n, b_n)    .               $$
For all $n \in \mathbb{N}$, let $\gamma_n$ denote the geodesic segment in $X$ going from $a_n$ to $b_n$, which is unique since $X$ is CAT(0). If $\gamma_n$ is maximally singular, then there is nothing to do at this step. If not, consider a maximal flat $F_n$ containing $\gamma_n$ and we will denote $w_n := \frac{d_n}{\| \gamma_n'(0) \|} \gamma_n'(0) $ the directing vector of $\gamma_n$ with magnitude $d_n$.
\\Let $C$ be a Weyl sector containing $\gamma_n$, i.e.\ a Weyl sector of $F_n$ with tip at $a_n$, such that $b_n \in \overline{C}$. So there exist $u_1,\dots,u_p$ maximally singular vectors at $a_n$ generating $C$, in particular there exist $\delta_1,\dots,\delta_p \geq 0 $ such that $     w_n = \sum_{i=1}^p \delta_i u_i  $.
\\
\\Now let us consider the geodesic segment $\gamma_n$ as the $1$-chain $C^1(a_n;w_n)$ in $F_n$, and its boundary $\partial C^1(a_n;w_n) = P^0(a_n;w_n) $. We have that 
$$  P^0(a_n;w_n) =   P^0(a_n;\delta_1 u_1) +  P^0(a_n + \delta_1 u_1;\delta_2 u_2) + \dots +  P^0(a_n+ \sum_{i=1}^{p-1} \delta_i u_i;\delta_p u_p).                             $$
So, by applying $f$ which is functorial, we have 
$$ f( P^0(a_n;w_n)) =   \sum_{i=1}^p f(  P^0(a_n+ \sum_{s=1}^{i-1} \delta_s u_s;\delta_i u_i)   ).                          $$
So 
\begin{equation*}
    \begin{split}
         \textup{FillVol}_{1}^{Y, \mathrm{cr}} ( f( P^0(a_n;w_n)) ) &=   \textup{FillVol}_{1}^{Y, \mathrm{cr}}    \big(    \sum_{i=1}^p f(  P^0(a_n+ \sum_{s=1}^{i-1} \delta_s u_s;\delta_i u_i)   )        \big)   
     \\    & \leq      \sum_{i=1}^p \textup{FillVol}_{1}^{Y, \mathrm{cr}} \big( f(  P^0(a_n+ \sum_{s=1}^{i-1} \delta_s u_s;\delta_i u_i)   )  \big)  .   
    \end{split}
\end{equation*}
Since $f(P^0(a_n;w_n)) = [f(b_n)] - [f(a_n)] $, by lemma \ref{filling current distance} $$\textup{FillVol}_{1}^{Y, \mathrm{cr}} ( f( P^0(a_n;w_n)) ) = d_Y(f(a_n),f(b_n)). $$
So there exists $i_0 \in \{1,\dots,p       \} $ such that
\begin{equation*}
    \begin{split}
        \frac{1}{p}d_Y(f(a_n),f(b_n)) &\leq  \textup{FillVol}_{1}^{Y, \mathrm{cr}} \big( f(  P^0(a_n+ \sum_{s=1}^{i_0-1} \delta_s u_s;\delta_{i_0} u_{i_0})   )  \big)
        \\ &= d_Y(f(a_n+ \sum_{s=1}^{i_0-1} \delta_s u_s),f(a_n+ \sum_{s=1}^{i_0} \delta_s u_s)).
    \end{split}
\end{equation*}
Let us denote $a_n' = a_n+ \sum_{s=1}^{i_0-1} \delta_s u_s $ and $b_n' =  a_n+ \sum_{s=1}^{i_0} \delta_s u_s $. So we have
$$ \frac{\lambda}{p} d_X(a_n,b_n)  \leq     \frac{1}{p} d_Y(f(a_n),f(b_n)) \leq  d_Y(f(a_n'),f(b_n')) .              $$
This implies in particular that $ d_Y(f(a_n'),f(b_n')) \to \infty $, which implies that $d_X(a_n',b_n') \to \infty $ because $f$ is a coarse embedding.
\\Since $C$ is an acute simplicial cone, we use corollary \ref{corollary of acute}. We have $$d_X(a_n,b_n) = \| w_n  \| \geq \| \delta_{i_0} u_{i_0} \| = d_X(a_n',b_n'). $$
We finally get
$$ \frac{\lambda}{p} d_X(a_n',b_n') \leq  d_Y(f(a_n'),f(b_n')) .              $$
Now we have two sequences $(a_n')_n$ and $(b_n')_n$ in $X$ that satisfy  $d_X(a_n', b_n') \to \infty$, and $\forall n \in \mathbb{N}$
$$  d_Y(f(a_n'),f(b_n')) \geq \lambda' d_X(a_n', b_n')             . $$
Where $\lambda' = \frac{\lambda}{p}$, and more importantly, the geodesic between $a_n'$ and $b_n'$ is maximally singular.
\\
\\To simplify the notations for the rest of the proof, we will just denote $a_n'$ by $a_n$, $b_n'$ by $b_n$ and $\lambda'$ by $\lambda$.
\\For every $n \in \mathbb{N}$, let us denote $\gamma_n$ the bi-infinite geodesic in $X$ that contains $a_n$ and $b_n$, and consider $P_X(\gamma_n)$ the parallel set of $\gamma_n$. $P_X(\gamma_n)$ is a convex subset that splits metrically as 
$$  P_X(\gamma_n)  =    \mathbb{R} \times C_X(\gamma_n) .         $$
The cross section $C_X(\gamma_n)$ is either :
\\• A symmetric space of non-compact type of rank $\geq 1$, when $X$ is a symmetric space, 
\\• Or a Euclidean building with bounded geometry of dimension $\geq 1$, with no Euclidean factor, when $X$ is a Euclidean building. 
\\Let us denote $C_X(\gamma_n)$ by $X_n'$. Note that all such cross sections have exponential growth, and by Proposition \ref{prop cross section growth}, there is a uniform lower bound on their growth. Therefore, if we fix $\varepsilon >0$, there exists $\mu >0$ such that for all $n \in \mathbb{N}$ and for all $R>0$
$$        \textup{Vol}^{\varepsilon} \left( B_{X_n'}(R)  \right) \geq \exp(\mu R)  .              $$
\\Now we will work in this product space and apply the same strategy as when $X$ is a product space.
$$  P_X(\gamma_n)  =    \mathbb{R} \times X_n' .         $$
\\Note that the slices $ \mathbb{R} \times \{x'\}$, with $x' \in X'$, are the geodesics that are parallel to $\gamma$. Let $x_n' \in X'$ such that $\gamma_n =  \mathbb{R} \times \{x_n'\}$ .
\\For all $y \in B_{X_n'}\big( x_n',\varphi_2(d_n)      \big) $, consider the geodesic $ [ x_n',y  ]$ in $X_n'$. Denote $\gamma_n'$ the (unique since $X_n'$ is CAT(0)) bi-infinite geodesic extension of $ [ x_n',y  ]$, and $F_n'' = F_n' \times \gamma_n'$ the $2$-flat. Consider the $1$-parallelogram $ P^{1}(a_n;u_1^n,u_{2}^n) $ in $F_n''$, where $u_1$ is the (maximally singular) directing vector of $\gamma_n$ that satisfies $\| u_{1}^n  \| = d_{X}(a_n,b_n)  $, and $u_{2}^n$ is the unique vector in $T_{x_n}F_n''$ that is parallel to $\gamma_n'$ and satisfies $\| u_{2}^n  \| = d_{X_n'}(x_n',y)  $. This $1$-parallelogram satisfies the induction hypothesis: 
$$     \textup{FillVol}_{2}^{Y, \mathrm{cr}}(f(P^{1}(a_n;u_1^n,u_{2}^n)))            \leq \phi (d_n \,\varphi_2(d_n) ).         $$
Which implies that there exists $V_n \in \mathbf{I}_{2}(Y)$ in $Y$ such that $\partial V_n = f(P^{1}(a_n;u_1^n,u_{2}^n))$ and 
$$
\mass{V_n} \leq \phi (d_n \,\varphi_2(d_n)).
$$
Now consider the 1-Lipschitz map $\pi : Y \to \mathbb{R}$, $\pi(z) = d_Y(z,f(C^{1}(a_n;u_1^n))) $.
\\By the Slicing Theorem, we have that for a.e.\ $t \in \mathbb{R}$, there exists $<V_n,\pi,t> \, \in \mathbf{I}_1(Y)$ such that $\slice{V_n}{\pi}{t}= \partial(V_n\rstr\{\pi \leq t\}) - (\partial V_n)\rstr\{\pi\leq t\}$, and by integrating the co-area formula over the distance $t$, we have
$$\mass{\slice{V_n}{\pi}{t}}\leq\frac{d}{dt}\mass{V_n\rstr\{\pi\leq t\}},$$
$$ \int_0^{+\infty} \mass{\slice{V_n}{\pi}{t}} dt \leq \int_0^{+\infty} \frac{d}{dt}\mass{V_n\rstr\{\pi\leq t\}},$$
$$ \int_0^{+\infty} \mass{\slice{V_n}{\pi}{t}} dt \leq \mass{V_n}.$$
Since $\mass{V_n} \leq \phi (d_n \,\varphi_2(d_n)) $, we get
\begin{equation}\label{equation 7}
\int_0^{D} \mass{\slice{V_n}{\pi}{t}} dt \leq \int_0^{+\infty} \mass{\slice{V_n}{\pi}{t}} dt \leq \phi (d_n \,\varphi_2(d_n)),
\end{equation}
where $D=d_Y\big(f(C^{1}(a_n;u_1^n)),f(C^{1}(a_n+u_2^n;u_1^n))\big)$.
\\However, for a.e.\ $t \in \left]0,D\right[$, $\mass{\slice{V_n}{\pi}{t}}$ cannot be too small since $\slice{V_n}{\pi}{t}$ almost gives a filling of the $0$-cycle $\partial \left( f(C^{1}(a_n;u_1^n)) \right) =  [\![ f(b_n) ]\!] - [\![ f(a_n) ]\!] $.
\begin{claim}
 For $n$ big enough: for a.e.\ $t \in \left]0,D\right[$, $\mass{\slice{V_n}{\pi}{t}} \geq \frac{\lambda}{2} d_n$.
\end{claim}
\begin{proof}
For a.e.\ $t \in \left]0,D\right[$,
\begin{equation*}
    \begin{split}
       \partial(V_n) \rstr\{\pi \leq t\} &= \partial(V_n) \rstr\{\pi =0\} + \partial(V_n) \rstr\{ 0 < \pi \leq t\} 
       \\ &= \big( f(b_n)-f(a_n) \big) + H_t.
    \end{split}
\end{equation*}
Where $H_t$ is the $1$-current $\partial (V_n) \rstr\{ 0 < \pi \leq t\}$.
\\Since $\|\slice{V_n}{\pi}{t}\|$ is concentrated on $\pi^{-1}(\{t\})$,
$$\partial(V_n\rstr\{\pi \leq t\}) = \slice{V_n}{\pi}{t} + \big( f(b_n)-f(a_n) \big) + H_t. $$
Which means that $\slice{V_n}{\pi}{t} + \big( f(b_n)-f(a_n) \big) + H_t$ is a $1$-current that is actually a cycle.
\\So $-(\slice{V_n}{\pi}{t} + H_t)$ is a $1$-chain that fills the $0$-cycle $ [f(b_n)] - [f(a_n)]$. Therefore,
$$\mass{ -(\slice{V_n}{\pi}{t} + H_t) } \geq \textup{FillVol}_{k}^{Y, \mathrm{cr}}( f(b_n)-f(a_n)) \geq \lambda d_n.   $$
So 
$$\mass{\slice{V_n}{\pi}{t}} + \mass{H_t} \geq \lambda \, d_n .$$
Note that
\begin{equation*} 
\begin{split}
H_t & =  \partial (V_n) \rstr\{ 0 < \pi \leq t\} 
     \\ & =\big(  f(C^1(a_n;u_2^n) ) + f(C^1(b_n;u_2^n) ) \big) \rstr\{ 0 <\pi \leq t \} .
\end{split}
\end{equation*}
By taking the mass
\begin{equation*} 
\begin{split}
\mass{H_t} & = \mass{( \ f(C^1(a_n;u_2^n) ) + f(C^1(b_n;u_2^n) ) ) \rstr\{ 0 <\pi \leq t \} }  
\\ & \leq \mass{   f(C^1(a_n;u_2^n) ) \rstr\{ 0<\pi \leq t \}} + \mass{   f(C^1(b_n;u_2^n) ) \rstr\{ 0<\pi \leq t \}}      \\
& \leq \mass{   f(C^1(a_n;u_2^n) )} + \mass{   f(C^1(b_n;u_2^n) ) }  \\
& \leq \textup{Lip}(f) \big( \mass{   C^1(a_n;u_2^n)} + \mass{   C^1(b_n;u_2^n) }  \big).
\end{split}
\end{equation*}
Since $\mass{   C^1(a_n;u_2^n)} = \mass{   C^1(b_n;u_2^n) } \leq \varphi_2(d_n) $, so
$$\mass{H_t} \leq   2 \textup{Lip}(f) \, \varphi_2(d_n). $$
So 
$$\mass{\slice{V_n}{\pi}{t}} \geq \lambda \, d_n - \mass{H_t} \geq \lambda \, d_n - 2 \textup{Lip}(f) \,\varphi_2(d_n). $$
Since $\varphi_2(d_n) = o(d_n)$, there exists $N \in \mathbb{N}$ such that for all $n \geq N$
$$\lambda \, d_n - 2 \textup{Lip}(f)  \,\varphi_2(d_n) \geq \frac{\lambda}{2} \, d_n. $$
So we conclude that for all $n \geq N$ and for a.e.\ $t \in \left]0,D\right[$, $$\mass{\slice{V_n}{\pi}{t}} \geq \frac{\lambda}{2} d_n. \qquad\qedhere $$
\end{proof}
So, by \eqref{equation 7}, we have
\begin{equation*}
\begin{split}
    \phi ( d_n \,\varphi_2(d_n)) & \geq \int_0^{D} \mass{\slice{V_n}{\pi}{t}} dt\\ 
& \geq D \,  \frac{\lambda}{2} d_n.
\end{split}
\end{equation*}
Which implies that
\begin{equation*}
 D \leq \frac{2}{\lambda}  \frac{\phi ( d_n \,\varphi_2(d_n))}{d_n} .
\end{equation*}
Let us denote $\psi(\varphi_2(d_n)) = \frac{2}{\lambda}  \frac{\phi ( d_n \,\varphi_2(d_n))}{d_n}$. Note that $\psi$ is sublinear: $\frac{\psi(\varphi_2(d_n))}{\varphi_2(d_n)}$ tends to $0$.
\\Since $D=d_Y\big(f(C^{1}(a_n;u_1^n)),f(C^{1}(a_n+u_2^n;u_1^n))\big)$, the last inequality implies that there exists $z \in C^{1}(a_n+u_2^n;u_1^n)$ such that 
$$d_Y(f(C^{1}(a_n;u_1^n)),f(z)) \leq \psi(\varphi_2(d_n)).$$
But $z \in C^{1}(a_n+u_2^n;u_1^n)$ implies that $\textup{proj}_{X_n'}(z) = y$.
\\
\\If we choose another $y \in B_{X_n'}\big( x_n',\varphi_2(d_n)        \big)$, we get another $z$ such that 
$$\textup{proj}_{X_n'}(z) = y,$$ 
$$d_Y(f(C^{1}(a_n;u_1^n)),f(z)) \leq \psi(\varphi_2(d_n)).$$
\\By doing this process $\forall y \in B_{X_n'}\big( x_n',\varphi_2(d_n)        \big) $, we get a subset $C_n \subset X$ that projects onto $B_{X_n'}\big( x_n',\varphi_2(d_n)    \big)$, i.e.\ 
\begin{equation*}
    B_{X_n'}\big( x_n',\varphi_2(d_n)      \big) \subset \textup{proj}_{X_n'}(C_n) , 
\end{equation*} 
and such that
$$ f(C_n) \subset N_{\psi\left(\varphi_2(d_n)\right)}   (f(C^{1}(a_n;u_1^n))).$$
\\Since the projection onto $X_n'$ is 1-Lipschitz, it implies that 
\begin{equation*} 
    \textup{Vol}_Y^{\varepsilon}(f(C_n)) \leq \textup{Vol}_Y^{\varepsilon}\left(N_{\psi\left(\varphi_2(d_n)\right)}   (f(C^{1}(a_n;u_1^n)))\right).
\end{equation*}
$f$ coarsely preserves volumes, so there exist $\delta, \delta' >0$ such that 
$$  \delta \, \textup{Vol}_X^{\varepsilon} ( C_n )  \leq      \textup{Vol}_Y^{\varepsilon} ( f( C_n ) ) \leq \delta'  \, \textup{Vol}_X^{\varepsilon} ( C_n ).      $$
And 
\begin{equation*}
  \textup{Vol}_Y^{\varepsilon}\left(N_{\psi\left(\varphi_2(d_n)\right)}   (f(C^{1}(a_n;u_1^n)))\right) \leq \beta_Y^{\varepsilon}\left( \varepsilon + \psi(\varphi_2(d_n) \right)  \times  \textup{Vol}_Y^{\varepsilon}(   f(C^{1}(a_n;u_1^n))).
\end{equation*}
$Y$ has at most exponential growth, so $\exists \beta >0$ such that $ \forall R>0$,
$ \beta_Y^{\varepsilon}\left( R \right) \leq      e^{\beta R} $.
In particular, we have on hand that
$$\beta_Y^{\varepsilon}\left( \varepsilon + \psi(\varphi_2(d_n) \right) \leq    \textup{exp}( 2 \beta      \psi ( \varphi_2(d_n))            ) . $$ 
On the other hand, by taking a partition of the geodesic segment $[a_n,b_n]$ into sub-intervals of length $\varepsilon$, we get:
$$ \textup{Vol}_X^{\varepsilon} \big( C^{1}(a_n;u_1^n)\big)         \leq   \frac{ \| u_1^n  \|}{\varepsilon} +1    \leq   \frac{ 2\| u_1^n  \|}{\varepsilon}    \leq \big(2/\varepsilon \big)d_n.              $$
So, by denoting $A =  2/\varepsilon $, we have
\begin{equation*}
    \begin{split}
        \textup{Vol}_Y^{\varepsilon}(   f(C^{1}(a_n;u_1^n)))
         &  \leq \delta' \, \textup{Vol}_X^{\varepsilon} \big( C^{1}(a_n;u_1^n)\big)   
         \\ &\leq \delta' A \, d_n .
    \end{split}
\end{equation*}
Therefore
\begin{equation*}
    \textup{Vol}_Y^{\varepsilon}\left(N_{\psi\left(\varphi_2(d_n)\right)}   (f(C^{1}(a_n;u_1^n)))\right) \leq \textup{exp}( 2 \beta      \psi ( \varphi_2(d_n))   ) \times \delta' A \, d_n.
\end{equation*}
We conclude from all the previous inequalities that 
\begin{equation*}
   \delta \, \textup{exp}(\mu \varphi_2(d_n)) \leq \delta \, \textup{Vol}_X^{\varepsilon} ( C_n ) \leq  \textup{Vol}_Y^{\varepsilon} ( f( C_n ) ) \leq \textup{exp}( 2 \beta      \psi ( \varphi_2(d_n))   ) \times \delta' A \, d_n .
\end{equation*}
Which implies finally that for all $n\geq N$
\begin{equation*}
   \delta \, \textup{exp}(\mu \varphi_2(d_n)) \leq \textup{exp}( 2 \beta      \psi ( \varphi_2(d_n))   ) \times \delta' A \, d_n .
\end{equation*}
Which is not possible when $d_n \to \infty$ because $\psi$ is sublinear.
\\This completes the proof of the case "rank $X \geq 2$" .
\\
\\
\\Now suppose that it is true for rank $X \geq k$ for some $k \in \{2,\dots,r-1\}$, and let us prove it for rank $\geq k+1$.
\\Let $X$ be of rank $p \geq k+1$ and suppose that such a coarse embedding $f : X \to Y$ exists, i.e.\ there exists a sublinear function $\phi$ such that $\forall x \in X$, for every maximal flat $F$ that contains $x$, $\forall d >0$, $\forall u_1,\dots,u_{k+1} \in T_x F \simeq F$ that satisfy
\\• $u_1,\dots,u_{k}$ are maximally singular,
\\• $\| u_1  \| = \varphi_1(d) = d$ and $\forall i = 2, \dots, k+1$, $\| u_i  \| \leq \varphi_i(d)$, we have
$$     \textup{FillVol}_{k+1}^{Y, \mathrm{cr}}(f(P^{k}(x;u_1,\dots,u_{k+1})))            \leq \phi (\varphi_1(d)\times \dots \times \varphi_{k+1}(d)).         $$
$X$ is of rank $\geq k+1$, so it is of rank $\geq k$ and coarsely embeds into $Y$, therefore it does not satisfy the "rank $\geq k$" case. This means that there exist a constant $\lambda >0$, a sequence $(x_n)_n \in X$, maximal flats $F_n$ containing $x_n$, a sequence $d_n$ that goes to $+\infty$ and $u_1^n,\dots,u_k^n \in T_{x_n}F_n$ such that 
\\• $u_1^n,\dots,u_{k-1}^n$ are maximally singular,
\\• $\| u_1^n  \| = \varphi_1(d_n) = d_n$ and $\forall i = 2, \dots, k$, $\| u_i ^n \| \leq \varphi_i(d_n)$, and
$$     \textup{FillVol}_{k}^{Y, \mathrm{cr}}(f(P^{k-1}(x_n;u_1^n,\dots,u_{k}^n)))            \geq \lambda \prod_{i=1}^{k}  \varphi_i(d_n).         $$
The goal now is to find $\lambda'>0$ and for every $n \in \mathbb{N}$ : $x_n' \in F_n$ and maximally singular vectors ${u_1^{n}}',\dots,{u_{k}^{n}}' \in F_n$ such that  $\| {u_1^{n}} '  \| = d_n$ and $\forall i = 2, \dots, k$, $\| {u_i ^{n}} ' \| \leq \varphi_i(d_n)$, and they all tends to $+\infty$, that satisfy
\begin{equation*}
    \textup{FillVol}_{k}^{Y, \mathrm{cr}}(f(P^{k-1}(x_n ';{u_1^n} ',\dots,{u_k^n} ')))            \geq \lambda ' \prod_{i=1}^{k}  \varphi_i(d_n).    
\end{equation*}
Let $n \in \mathbb{N}$. If $u_{k}^n$ is maximally singular, there is nothing to do at this step. If not, take $F_n$ a maximal flat at $x_n$ that contains $u_1^n,\dots,u_k^n$, which is possible since rank $X = p \geq k+1$.
\\Let $C$ be a Weyl sector of $F_n$ with tip at $x_n$ that contains $u_k^n$, i.e.\ such that $x_n + u_k ^n \in \overline{C}$. Since $C$ is a simplicial open cone generated by maximally singular vectors, there exist $e_1,\dots,e_p \in F_n$ maximally singular vectors generating $C$ and $\delta_1,\dots,\delta_p \geq 0$ such that $u_k^n = \sum_{i=1}^p    \delta_i e_i$.      
\\Now let us consider the $(k-1)$-parallelogram $P^{k-1}(x_n;u_1^n,\dots,u_{k}^n)$ seen as a $(k-1)$-cycle in $F_n \simeq \mathbb{R}^p $. By the decomposition lemma \ref{decomp result}, we decompose our parallelogram along the walls of the Weyl sector:
\begin{equation*}
 P^{k-1}(x_n;u_1^n,\dots,u_{k}^n) = \sum_{i=1}^p          P^{k-1}(x_n+\sum_{s=1}^{i-1} \delta_s e_s;u_1^n,\dots,u_{k-1}^n,\delta_i e_i) + R_n^{k-1},   
\end{equation*}
where $R_n^{k-1}$ is a $(k-1)$-cycle in $F_n$, sum of $2(k-1)(p+1)$ $(k-1)$-chains, all of the form $C^{k-1}(y;w_1,\dots,w_{k-1})$, where $w_{k-1} \in \{u_k^n,\delta_1 e_1,\dots,\delta_p e_p \}$ and \\$(w_1,\dots,w_{k-2}) \in \{(u_1^n,\dots,\hat{u_j^n},\dots,u_{k-1}^n) \mid j \in \{1,\dots,k-1 \} \}$.
\\So every $(k-1)$-chain $\Sigma = C^{k-1}(y,w_1,\dots,w_{k-1})$ in $R_n^{k-1}$ satisfies 
\begin{equation*}
    \begin{split}
        \textup{Vol}_{F_n}^{k-1}(\Sigma) &\leq \| w_1  \|\cdots \| w_{k-2}  \|\cdot \| w_{k-1}  \|
        \\ &\leq \| u_1^n  \|\cdots\| u_{k-2}^n  \|\cdot\| u_{k}^n  \|
        \\& \leq \varphi_1(d_n) \times \dots \times \varphi_{k-2}(d_n) \times \varphi_k(d_n)
        \\& = \frac{\prod_{i=1}^{k}  \varphi_i(d_n)}{\varphi_{k-1}(d_n)}
    \end{split}
\end{equation*}
Therefore
$$   \textup{Vol}_{F_n}^{k-1}(R_n^{k-1}) \leq   2(k-1)(p+1) \,   \prod_{\substack{i=1 \\ i \ne k-1}}^{k}  \varphi_i(d_n).            $$
By applying $f$ to the parallelogram decomposition:
$$   f(P^{k-1}(x_n;u_1^n,\dots,u_{k}^n)) = \sum_{i=1}^p          f(P^{k-1}(x_n+\sum_{s=1}^{i-1} \delta_s e_s;u_1^n,\dots,u_{k-1}^n,\delta_i e_i)) + f(R_n^{k-1}   )  .                          $$
Since a filling of every cycle of the right-hand side gives a filling of the left-hand side, we have 
\begin{equation*}
    \begin{split}
       \textup{FillVol}_{k}^{Y, \mathrm{cr}}(f(P^{k-1}(x_n;u_1^n,\dots,u_{k}^n)))       & \leq \sum_{i=1}^p \textup{FillVol}_{k}^{Y, \mathrm{cr}} \big(          f(P^{k-1}(x_n+\sum_{s=1}^{i-1} \delta_s e_s;u_1^n,\dots,u_{k-1}^n,\delta_i e_i))\big) 
       \\ &+ \textup{FillVol}_{k}^{Y, \mathrm{cr}} \big( f(R_n^{k-1}   )  \big) .  
    \end{split}
\end{equation*}
However, $\textup{FillVol}_{k}^{Y, \mathrm{cr}} \big( f(R_n^{k-1}   )  \big)   $ is very small compared to $ \textup{FillVol}_{k}^{Y, \mathrm{cr}}(f(P^{k-1}(x_n;u_1^n,\dots,u_{k}^n)))$. Indeed, by considering the images by $f$ of fillings of $R_n^{k-1}$ in $F_n \subset X$, we have \begin{equation*}
    \begin{split}
        \textup{FillVol}_{k}^{Y, \mathrm{cr}} \big( f(R_n^{k-1}   )  \big) &\leq \textup{Lip}(f) \, \textup{FillVol}_{k}^{X, \mathrm{cr}} \big( R_n^{k-1}     \big)
        \\ & \leq \textup{Lip}(f) \, \textup{FillVol}_{k}^{F_n, \mathrm{cr}} \big( R_n^{k-1}     \big).
    \end{split}
\end{equation*}
By the Euclidean filling of $R_n^{k-1}$ in the flat $F_n \simeq \mathbb{R}^p$, we have
\begin{equation*}
    \begin{split}
        \textup{FillVol}_{k}^{F_n, \mathrm{cr}} \big( R_n^{k-1}     \big) & \leq \big( \textup{Vol}_{F_n}^{k-1}(R_n^{k-1}) \big)^{\frac{k}{k-1}}
        \\ & \leq \left( 2(k-1)(p+1) \,    \prod_{i \ne k-1}^{k}  \varphi_i(d_n) \right)^{\frac{k}{k-1}}.
    \end{split}
\end{equation*}
So 
\begin{equation*}
    \begin{split}
        \textup{FillVol}_{k}^{Y, \mathrm{cr}} \big( f(R_n^{k-1}   )  \big) &\leq \textup{Lip}(f) \, ( 2(k-1)(p+1))^{\frac{k}{k-1}} \,  \left(     \prod_{i \ne k-1}^{k}  \varphi_i(d_n)\right)^{\frac{k}{k-1}} .
    \end{split}
\end{equation*}
This is where the additional assumption on the functions is needed.
\\Indeed, since $\left( \prod_{i \ne k-1}^{k}  \varphi_i(d_n) \right)^{\frac{k}{k-1}} \ll \prod_{i=1}^{k}  \varphi_i(d_n) $, there exists $N \in \mathbb{N}$ such that for all $n \geq N$:
$$\lambda \prod_{i=1}^{k}  \varphi_i(d_n) -\textup{Lip}(f) \, [ 2(k-1)(p+1)]^{\frac{k}{k-1}} \,   \left(    \prod_{i \ne k-1}^{k}  \varphi_i(d_n)\right)^{\frac{k}{k-1}}  \geq \frac{\lambda}{2} \, \prod_{i=1}^{k}  \varphi_i(d_n).            $$
In particular, for $n \geq N$: 
$$   \textup{FillVol}_{k}^{Y, \mathrm{cr}}(f(P^{k-1}(x_n;u_1^n,\dots,u_{k}^n)))  -\textup{FillVol}_{k}^{Y, \mathrm{cr}} \big( f(R_n^{k-1}   )  \big) \geq   \frac{\lambda}{2} \, \prod_{i=1}^{k}  \varphi_i(d_n)        .          $$
i.e.\ for $n \geq N$: 
$$  \sum_{i=1}^p \textup{FillVol}_{k}^{Y, \mathrm{cr}} \big(          f(P^{k-1}(x_n+\sum_{s=1}^{i-1} \delta_s e_s;u_1^n,\dots,u_{k-1}^n,\delta_i e_i))\big) \geq    \frac{\lambda}{2} \, \prod_{i=1}^{k}  \varphi_i(d_n)    .                    $$
So there exists $i_0 \in \{1,\dots,p       \} $ such that
$$   \textup{FillVol}_{k}^{Y, \mathrm{cr}} \big(          f(P^{k-1}(x_n+\sum_{s=1}^{i_0-1} \delta_s e_s;u_1^n,\dots,u_{k-1}^n,\delta_{i_0} e_{i_0}))\big) \geq    \frac{\lambda}{2 p} \, \prod_{i=1}^{k}  \varphi_i(d_n)  .                        $$
Let us denote $x_n' = x_n+\sum_{s=1}^{i_0-1} \delta_s e_s$, ${u_1^n}' = u_1^n, \dots,{u_{k-1}^n}' = u_{k-1}^n, {u_k^n}' = \delta_{i_0} e_{i_0} $, which are all maximally singular.
\\Remark that $\forall i \in \{2,\dots,k-1       \} $, $\| {u_i^n}'  \| \to \infty $ with a speed comparable to that of $\varphi_i(d_n) $. Indeed, on one hand we have 
$$   \textup{FillVol}_{k}^{Y, \mathrm{cr}} \big(          f(P^{k-1}(x_n';{u_1^n}',\dots,{u_{k}^n}'))\big) \geq    \frac{\lambda}{2 p} \, \prod_{i=1}^{k}  \varphi_i(d_n)      .                  $$
Note also that
\begin{equation*}
    \begin{split}
        \textup{FillVol}_{k}^{Y, \mathrm{cr}} \big(          f(P^{k-1}(x_n';{u_1^n}',\dots{u_{k}^n}'))\big) 
        &\leq  \textup{Lip}(f) \, \textup{FillVol}_{k}^{X, \mathrm{cr}} \big(          P^{k-1}(x_n';{u_1^n}',\dots,{u_{k}^n}')\big) 
        \\ &\leq \textup{Lip}(f) \, \textup{FillVol}_{k}^{F_n, \mathrm{cr}} \big(          P^{k-1}(x_n';{u_1^n}',\dots,{u_{k}^n}')\big) 
        \\ &\leq \textup{Lip}(f) \,
        \| {u_1^n}'  \| \times \dots \times \| {u_k^n}'  \|.
    \end{split}
\end{equation*}
Therefore 
$$      \frac{\lambda}{2 p \, \textup{Lip}(f)} \, \prod_{i=1}^{k}  \varphi_i(d_n)                   \leq 
        \| {u_1^n}'  \| \times \dots \times \| {u_k^n}'  \|    .               $$
Since $ \| {u_1^n} '  \| = \varphi_1(d_n)  $ and $\forall i \in \{2,\dots,k      \} $, $ \| {u_i^n} '  \| \leq \varphi_i(d_n)  $, we get  $\forall i \in \{2,\dots,k       \} $:
$$      \frac{\lambda}{2 p \, \textup{Lip}(f)} \, \varphi_i(d_n)                      \leq  \| {u_i^n}'  \|  \leq \varphi_i(d_n)      .            $$
In particular, $  \| {u_k^n}'  \| \gg \varphi_{k+1}(d_n)          $.
\\
\\Now that we found the point $x_n' \in F_n$, the maximally singular vectors ${u_1^{n}}',\dots,{u_{k}^n}' \in F_n$, and the positive constant $\lambda' = \frac{\lambda}{2p}$ that satisfy the desired inequality, and to simplify the notations for the rest of the proof, we will just denote $x_n'$ by $x_n$, $\lambda'$ by $\lambda$ and ${u_{i}^{n}}'$ by $u_{i}^{n}$.
So that we have 
$$     \textup{FillVol}_{k}^{Y, \mathrm{cr}}(f(P^{k-1}(x_n;u_1^n,\dots,u_{k}^n)))            \geq \lambda  \prod_{i=1}^{k}  \varphi_i(d_n)  ,       $$
with $\| u_1^n  \| = \varphi_1(d_n)$ and $\forall i = 2, \dots, k$, $ \frac{\lambda}{\textup{Lip(f)}} \varphi_i(d_n) \leq \| u_i ^n \| \leq \varphi_i(d_n)$.
\\
\\Let $n \geq N$. Consider the singular $k$-flat $F_n' \subset F_n$ at $x_n$ generated by the maximally singular vectors $u_1^{n},\dots,u_{k}^{n}$, and let $P_X(F_n')$ be its parallel set. It splits metrically as
$$  P_X(F_n')  =    \mathbb{R}^k \times C_X(F_n') .         $$
Where the cross section $C_X(F_n')$, that we will denote $X_n'$, is either :
\\• a symmetric space of non-compact type of rank $\geq 1$, when $X$ is a symmetric space, 
\\• or a Euclidean building with bounded geometry of dimension $\geq 1$, with no Euclidean factor, when $X$ is a Euclidean building. 
\\By the same argument as before, if we fix $\varepsilon >0$, there exists $\mu >0$ such that for all $n \in \mathbb{N}$ and for all $R>0$
\begin{equation} \label{sections exp growth}
     \textup{Vol}^{\varepsilon} \left( B_{X_n'}(R)  \right) \geq \exp(\mu R)  .
\end{equation}
\\Now we will work in this product space and apply the same strategy as when $X$ is a product space. We have 
$$  P_X(F_n')  =    \mathbb{R}^k \times X_n' .         $$
Note that the slices $ \mathbb{R}^k \times \{x'\}$, with $x' \in X'$, are the flats that are parallel to $F_n'$. So the $k$-flat $F_n'$ is such a slice. Let $x_n' \in X'$ such that $F_n' =  \mathbb{R}^k \times \{x_n'\}$ .
\\For all $y \in B_{X_n'}\big( x_n',d_n^{\alpha_{k+1}}        \big) $, consider the geodesic $ [ x_n',y  ]$ in $X_n'$. Denote $\gamma_n'$ the (unique) bi-infinite geodesic extension of $ [ x_n',y  ]$, and $F_n'' = F_n' \times \gamma_n'$ the $(k+1)$-flat. Consider the $k$-parallelogram $ P^{k}(x_n;u_1^n,\dots,u_{k}^n,u_{k+1}^n) $ in $F_n''$, where $u_{k+1}^n$ is the unique vector in $T_{x_n}F_n''$ that is parallel to $\gamma_n'$ and satisfies $\| u_{k+1}^n  \| = d_{X_n'}(x_n',y)  $. This $k$-parallelogram satisfies the induction hypothesis: 
$$     \textup{FillVol}_{k+1}^{Y, \mathrm{cr}}(f(P^{k}(x_n;u_1^n,\dots,u_{k+1}^n)))            \leq \phi \left(\prod_{i=1}^{k+1}  \varphi_i(d_n)\right).         $$
Which implies that there exists $V_n \in \mathbf{I}_{k+1}(Y)$ in $Y$ such that $\partial V_n = f(P^{k}(x_n;u_1^n,\dots,u_{k+1}^n))$ and 
$$
\mass{V_n} \leq \phi \left(\prod_{i=1}^{k+1}  \varphi_i(d_n)\right).
$$
Now consider the 1-Lipschitz map $\pi : Y \to \mathbb{R}$, $\pi(z) = d_Y(z,f(C^{k}(x_n;u_1^n,\dots,u_{k}^n))) $.
\\Again, by the Slicing Theorem, we have that for a.e.\ $t \in \mathbb{R}$, there exists $\langle V_n,\pi,t \rangle \, \in \mathbf{I}_k(Y)$ such that $\slice{V_n}{\pi}{t}= \partial(V_n\rstr\{\pi \leq t\}) - (\partial V_n)\rstr\{\pi\leq t\}$, and by integrating the co-area formula over the distance $t$, we have that
$$ \int_0^{+\infty} \mass{\slice{V_n}{\pi}{t}} dt \leq \mass{V_n}.$$
Since $\mass{V_n} \leq \phi \left(\prod_{i=1}^{k+1}  \varphi_i(d_n)\right) $, we get
\begin{equation}\label{equation 8}
\int_0^{D} \mass{\slice{V_n}{\pi}{t}} dt \leq \int_0^{+\infty} \mass{\slice{V_n}{\pi}{t}} dt \leq \phi \left(\prod_{i=1}^{k+1}  \varphi_i(d_n)\right),
\end{equation}
where $D=d_Y\big(f(C^{k}(x_n;u_1^n,\dots,u_{k}^n)),f(C^{k}(x_n+u_{k+1}^n;u_1^n,\dots,u_{k}^n))\big)$.
\\However, for a.e.\ $t \in \left]0,D\right[$, $\mass{\slice{V_n}{\pi}{t}}$ cannot be too small since it gives a filling of the $(k-1)$-cycle $\partial f(C^{k}(x_n;u_1^n,\dots,u_{k}^n)) = f(P^{k-1}(x_n;u_1^n,\dots,u_{k}^n)) $.
\begin{claim}
For $n$ big enough: for a.e.\ $t \in \left]0,D\right[$, $\mass{\slice{V_n}{\pi}{t}} \geq \frac{\lambda}{2} \prod_{i=1}^{k}  \varphi_i(d_n)$.
\end{claim}
\begin{proof}
For every $t \in \left]0,D\right[$,
\begin{equation*}
    \begin{split}
       \partial(V_n) \rstr\{\pi \leq t\} &= \partial(V_n) \rstr\{\pi =0\} + \partial(V_n) \rstr\{ 0 < \pi \leq t\} 
       \\ &= f(P^{k-1}(x_n;u_1^n,\dots,u_{k}^n)) + H_t
    \end{split}
\end{equation*}
Where $H_t$ is the $k$-current $\partial (V_n) \rstr\{ 0 < \pi \leq t\}$.
\\Since $\|\slice{V_n}{\pi}{t}\|$ is concentrated on $\pi^{-1}(\{t\})$,
$$\partial(V_n\rstr\{\pi \leq t\}) = \slice{V_n}{\pi}{t} + f(P^{k-1}(x_n;u_1^n,\dots,u_{k}^n)) + H_t. $$
Which means that $\slice{V_n}{\pi}{t} + f(P^{k-1}(x_n;u_1^n,\dots,u_{k}^n)) + H_t$ is a $k$-current that is actually a cycle.
\\So $-(\slice{V_n}{\pi}{t} + H_t)$ is a $k$-chain that fills the $(k-1)$-cycle $f(P^{k-1}(x_n;u_1^n,\dots,u_{k}^n))$. Therefore,
$$\mass{ -(\slice{V_n}{\pi}{t} + H_t) } \geq \textup{FillVol}_{k}^{Y, \mathrm{cr}}(f(P^{k-1}(x;u_1,\dots,u_{k}))) \geq \lambda \prod_{i=1}^{k}  \varphi_i(d_n).   $$
So 
$$\mass{\slice{V_n}{\pi}{t}} + \mass{H_t} \geq \lambda \, \prod_{i=1}^{k}  \varphi_i(d_n).$$
Note that
\begin{equation*} 
\begin{split}
H_t & =  \partial (V_n) \rstr\{ 0 < \pi \leq t\} 
     \\ & =\big( \sum_{F \in \Delta} f(F ) \big) \rstr\{ 0 <\pi \leq t \},
\end{split}
\end{equation*}
where $\Delta$ is the set of side faces of $ P^{k}(x_n;u_1^n,\dots,u_{k+1}^n)$, i.e.\ faces whose vectors are not $(u_1^n,\dots,u_{k}^n)$, because $0< \pi < D$. Indeed, by lemma \ref{paralello sum}
\begin{equation*}
    \begin{split}
    P^{k}(x_n;u_1^n,\dots,u_{k+1}^n) &=  \sum_{s=1}^{k+1} (-1)^{s}   C^{k}(x;u_1^n, \dots,\hat{u_s^n},\dots ,u_{k+1}^n)
    \\ &+ (-1)^{k+1}  \sum_{k=1}^{k+1} (-1)^{s}  C^{k}(x+\sum_{i=1}^{k+1} u_i^n;-u_1^n \dots,-\hat{u_s^n},\dots, -u_{k+1}^n). 
    \end{split}
\end{equation*}
Therefore
$$\Delta = \{  C^{k}(x;u_1^n, \dots,\hat{u_s^n},\dots ,u_{k+1}^n),C^{k}(x+\sum_{i=1}^{k+1} u_i^n;-u_1^n \dots,-\hat{u_s^n},\dots, -u_{k+1}^n) \mid s\in \{ 1, \dots,k     \}           \}.$$
\\So by taking the mass
\begin{equation*}
\begin{split}
\mass{H_t} & = \mass{\big( \sum_{F \in \Delta} f(F ) \big) \rstr\{ 0 <\pi \leq t \} }  
\\ & \leq \sum_{F \in \Delta} \mass{ ( f(F) ) \rstr\{ 0<\pi \leq t \}}      \\
& \leq \sum_{F \in \Delta} \mass{  f(F) }    \\
& \leq \textup{Lip}(f) \sum_{F \in \Delta} \mass{ F }.
\end{split}
\end{equation*}
Since every side face $F$ satisfies $\mass{ F } \leq \frac{\prod_{i=1}^{k+1}  \varphi_i(d_n)}{\varphi_{s}(d_n)}$, where $s \in \{ 1,\dots,k\}$, so every $F \in \Delta$ satisfies $\mass{F} \leq \frac{\prod_{i=1}^{k+1}  \varphi_i(d_n)}{\varphi_{k}(d_n)} $. And there are $2k$ side faces,
so
$$\mass{H_t} \leq   2 k \textup{Lip}(f) \prod_{i \ne k}^{k+1}  \varphi_i(d_n) . $$
So 
$$\mass{\slice{V_n}{\pi}{t}} \geq \lambda \, \prod_{i=1}^{k}  \varphi_i(d_n) - \mass{H_t} \geq \lambda \, \prod_{i=1}^{k}  \varphi_i(d_n) - 2 k \textup{Lip}(f) \prod_{i \ne k}^{k+1}  \varphi_i(d_n) . $$
Since for all $i \,$,  $\varphi_i \gg \varphi_{i+1} $, so $\prod_{i=1}^{k}  \varphi_i(d_n) \gg \prod_{i \ne k}^{k+1}  \varphi_i(d_n) $. There exists $N \in \mathbb{N}$ such that for all $n\geq N$,
$$\lambda \, \prod_{i=1}^{k}  \varphi_i(d_n) - 2 k \textup{Lip}(f) \prod_{i \ne k}^{k+1}  \varphi_i(d_n) \geq \frac{\lambda}{2} \prod_{i=1}^{k}  \varphi_i(d_n). $$
We conclude that for all $n\geq N$ and for a.e.\ $t \in \left]0,D\right[$, $$\mass{\slice{V_n}{\pi}{t}} \geq \frac{\lambda}{2} \prod_{i=1}^{k}  \varphi_i(d_n). \qquad\qedhere$$
\end{proof}
Therefore, by \eqref{equation 8}, we have 
\begin{equation*}
\begin{split}
    \phi \left(\prod_{i=1}^{k+1}  \varphi_i(d_n)\right) & \geq \int_0^{D} \mass{\slice{V_n}{\pi}{t}} dt\\ 
& \geq D \,  \frac{\lambda}{2} \prod_{i=1}^{k}  \varphi_i(d_n).
\end{split}
\end{equation*}
This implies that
\begin{equation*}
 D \leq \frac{2}{\lambda}  \frac{\phi \left(\prod_{i=1}^{k+1}  \varphi_i(d_n)\right)}{\prod_{i=1}^{k}  \varphi_i(d_n)} 
\end{equation*}
Let us denote $\psi(\varphi_{k+1}(d_n)) = \frac{2}{\lambda}  \frac{\phi \left(\prod_{i=1}^{k+1}  \varphi_i(d_n)\right)}{\prod_{i=1}^{k}  \varphi_i(d_n)}$. Note that $\psi$ is sublinear: $\frac{\psi(\varphi_{k+1}(d_n))}{\varphi_{k+1}(d_n)}$ tends to $0$.
\\
\\Since $D=d_Y\big(f(C^{k}(x_n;u_1^n,\dots,u_{k}^n)),f(C^{k}(x_n+u_{k+1}^n;u_1^n,\dots,u_{k}^n))\big)$, the last inequality implies that there exists $z \in C^{k}(x_n+u_{k+1}^n;u_1^n,\dots,u_{k}^n) $ such that 
$$d_Y(f(C^{k}(x_n;u_1^n,\dots,u_{k}^n)),f(z)) \leq \psi(\varphi_{k+1}(d_n)).$$
But $z \in C^{k}(x_n+u_{k+1}^n;u_1^n,\dots,u_{k}^n) $ implies that $\textup{proj}_{X_n'}(z) = y$.
\\
\\If we choose another $y \in B_{X_n'}\big( x_n',\varphi_{k+1}(d_n)        \big)$, we get another $z \in C^{k}(x_n+u_{k+1}^n;u_1^n,\dots,u_{k}^n) $ such that 
$$\textup{proj}_{X_n'}(z) = y,$$ 
$$d_Y(f(C^{k}(x_n;u_1^n,\dots,u_{k}^n)),f(z)) \leq \psi(\varphi_{k+1}(d_n)).$$
\\By doing this process for all $ y \in B_{X_n'}\big( x_n',\varphi_{k+1}(d_n)        \big) $, we get subsets $C_n \subset X$ that projects onto $B_{X_n'}\big( x_n',\varphi_{k+1}(d_n)       \big)$, i.e.\ 
\begin{equation*}
    B_{X_n'}\big( x_n',\varphi_{k+1}(d_n)        \big) \subset \textup{proj}_{X_n'}(C_n)  ,
\end{equation*} 
and such that
$$ f(C_n) \subset N_{\psi\left(\varphi_{k+1}(d_n)\right)}   (f(C^{k}(x_n;u_1^n,\dots,u_{k}^n))).$$
\\Since the projection onto $X_n'$ is 1-Lipschitz, if we fix $\varepsilon >0$
\begin{equation*} 
     \textup{Vol}_X^{\varepsilon}(B_{X_n'}\big( x_n',\varphi_{k+1}(d_n)        \big) )  
     \leq \textup{Vol}_X^{\varepsilon}(\textup{proj}_{X_n'}(C_n)  )   \leq \textup{Vol}_X^{\varepsilon}(C_n).
\end{equation*}
By the uniform lower bound on the volume growth of all such cross sections (\ref{sections exp growth}), we get
\begin{equation*}
  \textup{exp}(\mu \varphi_{k+1}(d_n)) \leq  \textup{Vol}_X^{\varepsilon}( B_{X_n'}\big( x_n',\varphi_{k+1}(d_n)        \big))  \leq  \textup{Vol}_X^{\varepsilon}(C_n) .
\end{equation*}
On the other hand, we get
\begin{equation*} 
    \textup{Vol}_Y^{\varepsilon}(f(C_n)) \leq \textup{Vol}_Y^{\varepsilon}(N_{\psi\left(\varphi_{k+1}(d_n)\right)}   (f(C^{k}(x_n;u_1^n,\dots,u_{k}^n)))).
\end{equation*}
$f$ coarsely preserves volumes, i.e.\ there exist $\delta, \delta' >0$ such that 
$$  \delta \, \textup{Vol}_X^{\varepsilon} ( C_n )  \leq      \textup{Vol}_Y^{\varepsilon} ( f( C_n ) ) \leq \delta'  \, \textup{Vol}_X^{\varepsilon} ( C_n )      $$
By lemma \ref{volume neighb}, we have
\begin{equation*}
   \textup{Vol}_Y^{\varepsilon}(N_{\psi\left(\varphi_{k+1}(d_n)\right)}   (f(C^{k}(x_n;u_1^n,\dots,u_{k}^n)))) \leq \beta_Y^{\varepsilon}\left( \varepsilon + \psi(\varphi_{k+1}(d_n) \right) \times  \textup{Vol}_Y^{\varepsilon}(   f(C^{k}(x_n;u_1^n,\dots,u_{k}^n))).
\end{equation*}
$Y$ has at most exponential growth, so there exists $\beta >0$ such that $\forall R>0$
$$  \beta_Y^{\varepsilon}\left( R \right) \leq      e^{\beta R}  .  $$
In particular, we have on hand that
$$\beta_Y^{\varepsilon}\left( \varepsilon + \psi(\varphi_{k+1}(d_n) \right)   \leq    \textup{exp}( 2 \beta      \psi ( \varphi_{k+1}(d_n) )            ) . $$ 
On the other hand, by taking a partition of each side vector into sub-intervals of length $\varepsilon$, we get a partition of the $k$-parallelepiped $C^{k}(x_n;u_1^n,\dots,u_{k}^n)$ :
$$ \textup{Vol}_X^{\varepsilon} \big( C^{k}(x_n;u_1^n,\dots,u_{k}^n)\big)         \leq \prod_{i=1}^{k} \big( \frac{ \| u_i^n  \|}{\varepsilon} +1    \big) \leq  \prod_{i=1}^{k} \big( \frac{ 2\| u_i^n  \|}{\varepsilon}    \big)  \leq \big(2/\varepsilon \big)^{k} \prod_{i=1}^{k}  \varphi_i(d_n)   .              $$
So, by denoting $A =  \big(2/\varepsilon \big)^{k} $, we have
\begin{equation*}
    \begin{split}
        \textup{Vol}_Y^{\varepsilon}(   f(C^{k}(x_n;u_1^n,\dots,u_{k}^n)))
         &  \leq \delta' \, \textup{Vol}_X^{\varepsilon} \big( C^{k}(x_n;u_1^n,\dots,u_{k}^n)\big)   
         \\ &\leq \delta' A \, \prod_{i=1}^{k}  \varphi_i(d_n) .
    \end{split}
\end{equation*}
Therefore
\begin{equation*}
    \textup{Vol}_Y^{\varepsilon}(N_{\psi\left(\varphi_{k+1}(d_n)\right)}   (f(C^{k}(x_n;u_1^n,\dots,u_{k}^n))))  \leq \textup{exp}( 2 \beta      \psi ( \varphi_{k+1}(d_n))   ) \times \delta' A \, \prod_{i=1}^{k}  \varphi_i(d_n).
\end{equation*}
We conclude from all the previous inequalities that 
\begin{equation*}
   \delta \, \textup{exp}(\mu d_n^{\alpha_{k+1}}) \leq \delta \, \textup{Vol}_X^{\varepsilon} ( C_n ) \leq  \textup{Vol}_Y^{\varepsilon} ( f( C_n ) ) \leq \textup{exp}( 2 \beta      \psi ( \varphi_{k+1}(d_n))   ) \times \delta' A \, \prod_{i=1}^{k}  \varphi_i(d_n) .
\end{equation*}
Which implies finally that for all $n\geq N$
\begin{equation*}
   \delta \, \textup{exp}(\mu \varphi_{k+1}(d_n)) \leq \textup{exp}( 2 \beta      \psi (\varphi_{k+1}(d_n))   ) \times \delta' A \, \prod_{i=1}^{k}  \varphi_i(d_n) .
\end{equation*}
Which is not possible when $d_n \to \infty$ because $\psi$ is sublinear. This completes the induction.
\end{proof}
Similarly, let us show that Theorem \ref{Thm 2 v2} follows from Theorem \ref{CE of model space}. 
\begin{proof}
Let $X= S \times B$ be a model space of rank $k\geq 2$, and let $Y$ be a Lipschitz-connected complete metric space with at most exponential growth such that $\textup{FV}_{k}^{Y}(\ell) = o\left(\ell^{\frac{k}{k-1}}\right)$. Let us show that there exist $\varphi_1,\varphi_2,\dots,\varphi_k$ functions as in Theorem \ref{CE of model space} that satisfies: $\forall x \in X$, for every $F$ maximal flat that contains $x$, $\forall d >0$, $\forall u_1,\dots,u_r \in T_x F \simeq F$ that satisfy
\\• $u_1,\dots,u_{r-1}$ are maximally singular,
\\• $\| u_1  \| = d$ and $\forall i = 2, \dots, r$, $\| u_i  \| \leq \varphi_i(d)$,
\\we have 
$$     \textup{FillVol}_{r}^{Y, \mathrm{cr}}(f(P^{r-1}(x;u_1,\dots,u_r)))            \ll \prod_{i=1}^{r}  \varphi_i(d).         $$
To do so, we can apply the same strategy to get $\varphi_1,\varphi_2,\dots,\varphi_k$ that satisfy the first three conditions, i.e.\ by taking for all $i = 1 \dots k$, $\varphi_i(d) = a(d)^{1-\frac{1}{i}}d$, where $a(d)$ is defined as in the proof of Theorem \ref{Thm 1 v2}. However, this sequence do not satisfy the fourth condition. Let us replace the sequence $((\frac{1}{i}))_i$ by a sequence $(\beta_i)_i$, i.e.\ $\varphi_i(d) = a(d)^{1-\beta_i}d$. To get the first three conditions, this sequence should satisfy $\beta_1 = 1$, it should be decreasing, and $\beta_i>0$ for all $i$. Let us do the computations to see what condition on this sequence does the fourth one imply. Let us denote $S_n = \sum_{i=1}^n \beta_i$. We have for all $p=2, \dots , k$, 
$$\left(\prod_{i\ne p-1}^{p}  \varphi_i(d) \right)^{\frac{p}{p-1}} \ll \prod_{i=1}^{p}  \varphi_i(d). $$
It implies that for all $p=2, \dots , k$, 
$$\left( \frac{d^{p-1}a(d)^{p-S_p} }{a(d)^{1-\beta_{p-1}}} \right)^{\frac{p}{p-1}} \ll d^p a(d)^{p-S_p} . $$
Hence 
$$ d^{p} \left( a(d)^{(p-1)-(S_p-\beta_{p-1})} \right)^{\frac{p}{p-1}} \ll d^p a(d)^{p-S_p} . $$
Therefore 
$$  a(d)^{p-{\frac{p}{p-1}}((S_p-\beta_{p-1}))}  \ll a(d)^{p-S_p} . $$
Since $a(d)$ tends to zero, this implies that for all $p=2, \dots , k$,
$$  {\frac{p}{p-1}}(S_p-\beta_{p-1}) < S_p.          $$
Equivalently, for all $p=2, \dots , k$,
$$  \frac{S_p}{p}< \beta_{p-1}.          $$
We can construct arbitrarily long finite sequences that satisfy the required properties. We will use the condition ${\frac{p}{p-1}}(S_p-\beta_{p-1}) < S_p$ to define our sequence by induction.
\\Let us consider the sequence $(\beta_n)_{n \in \mathbb{N^*}}$ defined by induction : 
$$ \beta_1 \in \mathbb{R}, \quad \textup{and} \quad \forall n \in \mathbb{N^*}, \quad \beta_{n+1} = (n+1) \beta_n - S_n - 1                   , $$
where $S_n = \sum_{k=1}^n \beta_k$. For all $n \in \mathbb{N^*}$
$$ \beta_{n+1}-\beta_{n} = (n+1)\beta_{n} - S_n -1 - n\beta_{n-1} +S_{n-1} +1 = n (\beta_{n}-\beta_{n-1}  )    .            $$
So 
$$ \beta_{n+1}-\beta_{n} = (\beta_{2}-\beta_{1}) \, n!          $$
Therefore
$$  \beta_{n} = \beta_{1} - (\beta_{1}-\beta_{2})  \sum_{k=1}^{n-1} k!       $$
Since $ \beta_{2} = \beta_{1} -  1         $, we get
$$  \beta_{n} = \beta_{1} - \sum_{k=1}^{n-1} k!       $$
So $(\beta_n)_{n \in \mathbb{N^*}}$ is strictly decreasing and tends to $-\infty$. For any $k\in \mathbb{N}^*$, there exists $\beta_1 \in \mathbb{N}$ such that the first $k$-terms of the sequence are positive. Up to re-normalizing, we can suppose that $\beta_1 = 1$. We conclude that $\beta_1, \dots , \beta_k$ satisfy the desired conditions and that the functions $\varphi_i(d) = a(d)^{1-\beta_i}d$, for $i=1, \dots , k$, satisfy the conditions of Theorem \ref{CE of model space}.\qquad
\end{proof} 
\begin{rem}
It turns out that such an infinite sequence does not exist. We thank Mingkun Liu for having pointed it out to us. That is why, unlike in Theorem \ref{CE of product}, we did not ask for an infinite sequence of functions $(\varphi_i)_i$, but only a finite one. Indeed, such infinite sequence $(\beta_n)_n$ must satisfy for all $n \in \mathbb{N}^*$
$$  (\beta_1-\beta_n)+ \dots + (\beta_{n-1}-\beta_n)+0+(\beta_{n+1}-\beta_n) <0             .  $$
Since the sequence is decreasing, for all $n \geq 2$
$$       (\beta_{n-1}-\beta_n)+(\beta_{n+1}-\beta_n) <0  .            $$
This implies that the sequence $(\beta_{n}-\beta_{n+1})_n$ is increasing, thus for all $n \geq 2$
$$ \beta_{1}-\beta_{2} < \beta_{n}-\beta_{n+1} .           $$
This is not possible since $\beta_{1}-\beta_{2} >0$, and $(\beta_{n}-\beta_{n+1})_n$ converges to $0$.
\end{rem}
\section{The domain has a one-dimensional Euclidean factor}
The goal of this section is to prove Theorem \ref{Thm 3 v2}, which implies Theorem \ref{Thm 3 v1}, and Theorem \ref{Thm 4}.
\subsection{Coarse embeddings of Euclidean spaces into lower rank}
Both theorems follow from the following one.
\begin{thm} \label{CE of euclidean spaces}
Let $p \geq k \geq 2$ be integers. Let $Y$ be a complete Lipschitz-connected metric space with linear $k$-dimensional filling function $\textup{FV}_k^{Y}(\ell)\sim \ell$, and let $f : \mathbb{R}^p \to Y $ be a coarse embedding. Then for any $x \in \mathbb{R}^p$, and every linearly independent vectors $u_1,\dots,u_{k-1} \in \mathbb{R}^p$,
$$  \frac{\textup{FillVol}_{k-1}^{Y, \mathrm{cr}}(f(P^{k-2}(x;u_1,\dots,u_{k-1})))   }{\| u_1  \| \times \dots \times \| u_{k-1}  \|}           \to 0, \textup{when min$\{ \| u_{i}  \| \} $ tends to $+\infty$ .} $$
\end{thm}
\begin{proof}
Let $C>0$ such that for any $(k-1)$-cycle $\Sigma$ we have  $\textup{FillVol}_{k}^{Y, \mathrm{cr}}(\Sigma) \leq C \, \textup{Vol}_Y^{\varepsilon}(\Sigma) $.
\\Let $x\in\mathbb{R}$, and $u_1,\dots,u_{k-1} \in \mathbb{R}^p$ linearly independent vectors. Without loss of generality, we can assume that $\| u_1  \| \geq  \dots \geq \| u_{k-1}  \|$. Let $u_k \in \mathbb{R}^p$ be a vector orthogonal to $\textup{span}\{u_1,\dots,u_{k-1}\}$ such that $\| u_k  \| = \| u_{k-1}  \|^{1/2}  $. Then the $(k-1)$-paralellogram $P^{k-1}(x;u_1,\dots,u_k)$ is a $(k-1)$-cycle. Therefore 
\begin{equation*}
    \begin{split}
        \textup{FillVol}_{k}^{Y, \mathrm{cr}}(f(P^{k-1}(x;u_1,\dots,u_{k}))) 
        &\leq C \, \textup{Vol}_{k-1}^{Y}(f(P^{k-1}(x;u_1,\dots,u_{k}))   )
        \\&  \leq  C \, \textup{Lip}(f) \, \textup{Vol}_{k-1}^{\mathbb{R}^p}(P^{k-1}(x;u_1,\dots,u_{k})   )
        \\&  \leq  C \, \textup{Lip}(f) \, 2(k+1) \, \| u_1  \| \times \dots \times \| u_{k-1}  \|,
    \end{split}
\end{equation*}
because $P^{k-1}(x;u_1,\dots,u_{k})$ contains $2(k+1)$ faces, and their volume is at most $\| u_1  \| \times \dots \times \| u_{k-1}  \|$ since $\| u_1  \| \geq  \dots \geq \| u_{k-1}  \|$.
\\Let $\Omega \in \mathbf{I}_{k}(Y)$ be a $k$-current in $Y$ such that $\partial \Omega = f(P^{k-1}(x_n;u_1,\dots,u_{k}))$ and 
$$
\mass{\Omega} \leq  C \, \textup{Lip}(f) \, 2(k+1) \, \| u_1  \| \times \dots \times \| u_{k-1}  \|.
$$
Consider the
$1$-Lipschitz map $\pi : Y \to \mathbb{R}$, $\pi(z) = d_Y(z,f(C^{k-1}(x;u_1,\dots,u_{k-1}))) $.
\\ By the Slicing Theorem, we have that for a.e.\ $t \in \mathbb{R}$, there exists $\langle \Omega,\pi,t \rangle \, \in \mathbf{I}_{k-1}(Y)$ such that $\slice{\Omega}{\pi}{t}= \partial(\Omega\rstr\{\pi \leq t\}) - (\partial \Omega)\rstr\{\pi\leq t\}$.
\\By integrating the co-area formula over the distance $t$, we have that
$$ \int_0^{+\infty} \mass{\slice{\Omega}{\pi}{t}} dt \leq \mass{\Omega}.$$
So
\begin{equation*}
\int_0^{D} \mass{\slice{\Omega}{\pi}{t}} dt \leq \int_0^{+\infty} \mass{\slice{\Omega}{\pi}{t}} dt \leq C \, \textup{Lip}(f)  \, 2(k+1) \, \| u_1  \| \times \dots \times \| u_{k-1}  \|.
\end{equation*}
Where $D=d_Y\big(f(C^{k-1}(x;u_1,\dots,u_{k-1})),f(C^{k-1}(x+u_k;u_1,\dots,u_{k-1}))\big)$. However:
$$D \geq \rho^{-}\left(d_{\mathbb{R}^p}( C^{k-1}(x;u_1,\dots,u_{k-1}),  C^{k-1}(x+u_k;u_1,\dots,u_{k-1})       )           \right).$$
$u_k$ is orthogonal to $\textup{span}\{u_1,\dots,u_{k-1}\}$, so $$d_{\mathbb{R}^p}( C^{k-1}(x;u_1,\dots,u_{k-1}),  C^{k-1}(x+u_k;u_1,\dots,u_{k-1})       )  =  \| u_{k}  \| = \| u_{k-1}  \|^{1/2}.    $$
Therefore $ D \geq \rho^{-}\left( \| u_{k-1}  \|^{1/2}  \right) $. Let us denote $D' = \rho^{-}\left( \| u_{k-1}  \|^{1/2} \right)$. So
\begin{equation*}
\int_0^{D'} \mass{\slice{\Omega}{\pi}{t}} dt \leq C \, \textup{Lip}(f)  \, 2(k+1) \, \| u_1  \| \times \dots \times \| u_{k-1}  \|.
\end{equation*}
So there exists $t_0 \in ]0,D'[$ that satisfies the Slicing Theorem and such that 
\begin{equation*}
\frac{D'}{2} \mass{\slice{\Omega}{\pi}{t_0}} \leq C \, \textup{Lip}(f)  \, 2(k+1) \, \| u_1  \| \times \dots \times \| u_{k-1}  \|.
\end{equation*}
However $\mass{\slice{\Omega}{\pi}{t_0}}$ almost gives a filling of the basis $f(P^{k-2}(x;u_1,\dots,u_{k-1}))$. Indeed, by (1) of the Slicing Theorem \ref{slicing thm}, we have
$$\partial(\Omega\rstr\{\pi \leq t_0\}) = \slice{\Omega}{\pi}{t_0} +f(P^{k-2}(x;u_1,\dots,u_{k-1})) + H_{t_0}, $$
where $H_{t_0}$ is the $(k-1)$-current $\partial (\Omega) \rstr\{ 0 < \pi < t_0\}$.
This means that $-(\slice{\Omega}{\pi}{t_0} + H_{t_0})$ is a $(k-1)$-chain that fills $f(P^{k-2}(x;u_1,\dots,u_{k-1}))$. Therefore,
\begin{equation*}
    \begin{split}
        \textup{FillVol}_{k-1}^{Y, \mathrm{cr}}(f(P^{k-2}(x;u_1,\dots,u_{k-1})))  & \leq     \mass{ -(\slice{\Omega}{\pi}{t_0} + H_t) }
        \\ &\leq  \mass{ (\slice{\Omega}{\pi}{t_0 }} + \mass{H_{t_0}}.
    \end{split}
\end{equation*}
Note that
\begin{equation*} 
\begin{split}
H_{t_0 } & =  \partial (\Omega) \rstr\{ 0 < \pi < {t_0 }\} 
     \\ & =\big( \sum_{F \in \Delta} f(F ) \big) \rstr\{ 0 <\pi < {t_0 } \} ,
\end{split}
\end{equation*}
where $\Delta$ is the set of side faces of $ P^{k-1}(x_n;u_1^n,\dots,u_{k}^n)$, i.e.\ faces whose vectors are not $(u_1^n,\dots,u_{k-1}^n)$, because $0< \pi < D$. Indeed, by lemma \ref{paralello sum}
\begin{equation*}
    \begin{split}
    P^{k-1}(x_n;u_1^n,\dots,u_{k}^n) &=  \sum_{s=1}^{k} (-1)^{s}   C^{k-1}(x;u_1^n, \dots,\hat{u_s^n},\dots ,u_{k}^n)
    \\ &+ (-1)^{k}  \sum_{s=1}^{k} (-1)^{s}  C^{k-1}(x+\sum_{i=1}^{k} u_i^n;-u_1^n \dots,-\hat{u_s^n},\dots, -u_{k}^n). 
    \end{split}
\end{equation*}
So $\Delta = \{  C^{k-1}(x;u_1^n, \dots,\hat{u_s^n},\dots ,u_{k}^n),C^{k-1}(x+\sum_{i=1}^{k} u_i^n;-u_1^n \dots,-\hat{u_s^n},\dots, -u_{k+1}^n) \mid s\in \{ 1, \dots,k-1     \}           \}$.
\\By taking the mass
\begin{equation*} 
\begin{split}
\mass{H_{t_0}} & = \mass{\big( \sum_{F \in \Delta} f(F ) \big) \rstr\{ 0 <\pi \leq t_0 \} }  
\\ & \leq \sum_{F \in \Delta} \mass{ ( f(F) ) \rstr\{ 0<\pi \leq t_0 \}}      \\
& \leq \sum_{F \in \Delta} \mass{  f(F) }    \\
& \leq \textup{Lip}(f) \sum_{F \in \Delta} \mass{ F }.
\end{split}
\end{equation*}
Since every side face $F$ satisfies $\mass{ F } \leq \frac{\| u_1  \| \times \dots \times \| u_{k}  \|}{\| u_{s}  \|}$, where $s \in \{ 1,\dots,k-1\}$, every $F \in \Delta$ satisfies $\mass{ F } \leq \| u_1  \| \times \dots \times \| u_{k-2}  \|\| u_{k}  \|$. There are $2k$ side faces,
So
$$\mass{H_{t_0}} \leq   2 k \textup{Lip}(f)  \| u_1  \| \times \dots \times \| u_{k-2}  \|\| u_{k}  \|. $$
So
$$ \mass{ (\slice{\Omega}{\pi}{t_0 }} + \mass{H_{t_0}} \leq \| u_1  \| \times \dots \times \| u_{k-1}  \| \Big( \frac{2C \, \textup{Lip}(f)  \, 2(k+1)}{D'}  + \frac{2 k \, \textup{Lip}(f)}{\| u_{k-1}  \|^{1/2}}         \Big) .  $$
Therefore
$$  \frac{\textup{FillVol}_{k-1}^{Y, \mathrm{cr}}(f(P^{k-2}(x;u_1,\dots,u_{k-1})))   }{\| u_1  \| \times \dots \times \| u_{k-1}  \|}   \leq \Big( \frac{2C \, \textup{Lip}(f)  \, 2(k+1)}{\rho^{-}\left( \| u_{k-1}  \|^{1/2} \right)}  + \frac{2 k \, \textup{Lip}(f)}{\| u_{k-1}  \|^{1/2}}         \Big) .$$
and the right hand side clearly tends to $+\infty$ when $\| u_{k-1}  \|$ tends to $+\infty$.
\end{proof}
\subsection{Proof of Theorem \ref{Thm 3 v2}}
This theorem is an immediate consequence of Theorem \ref{CE of euclidean spaces}.
\begin{proof}
Let $k\in \mathbb{N}$, $X,Y$ be as in Theorem \ref{Thm 3 v2}, and let $f : X \times \mathbb{R} \to Y$ be a coarse embedding. 
\\Let $\varphi_1,\varphi_2,\dots,\varphi_k$ be functions as in Theorem \ref{CE of model space} (therefore as in Theorem \ref{CE of product}). Let $x\in X$ and let $F$ be a maximal flat in $X$ containing $x$. Consider the maximal flat $E = F \times \mathbb{R}$ in $X \times \mathbb{R}$. By Theorem \ref{CE of euclidean spaces}, there exists a sublinear function $\phi$ such that for all $d>0$, for all $u_1,\dots,u_k$ linearly independent vectors in $T_{x}E \simeq E$ such that $ \| u_i  \| \leq \varphi_i(d) $,
$$     \textup{FillVol}_{k}^{Y, \mathrm{cr}}(f(P^{k-1}(x;u_1,\dots,u_k)))            \leq \phi \left(\prod_{i=1}^{k}  \varphi_i(d)\right) .     $$
In particular, $u_1,\dots,u_k$ can be chosen in $T_xF \subset T_xE$. Therefore, when restricted to a copy $X \times \{ z \}$ for some $z \in \mathbb{R}$, the coarse embedding $ \Tilde{f}:= f\restriction_X : X \to Y $ satisfies the condition of Theorem \ref{CE of model space}, or Theorem \ref{CE of product}, which is not possible. Therefore, a coarse embedding $f : X \times \mathbb{R} \to Y$ cannot exist.
\end{proof}
\subsection{Proof of Theorem \ref{Thm 4}}
This is another consequence of Theorem \ref{CE of euclidean spaces}.
\begin{proof}
Let us suppose that there exists a subspace $E \simeq \mathbb{R}^{k} \subset \mathbb{R}^{p} $ that is sent quasi-isometrically with constants $(\lambda,c)$. By Lemma \ref{connect the dots}, we may assume that $f$ is $\lambda$-Lipschitz. Let us also denote by $f$ its restriction to $E$. So $f(E)$ is a maximal quasi-flat in $Y$. By the quasi-flats Theorem in \cite{kleiner1997rigidity} and \cite{eskin1997quasi}, there exist $\delta>0$ and maximal flats $F_1,\dots,F_r$ in $Y$ such that $f(E) \subset N_{\delta}\left(F_1 \cup \dots \cup F_r     \right) $. We can assume that the union $F_1 \cup \dots \cup F_r $ is minimal in the sens that for all $i \in \{1,\dots,r  \} $, $f(E)$ is not in a bounded neighborhood of $F_1 \cup \dots \cup F_r \backslash F_i$ (elsewhere we can just remove $F_i$ and modify $\delta$). Therefore by minimality of this union, there exists a sequence $(x_n)_n \in E$ such that $d_Y\left(f(x_n),F_1 \cup \dots \cup F_{r-1}\right) \geq n^2 $, i.e.\ $(f(x_n))_n$ only stays in the $\delta$-neighborhood of $F_r$.
\\Let $z\in B_E(x_n,n)$, so $d_Y(f(z),f(x_n)) \leq \lambda \, n $. On one hand we have 
$$ d_Y\big(f(x_n),F_1 \cup \dots \cup F_{r-1} \big) \leq d_Y(f(z),f(x_n)) + d_Y\big(f(z),F_1 \cup \dots \cup F_{r-1} \big)   .                 $$
So, for $n$ big enough we have
$$ \delta < n^2 - \lambda n \leq  d_Y\big(f(z),F_1 \cup \dots \cup F_{r-1} \big) .           $$
On the other hand $f(z) \in N_{\delta}\left(F_1 \cup \dots \cup F_r     \right) $. So for $n$ big enough, we have $f(z) \in N_{\delta}(F_r) $. Thus there exists $N \in \mathbb{N}$ such that for all $n\geq N$, $f \left(B_E(x_n,n)\right)\subset N_{\delta}(F_r)$.
\\Let us consider the projection map $\pi : Y \to F_r$, which is $1$-Lipschitz since $Y$ is CAT(0) and $F_r$ is a closed convex subset. Let us denote $A := \cup_{n\geq N}B_E(x_n,n) $ and
consider the map $g : A \to F_r$, $g(a) := \pi \circ f (a).$ It is a quasi-isometric embedding. Indeed, it is clearly $\lambda$-Lipschitz, and for $x,y \in A$ we have
\begin{equation*}
    \begin{split}
  d_Y(f(x),f(y)) &\leq d_Y(f(x),\pi(f(x))) + d_Y(\pi(f(x)),\pi(f(y))) + d_Y(\pi(f(y)),f(y)))
        \\& \leq d_{F_r}(\pi(f(x)),\pi(f(y))) + 2 \delta.
    \end{split}
\end{equation*}
Therefore, for all $x,y \in A $
$$ \lambda \, d_A(x,y)-c - 2 \delta \leq d_{F_r}(\pi(f(x)),\pi(f(y))) \leq \lambda \, d_A(f(x),f(y)).  $$
Let us denote $c' = c+2\delta$. 
\\
\\For all $n \geq N$, let $u_1^n,\dots,u_{k}^n \in E $ be orthogonal vectors such that $\| u_i  \| = n^{1/i}$, and consider the $(k-1)$-parallelograms $P^{k-1}(x_n;u_1^n,\dots,u_{k}^n) $ in $E$, which are actually boundaries of $k$-dimensional rectangles in this case. Note that for all $n\geq N$, $P^{k-1}(x_n;u_1^n,\dots,u_{k}^n) $ is in $B_E(x_n,n)$, so all the parallelograms are in $A$. Moreover, by the previous theorem, there exists a sublinear function $\phi$ such that for all $n\geq N$
$$     \textup{FillVol}_{k}^{Y, \mathrm{cr}}(f(P^{k-1}(x_n;u_1^n,\dots,u_k^n)))            \leq \phi (n^{H_k}) ,     $$
where $H_k = \sum_{i=1}^k \frac{1}{i} $. $\pi$ is 1-Lipschitz so
$$     \textup{FillVol}_{k}^{F_r, \mathrm{cr}}(g(P^{k-1}(x_n;u_1^n,\dots,u_k^n)))            \leq \phi (n^{H_k}) .     $$
Let $V_n \in \textbf{I}_k(F_r)$ such that $\partial V_n = g(P^{k-1}(x_n;u_1^n,\dots,u_k^n)) $ and $\mass{V} \leq  \phi (n^{H_k})  $, and consider the 1-Lipschitz map $d_1 : F_r \to \mathbb{R}$, $d_1(z) = d_{F_r}(z,g(C^{k-1}(x_n;u_1^n,\dots,u_{k-1}^n)    )  )$. By the Slicing Theorem, 
$$        \int_0^{D} \mass{\slice{V_n}{\pi}{t}} dt \leq \int_0^{+\infty} \mass{\slice{V_n}{\pi}{t}} dt \leq  \mass{V_n} \leq  \phi (n^{H_k}).        $$
Where $D=d_{F_r}\big(g(C^{k-1}(x_n;u_1^n,\dots,u_{k-1}^n)),g(C^{k-1}(x_n+u_{k}^n;u_1^n,\dots,u_{k-1}^n))\big)$.
$g$ is a $(\lambda,c')$ quasi-isometric embedding and $u_k^n$ is orthogonal to $\textup{span}\{u_1^n,\dots,u_{k-1}^n\}$, so 
$$    D \geq  \lambda d_{X}\big(C^{k-1}(x_n;u_1^n,\dots,u_{k-1}^n),C^{k-1}(x_n+u_{k}^n;u_1^n,\dots,u_{k-1}^n)\big) -c' \geq \lambda n^{1/k} -c'.                            $$
Up to taking bigger $N$, we can assume that for all $n\geq N$ and for all $i \in \{ 1 \dots k \}$,  
\\$\lambda n^{1/i} -c' \geq \frac{\lambda}{2}n^{1/i} $.
\\Therefore there exists $t_1 \in ]0,\frac{\lambda}{2}n^{1/k}[$ such that 
$$ \frac{\lambda}{4}n^{1/k} \, \mass{\slice{V_n}{\pi}{t_1}}    \leq  \mass{V_n}.        $$
Again, by adding a small current to $\slice{V_n}{\pi}{t_1}$ of mass less than $2 \lambda (k-1) n^{H_k - \frac{1}{k-1}} $, we get a filling of the cycle $g(P^{k-2}(x_n;u_1^n,\dots,u_{k-1}^n))$. So
$$     \textup{FillVol}_{k-1}^{F_r, \mathrm{cr}}(g(P^{k-2}(x_n;u_1^n,\dots,u_{k-1}^n)))  \leq  \frac{4}{\lambda}\frac{\phi(n^{1/k})}{n^{1/k}} + 2\lambda (k-1) n^{H_k - \frac{1}{k-1}}.   $$
Since $H_k - \frac{1}{k-1} < H_{k-1} $, the right hand side is $\ll n^{H_{k-1}} $. Let us denote by $\psi$ the sublinear function such that 
$$     \textup{FillVol}_{k-1}^{F_r, \mathrm{cr}}(g(P^{k-2}(x_n;u_1^n,\dots,u_{k-1}^n)))  \leq  \psi( n^{H_{k-1}} ).   $$
We can apply the same strategy again by considering the 1-Lipschitz map $d_2 : F_r \to \mathbb{R}$, $d_2(z) = d_{F_r}(z,g(C^{k-2}(x_n;u_1^n,\dots,u_{k-2}^n)    )  )$, and using the Slicing Theorem. By repeating this process $(k-1)$ times, we get a sublinear function $\varphi$ such that for all $n\geq N$
$$     \textup{FillVol}_{1}^{F_r, \mathrm{cr}}(g(P^{0}(x_n;u_1^n)))  \leq  \varphi( n ).   $$
If we denote $a_n = x_n$ and $b_n = x_n + u_1^n$, this implies that
$$ \frac{\lambda}{2}    n   \leq  d_{F_r}(g(a_n),g(b_n))  \leq  \varphi( n ).   $$
Which contradicts the fact that $g$ is a quasi-isometric embedding.
\end{proof}
Finally, let us give a proof of Corollary \ref{cor of Thm 4}.
\begin{proof}
By a result of Bonk--Schramm \cite{bonk2011embeddings}, there exists $n \in \mathbb{N}$ such that $Y$ quasi-isometrically embeds into $\mathbb{H}^n$. Let $g : Y \to \mathbb{H}^n $ be such an embedding with constants $(\lambda,c)$, and consider the coarse embedding $h = g \circ f : \mathbb{R}^p \to \mathbb{H}^n $. By Theorem \ref{CE of euclidean spaces}, for all $x\in \mathbb{R}^p$, and every $u \in \mathbb{R}^p$
$$  \frac{\textup{FillVol}_{1}^{\mathbb{H}^n, \mathrm{cr}}(h(P^{0}(x;u)))   }{\| u  \|}           \to 0, \textup{when $ \| u  \| $ tends to +$\infty$ .} $$
In other words, and by lemma \ref{filling current distance}, there exists a sublinear function $\phi$ such that for all $x,y \in \mathbb{R}^p$, $d_{\mathbb{H}^n}(h(x),h(y)) \leq \phi(d_{\mathbb{R}^p}(x,y))$. However, $g$ is a $(\lambda,c)$-quasi-isometric embedding so
$$  \frac{1}{\lambda}d_Y(f(x),f(y))-c  \leq d_{\mathbb{H}^n}(h(x),h(y)) .     $$
Therefore, for all $x,y \in \mathbb{R}^p$
$$     d_Y(f(x),f(y))     \leq \lambda \,\phi(d_{\mathbb{R}^p}(x,y)) + \lambda c,   $$
and $\psi(t) := \lambda \phi(t) + \lambda c$ is sublinear.
\end{proof}
\begin{rem}
In higher rank, Theorem \ref{Thm 4} does not necessarily imply that the coarse embedding is uniformly compressing, as shown by the following example:
\begin{align*}
\psi \; \colon \; \; \; \; \;  \mathbb{R}^3 &\longrightarrow \mathbb{H}^3 \times \mathbb{R} \\
(x,y,z)&\longmapsto \begin{alignedat}[t]{2} & \big( (x,y,1),z  \big)
\end{alignedat}
\end{align*}
where $\mathbb{H}^3$ is the upper half-space model. $\psi$ is the product of a horospherical embedding with the identity map on $\mathbb{R}$. It is clearly a coarse embedding, $\textup{rank}(\mathbb{H}^3 \times \mathbb{R}) = 2$, but it is not uniformly compressing since the $z$-axis is sent isometrically.
\end{rem}
\section{Further questions}

\begin{question} Does every cocompact geodesic metric space with exponential growth admits a coarse embedding of a binary tree?
\end{question}
This question was already asked by Shalom in \cite{shalom2004harmonic}. Note that if the answer is positive, then Theorem \ref{Thm 1 v1} is just a consequence of Theorem \ref{Thm 2 v1}. Indeed, Theorem \ref{Thm 2 v1} already treats the case when the domain $X$ is a product of regular trees, since the $(p+1)$-regular tree can be seen as the Bruhat-Tits building of $\SL_2(\mathbb{Q}_p)$. Moreover, all regular trees de degree $\geq 3$ are quasi-isometric. So if $X = X_1\times \dots \times X_k$ is a product of geodesic metric spaces of exponential growth, then $X$ contains an isometric copy of a model space of rank $k$ with no Euclidean factor. Therefore Theorem \ref{Thm 1 v1} follows from Theorem \ref{Thm 2 v1}. This question in full generality is still open even for groups.
In \cite{de2008quasi}, Cornulier and Tessera showed that
if a group G is either a connected Lie group, or a finitely generated solvable group with
exponential growth, then it contains a quasi-isometrically embedded free
sub-semigroup on 2 generators. Thus it contains quasi-isometrically embedded binary tree. However, it is not
known whether every locally compact compactly generated (or finitely generated) group with exponential growth
contains a quasi-isometrically (or even coarsely) embedded copy of a binary tree.

\begin{question}
Let $X$ be either $X_1 \times \dots \times X_k$ as in Theorem \ref{Thm 1 v1}, or a model space $S \times B$ of rank $k$ as in Theorem \ref{Thm 2 v1}, and let $Y = \mathbb{R}^n \times S' \times B'$ be a model space of rank $= k-1$. We saw that there is no coarse embedding $X \to Y$. Can adding a Euclidean factor in the target make it possible? i.e.\ can we embed $X$ into $Y \times \mathbb{R}^p$ for some $p\in \mathbb{N}$?
\end{question}
We answered this question negatively in Corollary \ref{cor 2} when $\textup{rank}(Y) < k-1$. But we do not know if the case $\textup{rank}(Y) = k-1$ is different. For example, there is no coarse embedding from the symmetric space $\SL_{3}(\mathbb{R})/\textup{SO}_{3}(\mathbb{R})$ of rank $2$ into a hyperbolic space $\mathbb{H}^n$. Can we have a coarse embedding $\SL_{3}(\mathbb{R})/\textup{SO}_{3}(\mathbb{R}) \to \mathbb{H}^n \times \mathbb{R}^p$ for some $p\in \mathbb{N}$?
\begin{question}
What can we say about the limit set $\overline{f(X)}\cap \partial Y$ of a coarse embedding $ f :X \to Y$ between model spaces, and in particular about coarse embeddings of $\mathbb{R}^{n}$ into $\mathbb{H}^{n+1}$? For instance, the limit sets of horospherical embeddings are reduced to a point. Is it the case for all coarse embeddings from $\mathbb{R}^{n}$ into $\mathbb{H}^{n+1}$? Is the image of such embedding always contained in a horoball? 
\end{question}
Note that Bowditch \cite{bowditch2017limit} showed that when $Y$ is Gromov-hyperbolic with bounded geometry and $X$ is a geodesic metric space that has ``fast growth'', in particular if it has exponential growth, then the limit set does not contain isolated points.
\\
\\Our results extends to uniform lattices in symmetric spaces, since they are quasi-isometric to them. A natural question to ask is the following.
\begin{question}What can be said about nonuniform lattices? For example, can we embed $\textup{SL}_n(\mathbb{Z})$ into a symmetric space of rank $<n-1$? \\Fisher--Whyte \cite{fisher2018quasi} proved that $\textup{SL}_2(\mathbb{R})\times \textup{SL}_2(\mathbb{R})$ quasi-isometrically embeds into $\textup{SL}_3(\mathbb{R})$. Can we have a coarse embedding from $\textup{SL}_2(\mathbb{Z})\times \textup{SL}_2(\mathbb{Z})$ into $\textup{SL}_3(\mathbb{Z})$?
Note that Leuzinger and Young managed recently to give higher filling functions of some nonuniform lattices \cite{leuzinger2021filling}.
\end{question}

\begin{question}Can we extend Theorem \ref{Thm 3 v1} and Theorem \ref{Thm 3 v2} to all proper cocompact CAT(0) spaces in the target? \end{question}
To do so, we need an analogue of Leuzinger's result, i.e.\ is the filling of a proper cocompact CAT(0) space linear above the rank? This is still an open question. Note that Goldhirsh-Lang \cite{goldhirsch2021characterizations} recently proved that it holds for cycles with controlled density. 
\hfill
\\
\bibliographystyle{alpha}
\bibliography{biblio.bib}
\begin{flushright}
\small{\textsc{
Max Planck Institute for Mathematics \& Université Paris Cité, IMJ-PRG}\\
\texttt{bensaid@mpim-bonn.mpg.de}}
\end{flushright}
\end{document}